\documentclass[a4paper,10pt,leqno]{article}
\usepackage[margin=2.5cm]{geometry}
\usepackage{mathtools}
\usepackage{stackrel}
\usepackage[all,2cell]{xy}
\UseAllTwocells
\usepackage{amsmath}
\usepackage{amssymb}
\usepackage{amsxtra}
\usepackage{amsthm}
\usepackage[pagebackref,breaklinks]{hyperref}
\usepackage{graphicx}
\usepackage[T1]{fontenc}
\usepackage{lmodern}
\usepackage{varioref}
\usepackage{makeidx}
\makeindex

\theoremstyle{definition}
\newtheorem{Definition}[subsection]{Definition}
\newtheorem{Construction}[subsection]{Construction}
\newtheorem{Convention}[subsection]{Convention}
\newtheorem{Notation}[subsection]{Notation}
\newtheorem{Example}[subsection]{Example}
\newtheorem{Remark}[subsection]{Remark}
\theoremstyle{plain}
\newtheorem{Lemma}[subsection]{Lemma}
\newtheorem{Proposition}[subsection]{Proposition}
\newtheorem{Theorem}[subsection]{Theorem}
\newtheorem{Corollary}[subsection]{Corollary}

\newcommand{\FFun}[1]{#1\mathrm{FoldFun}}
\newcommand{\dFFun}[1]{#1\mathbb{F}\mathrm{old}\mathbb{F}\mathrm{un}}
\newcommand{\Fun}{\mathrm{Fun}}
\newcommand{\bFun}{\mathbf{Fun}}
\newcommand{\GD}{\mathcal{GD}}
\newcommand{\PsFun}{\mathcal{P}s\mathcal{F}\mathnormal{un}}
\newcommand{\UPsFun}{\mathcal{UP}s\mathcal{F}\mathnormal{un}}
\newcommand{\PsNat}{\mathbf{PsNat}}
\newcommand{\rtwoFun}{2\mathrm{Fun}}
\newcommand{\btwoFun}{\mathbf{2Fun}}
\newcommand{\twoFun}{2\mathcal{F}\mathnormal{un}}
\newcommand{\twoFunps}{\twoFun^{\mathrm{ps}}}
\newcommand{\Comp}{\mathcal{C}\mathnormal{omp}}
\newcommand{\rComp}{\mathrm{Comp}}
\newcommand{\Cart}{\mathrm{Cart}}

\newcommand{\Set}{\mathbf{Set}}
\newcommand{\Cat}{\mathbf{Cat}}
\newcommand{\cCat}{\mathcal{C}at}
\newcommand{\twoCat}{\mathbf{2Cat}}
\newcommand{\twooneCat}{\mathbf{(2,1)Cat}}
\newcommand{\FoldCat}{\mathbf{FoldCat}}
\newcommand{\dComp}{\mathbb{C}\mathrm{omp}}
\newcommand{\red}{\mathrm{red}}
\newcommand{\op}{\mathrm{op}}
\newcommand{\coop}{\mathrm{coop}}
\newcommand{\twoop}{\mathrm{co}}
\newcommand{\rh}{\mathrm{h}}
\newcommand{\rv}{\mathrm{v}}
\newcommand{\hid}{\one^{\mathrm{h}}}
\newcommand{\vid}{\one^{\mathrm{v}}}

\newcommand{\adj}{\mathrm{adj}}
\newcommand{\pr}{\mathrm{pr}}
\newcommand{\Ob}{\mathrm{Ob}}
\newcommand{\Mor}{\mathrm{Mor}}
\newcommand{\Hor}{\mathrm{Hor}}
\newcommand{\Ver}{\mathrm{Ver}}
\newcommand{\rSq}{\mathrm{Sq}}
\newcommand{\Ar}{\mathcal{A}\mathnormal{r}}
\newcommand{\Tr}{\mathcal{T}\mathnormal{r}}
\newcommand{\Sq}{\mathcal{S}\mathnormal{q}}
\newcommand{\Inv}{\mathrm{Inv}}
\newcommand{\bzero}{\mathbf{0}}
\newcommand{\bA}{\mathbf{A}}
\newcommand{\bB}{\mathbf{B}}
\newcommand{\bC}{\mathbf{C}}
\newcommand{\bD}{\mathbf{D}}
\newcommand{\bI}{\mathbf{I}}
\newcommand{\bO}{\mathbf{O}}
\newcommand{\bT}{\mathbf{T}}
\newcommand{\cA}{\mathcal{A}}
\newcommand{\cB}{\mathcal{B}}
\newcommand{\cC}{\mathcal{C}}
\newcommand{\cD}{\mathcal{D}}
\newcommand{\cE}{\mathcal{E}}

\newcommand{\cG}{\mathcal{G}}
\newcommand{\cH}{\mathcal{H}}
\newcommand{\cL}{\mathcal{L}}

\newcommand{\cP}{\mathcal{P}}
\newcommand{\cR}{\mathcal{R}}
\newcommand{\cS}{\mathcal{S}}
\newcommand{\cT}{\mathcal{T}}

\newcommand{\cV}{\mathcal{V}}

\newcommand{\dC}{\mathbb{C}}
\newcommand{\dD}{\mathbb{D}}
\newcommand{\dE}{\mathbb{E}}
\newcommand{\dK}{\mathbb{K}}
\newcommand{\dQ}{\mathbb{Q}}
\newcommand{\one}{\mathrm{id}}
\newcommand{\id}{\mathrm{id}}
\newcommand{\res}{\mathbin{|}}

\newcommand\aprime{a$'$}
\newcommand\bprime{b$'$}
\newcommand\cprime{c$'$}

\labelformat{enumi}{(#1)}
\labelformat{enumii}{(#1)}

\begin{document}
\numberwithin{equation}{subsection}
\title{Gluing pseudo functors via $n$-fold categories}
\author{Weizhe Zheng\thanks{Morningside Center of Mathematics, Academy of Mathematics and Systems Science, Chinese
Academy of Sciences, Beijing 100190, China; email:
\texttt{wzheng@math.ac.cn}}}
\date{}
\maketitle

\footnotetext{Keywords: pseudo functor, gluing data, $2$-category, $n$-fold
category, hypercube}\footnotetext{Mathematics Subject Classification 2010:
18D05}

\begin{abstract}
Gluing of two pseudo functors has been studied by Deligne, Ayoub, and
others in the construction of extraordinary direct image functors in
\'etale cohomology, stable homotopy, and mixed motives of schemes. In this
article, we study more generally the gluing of finitely many pseudo
functors. Given pseudo functors $F_i\colon \cA_i\to \cD$ defined on
sub-$2$-categories $\cA_i$ of a $2$-category $\cC$, we are concerned with
the problem of finding pseudo functors $\cC\to \cD$ extending $F_i$ up to
pseudo natural equivalences. With the help of $n$-fold categories, we
organize gluing data for $n$ pseudo functors into $2$-categories. We
establish general criteria for equivalence between such $2$-categories for
$n$ pseudo functors and for $n-1$ pseudo functors, which can be applied
inductively to the gluing problem. Results of this article are used in
\cite{sixop} to construct extraordinary direct image functors in \'etale
cohomology of Deligne-Mumford stacks.
\end{abstract}

\paragraph{Note.} This article is accepted for publication in the \emph{J. Homotopy
Relat.\ Struct.} The final publication will be available at Springer via
\url{http://dx.doi.org/10.1007/s40062-016-0126-2}. \textbf{The numberings in
the published version differ from this preprint version.}

\section*{Introduction}
The extraordinary direct image functor $Rf_!$, one of Grothendieck's six
operations, between derived categories of \'etale sheaves, was constructed
in SGA 4 XVII \cite{Deligne}. For a morphism of schemes $f$ in a suitable
category of schemes $\bC$, $Rf_!$ is a functor between derived categories
and functoriality of the construction $f\mapsto Rf_!$ is encoded by a pseudo
functor $F\colon \bC\to \cD$, where $\cD=\cCat$ is the $2$-category of
categories. There are obvious candidates for the restrictions $F_{\bA}\colon
\bA\to \cD$ and $F_{\bB}\colon \bB\to \cD$ of $F$ to the subcategories $\bA$
and $\bB$ consisting of open immersions and proper morphisms, respectively.
The problem is thus to glue the pseudo functors $F_{\bA}$ and $F_{\bB}$,
namely to find a pseudo functor $F$ extending $F_{\bA}$ and $F_{\bB}$ up to
pseudo natural equivalences. Deligne developed a theory for gluing two
pseudo functors with the same target $2$-category $\cD$ and defined on
subcategories $\bA$ and $\bB$ of any category $\bC$
\cite[Section~3]{Deligne}\footnote{Deligne's original formulation focuses on
the case $\cD=\cCat$ and uses the equivalent language of cofibered
categories \cite[Section~8]{Grothendieck}, but the argument extends to any
target $2$-category $\cD$.}. Throughout this article, $2$-categories and
$2$-functors are assumed to be strict. Here $\bA$ and $\bB$ are assumed to
contain all objects of $\bC$ and $F_{\bA}$ and $F_{\bB}$ are assumed to be
identical on objects. Deligne assumes that every morphism $f$ of $\bC$ is
compactifiable in the sense that it can be decomposed (noncanonically) into
$f=pj$, where $j$ is in $\bA$ and $p$ is in $\bB$. One may then define
$F(f)$ to be $F_{\bB}(p)F_{\bA}(j)$, but the latter depends on the
noncanonical decomposition $f=pj$ and is not a priori functorial.  A
necessary condition for $F_{\bA}$ and $F_{\bB}$ to glue is that for every
commutative square $D$
\[\xymatrix{X\ar[r]^j\ar[d]_q &Y\ar[d]^p \\ Z\ar[r]^i &W}\]
where $i,j$ are in $\bA$ and $p,q$ are in $\bB$, we have an invertible
$2$-cell $G_D\colon F_{\bA}(i)F_{\bB}(q)\Rightarrow F_{\bA}(p)F_{\bB}(j)$,
compatible with compositions and units. The triple $(F_{\bA},F_{\bB},(G_D))$
is called a \emph{gluing datum}. Deligne shows that under mild assumptions
on $\bB$ and $\bC$, there exists an (essentially unique) pseudo functor $F$
associated to every given gluing datum. Deligne's theory has been applied to
other contexts, such as the construction of $Rf_!$ between triangulated
categories of mixed motives \cite[Proposition 2.2.7]{CisinskiD}. Ayoub
developed a variant of Deligne's theory \cite[Th\'eor\`eme~1.3.1]{Ayoub} in
order to construct $Rf_!$ relative to stable homotopical pseudo functors.

In this article, we study more generally the problem of gluing finitely many
pseudo functors with the same target. The case of three pseudo functors is
necessary for the construction of $Rf_!$ and base change morphisms in
\cite{sixop} for a morphism $f$ of Deligne-Mumford stacks, because
compactification of coarse spaces only allows us to decompose a morphism
into a sequence of three morphisms.

Let $\cC$ be a $(2,1)$-category, namely a $2$-category whose $2$-cells are
all invertible. Let $\cD$ be a $2$-category. Let $\cA_1,\dots,\cA_n$ be
locally full sub-$2$-categories of $\cC$ each of which contains all objects
of $\cC$. A \emph{gluing datum} is a pair $((F_i),(G_{ij}))$, where each
$F_i\colon \cA_i\to \cD$ is a pseudo functor, and each $G_{ij}=(G_{ijD})$ is
a collection of invertible $2$-cells such that $(F_i,F_j,G_{ij})$ is a
gluing data for two pseudo functors as above, subject to one extra condition
for cubes. We organize gluing data into a $2$-category
$\GD_{\cA_1,\dots,\cA_n}(\cC,\cD)$ (see Remark \ref{r.gd} for an explicit
description in the case $n=2$ and Remarks \ref{r.gdn}, \ref{r.gdna} for
explicit descriptions in the general case). Our first main result is for the
gluing of two pseudo functors.

\begin{Theorem}\label{p.Del}
  Let $\cC$ be a $(2,1)$-category and let $\cA$ and $\cB$ be locally full
  sub-$2$-categories of $\cC$, each containing all objects of $\cC$.
  Assume the following:
  \begin{enumerate}
    \item \label{p.Del1} For every morphism $f\colon X\to Y$ of $\cC$,
        there exist a morphism $j\colon X\to Z$ of $\cA$, a morphism
        $p\colon Z\to Y$ of $\cB$, and a $2$-cell $\alpha\colon pj
        \Rightarrow f$ of $\cC$.
    \item \label{p.Del2} Every diagram $X\to Y \leftarrow Z$ in $\cB$ can
        be completed into a commutative square with $2$-cell in $\cB$,
        Cartesian in $\cC$.
   \end{enumerate}
  Then the descent $2$-functor (as described in Remark \ref{r.QD})
  \[Q^2_\cD\colon \PsFun(\cC,\cD)\to \GD_{\cA,\cB}(\cC,\cD)\]
  is a $2$-equivalence for every $2$-category $\cD$.
\end{Theorem}

This is a common generalization of the results of Deligne and Ayoub, as our
assumptions are less restrictive. Our result is also more precise in the
sense that we establish a $2$-equivalence of $2$-categories, whereas
previous results only dealt with \emph{objects} of
$\GD_{\cA_1,\cA_2}(\cC,\cD)$. This precision is useful for the construction
of pseudo natural transformations. We refer the reader to Remark \ref{r.DA}
for more details on this comparison of results.

Our main result for gluing more than two pseudo functors is the following.

\begin{Theorem}\label{t.main}
Let $\cC$ and $\cA_1,\dots, \cA_n$ be as above with $n\ge 3$. Let $\cB$ be a
locally full sub-$2$-category of $\cC$ containing $\cA_1$ and $\cA_2$. Under
suitable assumptions on $\cA_1,\dots,\cA_n$, $\cB$, $\cC$, for every
$2$-category $\cD$, the canonical $2$-functor (as described in Remark
\ref{r.Qphi})
\[Q_\cD\colon \GD_{\cB,\cA_3,\dots,\cA_n}(\cC,\cD) \to
\GD_{\cA_1,\dots,\cA_n}(\cC,\cD)
\]
is a $2$-equivalence.
\end{Theorem}

We refer the reader to Theorem \ref{c.glue'} for a more precise statement.

Theorem \ref{t.main} reduces the gluing of $n$ pseudo functors to the gluing
of $n-1$ pseudo functors and can be applied recursively. Combining with
Theorem \ref{p.Del}, we obtain sufficient conditions for the descent
$2$-functor $Q^n_\cD\colon \PsFun(\cC,\cD)\to
\GD_{\cA_1,\dots,\cA_n}(\cC,\cD)$ to be a $2$-equivalence. We illustrate
this in the case $n=3$.

\begin{Corollary}\label{c.main}
Let $\cC$ be a $(2,1)$-category. Let $\cA_1,\cA_2,\cA_3$ be locally full
sub-$2$-categories of $\cC$ each of which contains all objects of $\cC$.
Assume that $\cC$ admits pseudo fiber products, and there exists a locally
full sub-$2$-category $\cB$ of $\cC$, containing $\cA_1,\cA_2$, and
satisfying the following conditions:
\begin{enumerate}
\item For every morphism $f$ of $\cC$, there exist a morphism $j$ of
    $\cB$, a morphism $p$ of $\cA_3$, and a $2$-cell $pj\Rightarrow f$ of
    $\cC$.
\item For every morphism $f$ of $\cB$, there exist a morphism $j$ of
    $\cA_1$, a morphism $p$ of $\cA_2$, and a $2$-cell $pj\Rightarrow f$
    of $\cC$.
\item For every morphism $f$ of $\cB\cap \cA_3$, there exist a morphism
    $j$ of $\cA_1\cap \cA_3$, a morphism $p$ of $\cA_2\cap \cA_3$, and a
    $2$-cell $pj\Rightarrow f$ of $\cC$.
\item The sub-$2$-categories $\cA_1,\cA_2,\cA_3,\cB$ are stable under base
    change in $\cC$. The sub-$2$-categories $\cA_3$ and $\cB$ are stable
    under taking diagonals in $\cC$.
\end{enumerate}
Then the descent $2$-functor
\[Q_\cD^3\colon\PsFun(\cC,\cD)\to
\GD_{\cA_1,\cA_2,\cA_3}(\cC,\cD)
\]
is a $2$-equivalence.
\end{Corollary}

\begin{proof}
Indeed, $Q^3_\cD$ can be decomposed as
\[\PsFun(\cC,\cD)\xrightarrow{Q^2_\cD} \GD_{\cB,\cA_3}(\cC,\cD)\xrightarrow{Q_\cD}
\GD_{\cA_1,\cA_2,\cA_3}(\cC,\cD),
\]
where both are $2$-equivalences by Theorems \ref{p.Del} and \ref{c.glue'}.
\end{proof}

There are other ways to combine the above theorems. In \cite[Proposition
1.5]{sixop}, we deduce from the above theorems and Proposition \ref{p.gluet}
a case where $Q^2_\cD$ is a $2$-equivalence but for which condition
\ref{p.Del1} of Theorem \ref{p.Del} is not satisfied.

Note that when $Q^n_\cD$ is a $2$-equivalence, given any gluing datum
$((F_i),(G_{ij}))$, there exists a pseudo functor $F$ extending $(F_i)$ up
to pseudo natural isomorphisms.

We study the gluing of pseudo functors in the framework of (strict) $n$-fold
categories. An $n$-fold category structure is an extended categorical
structure consisting of hypercubes up to dimension $n$, endowed with
composition laws in each of the $n$ directions. This often encodes more
information than a higher category structure of the same dimension, such as
an $n$-category. There is, however, a rich interplay between extended
categories and higher categories, of which the most relevant part to our
study is the relation between $n$-fold categories and $2$-categories. Given
a $2$-category $\cC$, we construct an $n$-fold category $\dQ^n\cC$ of
hypercubes in $\cC$, extending the double category of up-squares for $n=2$
due to Bastiani and Ehresmann \cite[2.C.1, p.~272]{BEhr}. We also consider
the $n$-fold subcategory $\dQ_{\cA_1,\dots,\cA_n}\cC$ spanned by hypercubes
whose edges in direction $i$ are in $\cA_i$. The construction $\dQ^n$ admits
a left adjoint, carrying an $n$-fold category $\dC$ to its reduced
$2$-category of paths $\cT^\red_n\dC$, related to Gray's tensor product of
$2$-categories. We further construct a $(2,1)$-category $\cL\cT_n\dC$,
variant of $\cT^\red_n\dC$. The $2$-category of gluing data
$\GD_{\cA_1,\dots,\cA_n}(\cC,\cD)$ can be identified with the $2$-category
$\twoFunps(\cE,\cD)$ of $2$-functors
$\cE=\cL\cT_n\dQ_{\cA_1,\dots,\cA_n}\cC\to \cD$, pseudo natural
transformations, and modifications. Properties of the $2$-category of gluing
data are established by studying the $(2,1)$-category $\cE$.

The article is organized as follows. In Section~\ref{s.1}, we fix some
conventions and prove some preliminary results on $2$-categories. We
introduce the $2$-category of paths $\cT\cC$ in a $2$-category $\cC$, which
allows one to straighten pseudo functors to $2$-functors. In
Section~\ref{s.2}, after recalling the definition of an $n$-fold category,
we investigate the relation between $2$-categories and $n$-fold categories.
We construct the $n$-fold category $\dQ^n\cC$ of hypercubes and the reduced
$2$-categories of paths $\cT_n^\red\dC$, and we establish the aforementioned
adjunction. A variant of the adjunction is used to define the descent
$2$-functor $Q^n\colon \PsFun(\cC,\cD)\to \GD_{\cA_1,\dots,\cA_n}(\cC,\cD)$.
In Section~\ref{s.3}, we study functorial properties of these constructions
with respect to the index set $\{1,\dots,n\}$, and construct $2$-functors
between various $2$-categories of gluing data. In Sections \ref{s.two} and
\ref{s.finiteglue}, we apply these constructions to study the gluing of
pseudo functors. In Section \ref{s.two}, we study the case of two pseudo
functors and prove Theorem \ref{p.Del}, extending results of Deligne and
Ayoub. We deduce the theorem from a general criterion involving the
$2$-category of compactifications. In Section \ref{s.finiteglue}, we study
the case of finitely many pseudo functors and prove Theorem \ref{t.main}. In
Sections \ref{s.6} through~\ref{s.adjoint}, we develop several tools for the
application of the main theorems. In Sections~\ref{s.6} and \ref{s.7}, we
introduce Cartesian gluing data, first for two pseudo functors, and then for
finitely many pseudo functors. Cartesian gluing data are an alternative set
of gluing data that only makes use of Cartesian squares instead of
commutative squares with $2$-cells, and are easier to construct in
applications. We show that the $2$-category of Cartesian gluing data is
isomorphic to the $2$-category of gluing data under mild conditions. In
Section~\ref{s.adjoint}, we check the axioms for gluing data in the case
when the data are constructed from base change maps via adjunctions.
Finally, we include the proof of a preliminary result in Section
\ref{s.proof} for completeness.

For the convenience of the reader, an index of notation is provided at the
very end of the article.

In joint work with Yifeng Liu \cite{LZ1}, we establish analogues of some
results of this article in the $\infty$-categorical setting, which are used
in \cite{LZ2} to construct Grothendieck's six operations on Artin stacks. We
remark that specific features of $2$-categories have been exploited in this
article and the full generality of our results cannot be deduced from
\cite{LZ1}.

\subsection*{Acknowledgments}
In the course of this work the author has benefited very much from
conversions with Yifeng Liu. The author also thanks Joseph Ayoub, Johan de
Jong, Ofer Gabber, and Luc Illusie for useful discussions. The author is
grateful to the referees for the numerous comments they made on previous
versions of this article. This work was partially supported by China's
Recruitment Program of Global Experts; National Natural Science Foundation
of China Grant 11321101; Hua Loo-Keng Key Laboratory of Mathematics, Chinese
Academy of Sciences; National Center for Mathematics and Interdisciplinary
Sciences, Chinese Academy of Sciences.

\section{Preliminaries on $2$-categories}\label{s.1}
In this section, we fix some conventions and notation on $2$-categories and
record some preliminary results. In particular, we introduce the
$2$-category of paths $\cT\cC$ in a $2$-category $\cC$ (Definition
\ref{d.T0}), which allows one to straighten pseudo functors to $2$-functors.

Throughout the article, we reserve the symbol $\cong$ for isomorphisms. The
symbol will not be used for equivalences or bi-equivalences, which are
instead stated verbally. We denote categories by bold letters ($\bC$, $\bD$,
etc.), $2$-categories by script letters ($\cC$, $\cD$, etc.), $n$-fold
categories by blackboard bold letters ($\dC$, $\dD$, etc.).

The notion of $2$-category was introduced by Ehresmann \cite[note
bibliographique, p.~324]{Ehr65} and B\'enabou \cite[p.~3824]{Benabou}. As in
\cite[Chapter~7]{Borceux}, $2$-categories, $2$-functors, and $2$-natural
transformations are assumed to be strict. Pseudo functor preserves
composition and unit up to coherent invertible $2$-cells. A
\emph{$2$-equivalence} is a $2$-functor $F\colon \cC\to \cD$ such that there
exist a $2$-functor $G\colon \cD\to \cC$ and $2$-natural \emph{isomorphisms}
$\one_\cC\cong GF$ and $FG\cong \one_\cD$. In this case we say that $F$ and
$G$ are $2$-quasi-inverses of each other. A \emph{bi-equivalence} is a
pseudo functor $F\colon \cC\to \cD$ such that there exist a pseudo functor
$G\colon \cD\to \cC$ and pseudo natural equivalences $\one_\cC\Rightarrow
GF$ and $FG\Rightarrow \one_\cD$ (some authors refer to such pseudo functor
as ``$2$-equivalence'' \cite[Corollary 1.5.26 (i)]{GabberR}). In this case
we say that $F$ and $G$ are pseudo inverses of each other. We use the term
\emph{morphism} for $1$-cell of a $2$-category.

Let us recall the notion of pseudo fiber product in a $2$-category $\cC$,
which is a type of ``$2$-limit'' in the terminology of \cite[Definition
1.6.1 (i)]{GabberR} or ``pseudo-bilimit'' in the terminology of
\cite[Chaper~7]{Borceux}. Let $Z\xrightarrow{i}W\xleftarrow{p} Y$ be
morphisms in $\cC$. For any object $X$ of $\cC$, consider the category
$\cC(X,Z)\times_{\cC(X,W)}\cC(X,Y)$ of triples $(q,j,\alpha)$ as shown by
the square
\[
  \xymatrix{X\ar[r]^j\ar[d]_q\drtwocell\omit{^\alpha} & Y \ar[d]^p\\
  Z\ar[r]^i & W,}
\]
where $\alpha$ is an invertible $2$-cell. A \emph{pseudo fiber product} of
$Z\xrightarrow{i}W\xleftarrow{p} Y$ is an object $X$ of $\cC$ equipped with
an object $(q,j,\alpha)$ of $\cC(X,Z)\times_{\cC(X,W)}\cC(X,Y)$, such that
for every object $X'$ of $\cC$, the functor $\cC(X',X)\to
\cC(X',Z)\times_{\cC(X',W)}\cC(X',Y)$ induced by composition with
$(q,j,\alpha)$ is an equivalence of categories. In this case the square is
called \emph{Cartesian}. Given a Cartesian square the morphism $q$ is called
a \emph{base change} of $p$ by $i$. By the diagonal of a morphism $f\colon
X\to Y$ in $\cC$, we mean the morphism $X\to X\times_Y X$, unique up to
equivalence, where $X\times_Y X$ is the pseudo fiber product. We will also
consider \emph{strict} fiber products of categories and $2$-categories.

\begin{Convention}\label{s.small}\index{cat@$\Cat$}\index{cat2@$\twoCat$}
Unless otherwise stated, all categories and $2$-categories are assumed to be
small (and strict). Big categories and big $2$-categories are occasionally
used as a linguistic tool to simplify the narrative. We let $\Cat$ denote
the big category of categories and functors and $\twoCat$ denote the big
category of $2$-categories and $2$-functors.
\end{Convention}

Definition \ref{d.twoone} and Notation \ref{n.op} below are standard.

\begin{Definition}\label{d.twoone}
A \emph{$(2,1)$-category} is a $2$-category whose $2$-cells are
    invertible.
\end{Definition}

\begin{Notation}\label{n.op}\index{-co@$(-)^\twoop$, $(-)^\coop$!$\cC^\twoop$, $\cC^\coop$}
Let $\cC$ be a $2$-category. We denote by $\cC^{\coop}$ (resp.\
$\cC^{\twoop}$) the $2$-category obtained
  from $\cC$ by reversing the morphisms and $2$-cells (resp.\ $2$-cells
  only). In other words $\Ob(\cC^{\coop})=\Ob(\cC^{\twoop})=\Ob(\cC)$,
  and, for any pair of objects $X$ and $Y$ of $\cC$, we have
  $\cC^{\coop}(Y,X)=\cC^{\twoop}(X,Y)=\cC(X,Y)^{\op}$.

If $\cC$ is a $(2,1)$-category, inversion of the $2$-cells defines an
isomorphism $\cC\cong\cC^{\twoop}$.
\end{Notation}

\begin{Notation}[Functor categories]\label{n.U}\index{Fun@$\Fun(\bC,\bD)$,
$\bFun(\bC,\bD)$}\index{2Fun@$\rtwoFun(\cC,\cD)$, $\btwoFun(\cC,\cD)$,
$\twoFun(\cC,\cD)$}\index{2Funps@$\twoFunps(\cC,\cD)$}\index{UPsFun@$\UPsFun(\cC,\cD)$}\index{PsFun@$\PsFun(\cC,\cD)$}
Let $\bC$ and $\bD$ be categories. We let $\Fun(\bC,\bD)$ denote the set of
functors $\bC\to \bD$. We let $\bFun(\bC,\bD)$ denote the category of
functors $\bC\to \bD$ and natural transformations.

Let $\cC$ and $\cD$ be $2$-categories. We denote by $\rtwoFun(\cC,\cD)$ the
set of $2$-functors $\cC\to \cD$. We denote by $\btwoFun(\cC,\cD)$ the
category of $2$-functors $\cC\to \cD$ and $2$-natural transformations. We
denote by $\twoFun(\cC,\cD)$ the $2$-category of $2$-functors $\cC\to \cD$,
$2$-natural transformations, and modifications. We denote by
$\twoFunps(\cC,\cD)$ the $2$-category of $2$-functors $\cC\to \cD$,
\emph{pseudo} natural transformations, and modifications. We denote by
$\UPsFun(\cC,\cD)$ the $2$-category of strictly unital pseudo functors
$\cC\to \cD$, pseudo natural transformations, and modifications. We denote
by $\PsFun(\cC,\cD)$ the $2$-category of pseudo functors $\cC \to \cD$,
pseudo natural transformations, and modifications.
\end{Notation}

For $2$-categories $\cC$, $\cD$, and $\cE$, we have
\[\twoFun(\cC,\twoFun(\cD,\cE))\cong \twoFun(\cC\times \cD,\cE).\]

For any pseudo functor $G\colon \cD\to \cE$, composition with $G$ provides a
pseudo functor $G\circ -\colon \PsFun(\cC,\cD)\to \PsFun(\cC,\cE)$, which is
a $2$-functor if $G$ is a $2$-functor. By contrast, for any pseudo functor
$F\colon \cC\to \cD$, composition with $F$ provides a $2$-functor
\[-\circ F\colon \PsFun(\cD,\cE)\to \PsFun(\cC,\cE).\]
Let $\epsilon\colon F\Rightarrow F'$ be a pseudo natural transformation.
Then $\epsilon$ induces only a \emph{pseudo} natural transformation $-\circ
F\Rightarrow -\circ F'$ in general. However, if $\epsilon(X)$ is an identity
for every object $X$ of $\cC$, then $\epsilon$ induces a $2$-natural
transformation between the $2$-functors $\UPsFun(\cD,\cE)\to
\PsFun(\cC,\cE)$ induced by $F$ and $F'$.

We will need to work over a base $2$-category as follows.

\begin{Definition}[$\cD$-$2$-Category]\label{d.equiv}
Let $\cD$ be a $2$-category. A \emph{$\cD$-$2$-category} is a pair $(\cC,F)$
consisting of a $2$-category $\cC$ and a $2$-functor $F\colon \cC\to \cD$.
If $(\cB,E)$ and $(\cC,F)$ are $\cD$-categories, we define the category of
$\cD$-$2$-functors $\btwoFun_{\cD}((\cB,E),(\cC,F))$ to be the strict fiber
at $E$ of the functor $\btwoFun(\cB,\cC)\to \btwoFun(\cB,\cD)$ induced by
$F$. Objects and morphisms of this category are called
\emph{$\cD$-$2$-functors} and \emph{$\cD$-$2$-natural transformations},
respectively. Thus a $\cD$-$2$-functor $(\cB,E)\to(\cC,F)$ is a $2$-functor
$G\colon \cB\to \cC$  such that $E=FG$. If $G,H\colon (\cB,E)\to (\cC,F)$
are $\cD$-$2$-functors, a $\cD$-$2$-natural transformation is a $2$-natural
transformation $\alpha\colon G\Rightarrow H$ such that $F*\alpha\colon
FG\Rightarrow FH$ is $\one_E$. We say that a $\cD$-$2$-functor $G\colon
(\cB,E)\to (\cC,F)$ is a \emph{$\cD$-$2$-equivalence} if there exist a
$\cD$-$2$-functor $H\colon (\cC,F)\to (\cB,E)$ and $\cD$-$2$-natural
isomorphisms $\one_\cC\Rightarrow GH$ and $HG\Rightarrow \one_\cB$. In this
case we say that $G$ and $H$ are \emph{$\cD$-$2$-quasi-inverses} of each
other.
\end{Definition}

We do not consider $\cD$-modifications because we will only be interested in
$\cD$-$2$-categories whose strict fibers are categories. A
$\cD$-$2$-equivalence $\cB\to \cC$ induces a $2$-equivalence between the
strict fiber $2$-categories $\cB_X\to \cC_X$ for every object $X$ of $\cD$.

Let us introduce some terminology on faithfulness.

\begin{Definition}[Faithfulness, fullness]\label{s.faith}
Let $\cC$ and $\cD$ be $2$-categories and let $F\colon \cC\to
    \cD$ be a pseudo functor. Given a property $(P)$ on functors, We say that $F$ is
    \emph{locally $(P)$} if for
    every pair of objects $X$ and $Y$ of $\cC$, the functor
  \[  F_{XY}\colon\cC(X,Y)\to \cD(FX,FY) \]
is $(P)$. In particular, we say that $F$ is a \emph{local equivalence} if
$F_{XY}$ is an equivalence of categories for all $X$ and $Y$ (this property
is called ``fully faithful'' in \cite[Definition 1.5.11 (ii)]{GabberR}).

We say that $F$ is \emph{$2$-faithful} if $F$ is locally fully faithful and
the underlying functor of $F$ is faithful in the $1$-categorical sense. We
say that $F$ is \emph{$2$-fully faithful} if $F$ is locally fully faithful
and the underlying functor of $F$ is fully faithful in the $1$-categorical
sense.  In other words, $F$ is $2$-fully faithful if and only if $F_{XY}$ is
an isomorphism of categories for all $X$ and $Y$ (this property is called
``strongly faithful'' in \cite[Definition 1.5.11 (ii)]{GabberR}).

We say that $F$ is \emph{pseudo surjective} if for every object $Y$ of
$\cD$, there exists an object $X$ of $\cC$ and an equivalence $FX\to Y$ in
$\cD$.

We say that a sub-$2$-category $\cC$ of $\cD$ is \emph{locally full} if for
all morphisms $f,g$ in $\cC$, the collection of $2$-cells between $f$ and
$g$ in $\cC$ is the same as in $\cD$. We say that a sub-$2$-category $\cC$
of $\cD$ is $2$-full if for all objects $X$ and $Y$ in $\cC$, the category
$\cC(X,Y)$ equals the category $\cD(X,Y)$, or equivalently, if the inclusion
$2$-functor $\cC\to \cD$ is $2$-fully faithful.
\end{Definition}

\begin{Example}
The sub-$2$-categories $\twoFunps(\cC,\cD)\subseteq
\UPsFun(\cC,\cD)\subseteq \PsFun(\cC,\cD)$ are $2$-full.
\end{Example}

Recall that a pseudo functor $F$ is a bi-equivalence if and only if $F$ is a
pseudo surjective local equivalence \cite[Definition 1.5.11 (iii), Corollary
1.5.26 (i)]{GabberR}). A $2$-functor $F$ is a $2$-equivalence if and only if
$F$ is $2$-fully faithful and essentially surjective in the $1$-categorical
sense. This extends to $\cD$-$2$-functors as follows.

\begin{Lemma}\label{l.faith}
A $\cD$-$2$-functor $G\colon \cB\to \cC$ is a $\cD$-$2$-equivalence if and
only if $G$ is $2$-fully faithful and for every object $Y$ of $\cC$, there
exists an object $X$ of $\cB$ and an isomorphism $\epsilon(Y)\colon GX\to Y$
in $\cC$ whose image in $\cD$ is an identity.
\end{Lemma}

\begin{proof}
The ``only if'' part is clear. For the ``if'' part, we construct a
$\cD$-$2$-quasi-inverse $H\colon \cC\to \cB$ carrying $Y$ to $X$ as follows.
For objects $Y$ and $Y'$ of $\cC$, we take $H_{Y,Y'}\colon \cC(Y,Y')\to
\cC(HY, HY')$ to be the unique functor such that $G_{HY,HY'}F_{Y,Y'}$ is
given by the formula $f\mapsto \epsilon({Y'})^{-1}f\epsilon(Y)$. Then
$\epsilon\colon GH\Rightarrow \id_\cC$ is a $\cD$-$2$-natural isomorphism.
For each object $X$ of~$\cB$, we let $\eta(X)\colon HGX\to X$ be the unique
morphism such that $G(\eta(X))=\epsilon({GX})$. Then $\eta\colon
HG\Rightarrow \id_\cB$ is a $\cD$-$2$-natural isomorphism.
\end{proof}

\begin{Notation}\label{d.bar}\leavevmode
\begin{enumerate}
\item Let $S$ be a set and let $\cD$ be a $2$-category. We view $S$ as a
    discrete $2$-category and denote by $\cD^S$ the $2$-category of
    $2$-functors $S\to \cD$. An object of $\cD^S$ is a map $S\to
    \Ob(\cD)$. A morphism $\alpha\colon F\to F'$ of $\cD^S$ is a family
  \[\left(\alpha(X)\colon FX \to F'X\right)_{X\in S}\]
  of morphisms of $\cD$. A $2$-cell $\Xi\colon \alpha \Rightarrow \beta$
  of $\cD^S$ is a family
  \[\left(\Xi(X)\colon \alpha(X) \Rightarrow \beta(X)\right)_{X\in S}\]
  of $2$-cells of $\cD$.

\item \label{d.bar2}\index{-@$\lvert-\rvert$} Let $\cC$ and $\cD$ be
    $2$-categories. We view $\PsFun(\cC,\cD)$ as a
    $\cD^{\Ob(\cC)}$-$2$-category via the forgetful $2$-functor
  \[\lvert -\rvert\colon \PsFun(\cC,\cD)\to \cD^{\Ob(\cC)}\]
  induced by the inclusion $\Ob(\cC)\to \cC$.
  \end{enumerate}
\end{Notation}

Note that the forgetful $2$-functor $\lvert -\rvert$ is locally faithful,
and the strict fibers are categories.

The following simple technique for modifying a pseudo functor will be of
use.

\begin{Lemma}\label{l.ps}
Let $\cC$ and $\cD$ be $2$-categories, let $F\colon \cC\to \cD$ be a pseudo
functor, let $H$ be an object of $\cD^{\Ob(\cC)}$, and let $\eta\colon
\lvert F\rvert \to H$ be an equivalence in $\cD^{\Ob(\cC)}$. Assume that to
every morphism $f\colon X\to Y$ of $\cC$ is associated a square in $\cD$
\[\xymatrix{FX\ar[r]^{Ff}\ar[d]_{\eta(X)}\drtwocell\omit{^{\alpha_f}} & FY\ar[d]^{\eta(Y)}\\
HX\ar[r]^{g_f}& HY}
\]
where $\alpha_f$ is an invertible $2$-cell. Then there exists a unique pair
$(G,\epsilon)$, where $G\colon \cC\to \cD$ is a pseudo functor and
$\epsilon\colon F\to G$ is a pseudo natural equivalence, such that $\lvert
G\rvert = H$, $\lvert \epsilon\rvert =\eta$, $G(f)=g_f$ and
$\epsilon(f)=\alpha_f$ for every morphism $f$ of $\cC$.
\end{Lemma}

\begin{proof}
Indeed, it remains to define the constraints of $G$, which are uniquely
determined by the constraints of $F$.
\end{proof}

Applying the lemma to the unital constraints of any pseudo functor $F\colon
\cC\to \cD$, we obtain a strictly unital pseudo functor $G\colon \cC\to \cD$
and a pseudo natural transformation $\epsilon\colon F\Rightarrow G$ with
$\epsilon(X)=\id_X$ for every object $X$ of $\cC$. By Lemma \ref{l.faith},
we obtain the following.

\begin{Proposition}\label{p.UPs}
  Let $\cC$ and $\cD$ be $2$-categories. Then the inclusion $\UPsFun(\cC,\cD)\subseteq
  \PsFun(\cC,\cD)$ is a $\cD^{\Ob(\cC)}$-$2$-equivalence.
\end{Proposition}

The proof of the following proposition is straightforward. For completeness
we give the proof in Section \ref{s.proof}.

\begin{Proposition}\label{p.equiv}
Let $F\colon \cC\to \cD$ be a pseudo functor such that $\Ob(F)\colon
\Ob(\cC)\to \Ob(\cD)$ is a bijection. For every $2$-category $\cE$, we
consider the $\cE^{\Ob(\cC)}$-$2$-functor
  \[\Phi_\cE\colon \PsFun(\cD,\cE)\to \PsFun(\cC,\cE)\]
  induced by $F$.

  \begin{enumerate}
  \item \label{p.equiv1} If for every morphism $g$ of $\cD$, the identity
      $2$-cell $\id_g$ can be decomposed as $g\Rightarrow Ff\Rightarrow
      g$, then $\Phi_\cE$ is $2$-faithful for every $2$-category $\cE$.

  \item \label{p.equiv2} The following conditions are equivalent:
\begin{enumerate}
\item \label{p.equiva} The pseudo functor $F$ is a bi-equivalence.

\item \label{p.equivb} There exists a pseudo functor $G\colon \cD\to
    \cC$ such that $\Ob(G)$ is the inverse of $\Ob(F)$ and there are
    pseudo natural isomorphisms $\eta\colon \one_\cC\Rightarrow GF$ and
    $\epsilon\colon FG\Rightarrow \one_\cD$ such that $\eta(X)=\one_X$
    and $\epsilon(Y)=\one_Y$ for all objects $X$ of $\cC$ and $Y$ of
    $\cD$.

\item \label{p.equivc} The $2$-functor $\Phi_\cE$ is an
    $\cE^{\Ob(\cC)}$-$2$-equivalence for every $2$-category $\cE$.

\item \label{p.equivd} The $2$-functor $\Phi_\cC$ is pseudo surjective
    and the $2$-functor $\Phi_\cD$ is a local equivalence.
\end{enumerate}
  \end{enumerate}
\end{Proposition}

Note that $\Phi_\cE$ is locally faithful. If $F$ is locally essentially
surjective, then the condition in \ref{p.equiv1} is satisfied. Moreover, the
condition in \ref{p.equiv1} is equivalent to the local essential
surjectivity of $F$ if $\cD$ is a $(2,1)$-category.

Pseudo functors can be straightened to $2$-functors via the $2$-category of
paths as follows.

\begin{Definition}[$2$-Category of paths $\cT\cC$ in a $2$-category $\cC$]\label{d.T0}\index{T@$\cT$}
  Let $\cC$ be a $2$-category. We define the \emph{$2$-category $\cT
  \cC$ of paths in $\cC$} as follows. The objects of $\cT\cC$ are the objects of $\cC$.
  A \emph{morphism} $X\to Y$ of $\cT \cC$ is a path
  \[X=X_0\xrightarrow{f_1} X_1\to \dots \to X_{m-1} \xrightarrow{f_m} X_m=Y, \quad m\ge 0\]
  where each $f_i$ is a morphism of $\cC$ for $1\le i\le m$. The identity morphism $\id_X\colon X\to
  X$ in $\cT\cC$ is the path of length~$0$. Composition of morphisms is given by concatenation of paths and is denoted by $*$.
  An \emph{atomic $2$-cell} between two paths sharing a source and a target is one of the
  following:
  \begin{enumerate}
    \item (creation of
        unit)\index{iota@$\iota_{g,f}$}\index{tzheta@$\theta_{g,f}$}
        $\iota_{g,f}\colon g* f\Rightarrow g* \one^\cC_Y *f$ or (deletion
        of unit) $\theta_{g,f}\colon g* \one^\cC_Y
        * f\Rightarrow g* f$, where $X\xrightarrow{f} Y\xrightarrow{g} Z$
        is a sequence of paths, and $\one^{\cC}_Y$ is the identity
        morphism on $Y$ in $\cC$ viewed as a path of length $1$.

    \item (composition of
        morphisms)\index{gzamma@$\gamma_{g,h',h,f}$}\index{delta@$\delta_{g,h',h,f}$}
        $\gamma_{g,h',h,f}\colon g* h'* h * f\Rightarrow g* (h'h)
        * f$ or (decomposition of morphism) $\delta_{g,h',h,f}\colon g* (h'h)
        * f\Rightarrow g* h'* h * f$, where $X\xrightarrow{f} Y\xrightarrow{h} Y'
        \xrightarrow{h'} Y''\xrightarrow{g} Z$ is a sequence of paths,
        with $h$ and $h'$ of length $1$ and $h'h$ being their composition
        in $\cC$.

    \item ($2$-cell of $\cC$)\index{sigma@$\sigma_{g,\alpha,f}$}
        $\sigma_{g,\alpha,f}\colon g*h*f\Rightarrow g*h'*f$, where
        $f\colon X\to Y$ and $g\colon Y\to Z$ are paths, and $h,h'\colon
        Y\to Z$ are morphisms of $\cC$, and $\alpha\colon h\Rightarrow h'$
        is a $2$-cell of $\cC$.
  \end{enumerate}
  A \emph{pre-$2$-cell} $f\Rightarrow g$ is a sequence of atomic
  $2$-cells $f=f_0\Rightarrow f_1\Rightarrow \dots \Rightarrow f_{n-1}\Rightarrow
  f_n=g$. Vertical composition of pre-$2$-cells, denoted by $\circ$, is given by concatenation.
  Horizontal composition (or whiskering) of pre-$2$-cells with morphisms is also given by concatenation: If $f,g\colon X\to Y$, $h\colon W\to
  X$, $h'\colon Y\to Z$ are paths and $\alpha\colon f\Rightarrow g$ is a pre-$2$-cell, then one has the
  pre-$2$-cell $h'*\alpha*h\colon h'*f*h\Rightarrow
  h'*g*h$. Consider systems which associate to every pair of paths $(f,g)$ sharing a source and a target,
  an equivalence relation on the set of pre-$2$-cells $f\Rightarrow
  g$. We say that one system $\sim$ is finer than another system
  $\sim'$ if $\alpha\sim \beta$ implies $\alpha\sim'\beta$. There is
  a finest system $\sim$ satisfying the following conditions:
  \begin{enumerate}
    \item \label{d.T01} (Stability under vertical and horizontal
        composition) If $f',f,g,g'\colon X\to Y$, $h\colon W\to X$,
        $h'\colon Y\to Z$ are paths, $\alpha\sim \beta\colon f\Rightarrow
        g$, $\eta\colon f'\Rightarrow f$, $\eta'\colon g\Rightarrow g'$
        are pre-$2$-cells, then $\eta'\circ \alpha \circ \eta \sim
        \eta'\circ \beta \circ\eta$ and $h'*\alpha*h\sim h'*\beta*h$.

    \item \label{d.T02} (Interchange law) If $f,f'\colon X\to Y$,
        $g,g'\colon Y\to Z$ are paths, $\alpha\colon f\Rightarrow f'$ and
        $\beta\colon g\Rightarrow g'$ are pre-$2$-cells, then $(f'*\beta)
        \circ (\alpha*g)\sim (\alpha*g')\circ (f*\beta)$.

    \item \label{d.T03} (Creation and deletion of unit) For every object
        $X$ of $\cC$,
        \[\theta_{\one_X,\one_X}\circ \iota_{\one_X,\one_X}\sim
        \one_{\one_X},\quad \iota_{\one_X,\one_X}\circ
        \theta_{\one_X,\one_X}\sim \one_{\one_X}.\]

    \item \label{d.T04} (Unit and composition of morphisms) For every
        morphism $f\colon X\to Y$ of $\cC$,
    \[\iota_{f,\one_X}\sim \delta_{\one_Y,f,\one^{\cC}_X,\one_X},\quad
        \iota_{\one_Y,f}\sim \delta_{\one_Y,\one^{\cC}_Y,f,\one_X},
        \]

    \item \label{d.T05} (Composition and decomposition of morphisms) For
        every sequence $X\xrightarrow{f} Y\xrightarrow{g} Z$ of morphisms
        in $\cC$,
        \[\delta_{\one_Z,g,f,\one_X}\circ
\gamma_{\one_Z,g,f,\one_X}\sim
        \one_{g*f},\quad \gamma_{\one_Z,g,f,\one_X}\circ
        \delta_{\one_Z,g,f,\one_X}\sim \one_{gf}.\]

    \item \label{d.T06} (Associativity of composition of morphisms) For
        every sequence $X\xrightarrow{f} Y\xrightarrow{g} Z\xrightarrow{h}
        W$ of morphisms in $\cC$,
        \[
        \gamma_{\one_W,h,gf,\one_X}\circ
        \gamma_{h,g,f,\one_X}\sim
        \gamma_{\one_W,hg,f,\one_X}\circ \gamma_{\one_W,h,g,f}.
        \]

    \item \label{d.T07} (Unit $2$-cell) For every morphism $f\colon X\to
        Y$ of $\cC$, we have
        \[\sigma_{\id_Y,\id_f^\cC,\id_X}\sim \id_f,\]
        where $\id_f^\cC$ is the identity $2$-cell on $f$ in $\cC$.

    \item \label{d.T08} (Vertical composition of $2$-cells) For a sequence
        of $2$-cells $g\xRightarrow{\alpha}g' \xRightarrow{\beta}g''$ in
        $\cC$, where $g,g',g''\colon X\to Y$, we have
        \[
        \sigma_{\id_Y,\beta,\id_X}\circ\sigma_{\id_Y,\alpha,\id_X}\sim \sigma_{\id_Y,\beta\alpha,\id_X}. \]

    \item \label{d.T09} (Horizontal composition of $2$-cells with
        morphisms)  For morphisms $f\colon X\to Y$ and $h\colon Z\to W$ of
        $\cC$, and a $2$-cell $\alpha\colon g\Rightarrow g'$ of $\cC$,
        where $g,g'\colon Y\to Z$, we have
        \begin{gather*}
        \sigma_{\id Z,\alpha,f}\sim \delta_{\id_Z,g',f,\id_X}\circ \sigma_{\id_Z,\alpha f,\id_X}\circ\gamma_{\id_Z,g,f,\id_X},\\
        \sigma_{h,\alpha,\id_Y}\sim \delta_{\id_W,h,g',\id_Y}\circ \sigma_{\id_W,h\alpha ,\id_Y}\circ\gamma_{\id_W,h,g,\id_Y}.
        \end{gather*}

\end{enumerate}
\end{Definition}

In the terminology of Street \cite[page 547]{Street}, we have given a
presentation of the $2$-category $\cT\cC$ consisting of the underlying graph
of $\cC$, the derivation scheme given by the atomic $2$-cells with $f$ and
$g$ being identities, and the relation on $2$-cells given by conditions
\ref{d.T03} through \ref{d.T09}. The $2$-category of objects, morphisms, and
pre-$2$-cells in the above definition is the free $2$-category on the
computad consisting of the graph and the derivation scheme.

It follows from conditions \ref{d.T03}, \ref{d.T05}, \ref{d.T07}, and
\ref{d.T08} that $\cT\cC$ is a $(2,1)$-category whenever $\cC$ is a
$(2,1)$-category.

We define the \emph{reduction $2$-functor} $R\colon \cT\cC\to
\cC$\index{R@$R$} as follows. We take the identity on objects. To a path
$f_m*\dots
* f_1$ (where each $f_i$ is a morphism of $\cC$) we associate its composition $f_m\dotsm f_1$ in $\cC$. To all atomic $2$-cells except $\sigma$
we associate identities, and to $\sigma_{g,\alpha,f}$ we associate the
horizontal composition $g*\alpha *f$ in $\cC$. We define a pseudo functor
$G\colon \cC\to \cT\cC$ as follows. We take the identity on objects. To a
morphism $f$ we associate $f$, considered as a path of length $1$. The
coherence constraint is given by $\iota$ and~$\gamma$. To a $2$-cell
$\alpha\colon f\Rightarrow g$, where $f,g\colon X\to Y$, we associate
$\sigma_{\id_Y,\alpha,\id_X}$. We have $RG=\one_\cC$.

\begin{Proposition}\label{p.TC}
The reduction $2$-functor $R$ is a bi-equivalence, with $G$ being a pseudo
inverse.
\end{Proposition}

\begin{proof}
We construct a pseudo natural isomorphism $\eta\colon
\id_{\cT\cC}\Rightarrow GR$ carrying objects to identities as follows. To a
morphism $f_m*\dots*f_1$ of $\cT\cC$, we associate the $2$-cell
$f_m*\dots*f_1\Rightarrow f_m\dotsm f_1$ given by $\iota$ (creation of unit)
for $m=0$, identity for $m=1$, and $\gamma$ (composition) for $m\ge 2$.
\end{proof}

The following result says that pseudo functors with source $\cC$ can be
straightened to $2$-functors with source $\cT\cC$.

\begin{Proposition}\label{l.PsFun}
Composition with $G$ induces an isomorphism $\twoFunps(\cT\cC,\cD)\cong
\PsFun(\cC,\cD)$ of $2$-categories.
\end{Proposition}

\begin{proof}
The inverse is easily constructed, as follows. A pseudo functor $H\colon
\cC\to \cD$ induces a $2$-functor $\cT\cC\to \cD$ carrying a path $f_m*\dots
* f_1$ to $H(f_m)\dotsm H(f_1)$ and carrying the classes of $\iota$ and
$\gamma$ to the $2$-cells of the coherence constraint.
\end{proof}

By the above propositions, the inclusion $\twoFunps(\cT\cC,\cD)\subseteq
\PsFun(\cT\cC,\cD)$ is a $\cD^{\Ob(\cC)}$-$2$-equivalence. This has the
following generalization, which will be used later.

\begin{Lemma}\label{l.PsFun2}
Let $\cC$ and $\cD$ be $2$-categories. Assume that the underlying category
of $\cC$ is a free category on a graph $S\rightrightarrows \Ob(\cC)$. Then
the inclusion $\twoFunps(\cC,\cD)\subseteq \PsFun(\cC,\cD)$ is a
$\cD^{\Ob(\cC)}$-$2$-equivalence.
\end{Lemma}

The assumption means that there exists a class $S$ of morphisms of $\cC$
such that every morphism of $\cC$ is a composition in a unique way of
morphisms in $S$.

\begin{proof}
By Lemma \ref{l.faith}, it suffices to show that for every strictly unital
pseudo functor $H\colon \cC\to \cD$, there exists a $2$-functor $H'\colon
\cC\to \cD$ and a pseudo natural isomorphism $\epsilon\colon H\Rightarrow
H'$ carrying objects to identities. For any morphism $f=f_m\dotsm f_1$ of
$\cC$ with $f_i\in S$, we take $H'(f)=H(f_m)\dotsm H(f_1)$ and take
$\epsilon(f)$ to be the $2$-cell $H(f_m)\dotsm H(f_1)\Rightarrow H(f_m\dotsm
f_1)$ provided by the unital constraint for $m=0$, identity for $m=1$, and
composition constraint for $m\ge 2$. For any $2$-cell $\alpha\colon
f\Rightarrow g$, the $2$-cell $H'(\alpha)$ is given by $H(\alpha)$ and the
unital and composition constraints.
\end{proof}

We finish this section with a couple of constructions that will be used
later.

\begin{Construction}[Adjoints of the inclusion functor $\Cat\to \twoCat$]\label{c.cO}\index{O@$\bO$}
For a $2$-category $\cC$, we define a category $\bO\cC$ with the same
objects as $\cC$ by $(\bO\cC)(X,Y)=\pi_0(\cC(X,Y))$ for objects $X$ and $Y$
of $\cC$. Here $\pi_0$ is the set of connected components. We obtain a
functor $\bO\colon \twoCat\to \Cat$. The $2$-functor $\cC\to \bO\cC$
carrying a morphism to its connected component exhibits $\bO$ as a left
adjoint to the inclusion functor $\Cat\to \twoCat$. The latter also admits a
right adjoint $\twoCat\to \Cat$ sending a $2$-category to its underlying
category and a $2$-functor to its underlying functor.

A \emph{pseudo} functor $F\colon \cC\to \cD$ also induces a functor $\bO
F\colon \bO\cC\to \bO\cD$. A pseudo natural transformation $\alpha\colon
F\to G$ induces a natural transformation $\bO\alpha\colon \bO F\to \bO G$.
If there exists a modification $\alpha\to \beta$, then $\bO \alpha=\bO
\beta$. In particular, $\bO$ carries bi-equivalences to equivalences and
natural equivalences to natural isomorphisms.
\end{Construction}

\begin{Construction}[Adjoints of the inclusion functor $\twooneCat\to \twoCat$]\label{c.twoone}\index{L@$\cL$}
Let $\cC$ be a $2$-category. We define two $(2,1)$-categories, $\cL\cC$ and
$\cR\cC$, with the same objects as $\cC$, as follows. For objects $X$ and
$Y$ of~$\cC$, $(\cR\cC)(X,Y)$ is the greatest subgroupoid of $\cC(X,Y)$ and
$(\cL\cC)(X,Y)$ is the fundamental groupoid of $\cC(X,Y)$, namely the
category obtained from $\cC(X,Y)$ by inverting all morphisms of $\cC(X,Y)$
\cite[Section I.1]{GZ} (there is no set-theoretic issue by Convention
\ref{s.small}). Thus $\cR\cC$ is the maximal sub-$(2,1)$-category of $\cC$.
If we let $\twooneCat$ denote the big category of $(2,1)$-categories and
$2$-functors, we obtain functors $\cL,\cR\colon \twoCat\to \twooneCat$. The
``localization'' $2$-functor $\cC\to \cL\cC$ and the inclusion $2$-functor
$\cR\cC\to \cC$ exhibit $\cL$ and $\cR$ as left and right adjoints of the
inclusion $2$-functor $\twooneCat\to \twoCat$, respectively.

The operation $\cL$ also acts on pseudo functors, pseudo natural
transformations, and modifications, so that $\cL$ preserves bi-equivalences
and pseudo natural equivalences. The same holds for $\cR$ except that $\cR$
does not act on modifications in general, but only on invertible
modifications. Moreover, the ``localization'' $2$-functor $\cC\to \cL\cC$
and the inclusion $2$-functor $\cR\cC\to \cC$ induce isomorphisms of
$2$-categories
\[\PsFun(\cL\cC,\cD)\to \PsFun(\cC,\cD),\quad \PsFun(\cD,\cR\cC)\to \cR\PsFun(\cD,\cC),\]
for every $(2,1)$-category $\cD$. For any $2$-category $\cD$, the
$2$-functor $\PsFun(\cL\cC,\cD)\to \PsFun(\cC,\cD)$ identifies
$\PsFun(\cL\cC,\cD)$ with the $2$-full sub-$2$-category of $\PsFun(\cC,\cD)$
spanned by pseudo functors that factor through $\cR\cD$.
\end{Construction}

Our general discussion on categories and $2$-categories ends here. In the
next section we add $n$-fold categories to the picture.

\section{$2$-Categories and $n$-fold categories}\label{s.2}
In this section, after recalling the definitions of $n$-fold categories and
$n$-fold functors in Definition \ref{s.nfcat}, we investigate the relation
between $2$-categories and $n$-fold categories. The functor $\dK^n\colon
\Cat\to n\FoldCat$ carrying a category to its $n$-fold category of
hypercubes (Example \ref{ex.rhos} \ref{ex.rho2}) admits a left adjoint
$\bT_n\colon n\FoldCat\to \Cat$ carrying an $n$-fold category $\dC$ to its
category of paths (Remark \ref{r.adj}). One goal of this section is to
establish an analogue of this adjunction with $\Cat$ replaced by $\twoCat$
(Proposition \ref{p.adj}). To do this, we extend $\dK^n$ to a functor
$\dQ^n\colon \twoCat \to n\FoldCat$ carrying a $2$-category $\cD$ to its
$n$-fold category of hypercubes $\dQ^n\cD$ (Definition \ref{d.Q}, extending
the double category of up-squares \cite[2.C.1, p.~272]{BEhr} in the case
$n=2$). We also construct a left adjoint $\cT^\red_n\colon n\FoldCat\to
\twoCat$ (Definition \ref{d.T}) carrying $\dC$ to its reduced $2$-category
of paths $\cT^\red_n\dC$, which is a refinement of $\bT_n\dC$. This
adjunction is related to Gray's tensor product of $2$-categories. The
functors fit into the following diagram:
\[\xymatrix{\Cat\ar@<0.5ex>[r]\ar@/^2pc/[rr]_-{\dK^n} & \twoCat\ar@<0.5ex>[r]^-{\dQ^n}\ar@<0.5ex>[l]^{\bO} & n\FoldCat.\ar@<0.5ex>[l]^-{\cT_n^\red}\ar@/^2pc/[ll]_-{\bT_n}}\]
We also construct a variant $\cT_n$ of $\cT^\red_n$, closely related to
pseudo functors. We use these constructions to define the $2$-category of
gluing data and the descent $2$-functor $Q^n\colon \PsFun(\cC,\cD)\to
\GD_{\cA_1,\dots,\cA_n}(\cC,\cD)$, which are the main objects of study of
this article.

The notion of $n$-fold categories was introduced by Ehresmann
\cite[D\'efinition~15, p.~396]{Ehr}. His original definition proceeds by
induction on $n$. For our purpose, it is more convenient to adopt a direct
combinatorial definition: an $n$-fold category is a collection  $(\dC_J)_J$
of sets, where $J$ runs through subsets of $\{1,\dots,n\}$ and $\dC_J$ can
be visualized as a set of hypercubes of dimension $\# J$, endowed with
various sources, targets, units, and compositions. To specify the
compatibility between these data, it is convenient to introduce the
following notation. Note that the map carrying $J$ to its characteristic
function $\chi_J$\index{czhi@$\chi_J$} is a bijection from the set of
subsets of $\{1,\dots,n\}$ onto $\{0,1\}^n$.

\begin{Notation}\label{n.cI}\index{I@$\bI$}
  For $m\ge 0$, we let $[m]$ denote the totally ordered set
  $\{0,1,\dots, m\}$. We let $\bI$ denote the
  category whose objects are $[0]$ and $[1]$ and whose morphisms are
  nondecreasing maps. In other words, the morphisms of $\bI$ are the identity maps,
  the face maps $d^1_0, d^1_1\colon [0]\to [1]$ which skip $0$ and $1$ respectively, and the degeneracy map $s\colon [1]\to [0]$.
  We denote by $\Set$ the big category of sets. A category $\bC$ can be viewed as a functor $\bC\colon \bI^{\op}\to \Set$ with
  $\Ob(\bC)=\bC([0])$ and $\Mor(\bC)=\bC([1])$ endowed with a
  composition map
  \[\bC([1])\mathrel{\mathop{\times}\limits_{\bC(d^1_1),\bC([0]),\bC(d^1_0)}}\bC([1])\to \bC([1]).\]
  We omit the maps $\bC(d^1_1)$ and $\bC(d^1_0)$ when no confusion arises.

  For $n\ge 0$, we identify objects of $\bI^n$ with elements of $\{0,1\}^n$ and let
  $\epsilon_i$\index{epsilon@$\epsilon_i$} denote the element such that $\epsilon_i(i)=1$ and
  $\epsilon_i(j)=0$ for $j\neq i$. For an object of $\alpha$, we let
  $\Sigma\alpha$ denote the sum $0\le \alpha_1+\dots +\alpha_n \le n$.
\end{Notation}

The following definition is similar to \cite[Definition
2.2]{FP}\footnote{Our axiom \ref{s.nfcat2} appears to be missing in
\cite[Definition 2.2]{FP}.}.

\begin{Definition}\label{s.nfcat}\index{FoldCat@$n\FoldCat$}\index{=@$\circ^i$}\index{FoldFun@$\FFun{n}(\dC,\dD)$}
  An \emph{$n$-fold category} is a functor $\dC\colon (\bI^n)^{\op}\to \Set$ endowed
  with a composition map
  \[\circ^i\colon \dC(\alpha')\times_{\dC({\alpha})}\dC(\alpha') \to
  \dC(\alpha')\]
  for every $1\le i\le n$ and every pair $(\alpha,\alpha')$ of objects of $\bI^n$
  satisfying $\alpha'=\alpha+\epsilon_i$, such that the following axioms
  hold:
  \begin{enumerate}
    \item For all $i$ and $(\alpha,\alpha')$ as above, the composite
        $\bI^{\op}\xrightarrow{\iota} (\bI^n)^{\op} \xrightarrow{\dC}
        \Set$, where $\iota$ is the functor carrying $[0]$ to $\alpha$ and
        $[1]$ to $\alpha'$, is a category.

    \item \label{s.nfcat2} (functoriality) For every $1\le i\le n$ and
        every morphism $\beta\to \alpha$ in $\bI^n$ satisfying
        $\alpha_i=\beta_i=0$, the diagram
    \[
      \xymatrix{\dC(\alpha')\times_{\dC(\alpha)}\dC(\alpha') \ar[r]^-{\circ^i}\ar[d] & \dC(\alpha')\ar[d]\\
      \dC(\beta') \times_{\dC(\beta)} \dC(\beta') \ar[r]^-{\circ^i} & \dC(\beta'),}
    \]
    where $\alpha'=\alpha+\epsilon_i$, $\beta'=\beta+\epsilon_i$,
    commutes.

    \item (interchange law) For all $1\le i<j\le n$ and every object
        $\alpha$ of $\bI^n$ satisfying $\alpha_i=\alpha_j=0$, the diagram
    \[\xymatrix{X\ar[r]^-{\circ^i\times\circ^i}\ar[d]_{\circ^j\times\circ^j}
    & \dC(\bar \alpha')\times_{\dC(\alpha')}\dC(\bar \alpha')\ar[d]^{\circ^j}\\
    \dC(\bar \alpha')\times_{\dC(\bar \alpha)}\dC(\bar \alpha') \ar[r]^-{\circ^i}
    &\dC(\bar \alpha')}
    \]
    commutes. Here $\alpha'=\alpha+\epsilon_i$,
    $\bar\alpha=\alpha+\epsilon_j$,
    $\bar\alpha'=\alpha+\epsilon_i+\epsilon_j$, and $X$ is the limit of
    the diagram
    \[\xymatrix{\dC(\bar \alpha')\ar[r]\ar[d] & \dC(\alpha') & \dC(\bar \alpha')\ar[d]\ar[l]\\
    \dC(\bar \alpha) & & \dC(\bar \alpha)\\
    \dC(\bar \alpha')\ar[r]\ar[u] & \dC(\alpha') & \dC(\bar \alpha').\ar[l]\ar[u]}\]
  \end{enumerate}

  Elements of the set $\Ob(\dC)\coloneqq \dC(\bzero)$, where $\bzero=(0,\dots,0)$, are
  called \emph{objects} of $\dC$. Elements of $\dC(\epsilon_j)$ are called \emph{morphisms in direction $j$}.
  Elements of $\dC(\alpha)$ are called \emph{$\alpha$-hypercubes} (or \emph{$J$-hypercubes} for $\alpha=\chi_J$) of $\dC$.

  An \emph{$n$-fold functor} $\dC\to \dD$ between $n$-fold categories $\dC$ and $\dD$ is a natural
  transformation compatible with the composition maps. We let
  $\FFun{n}(\dC,\dD)$ denote the set of $n$-fold functors from $\dC$
  to $\dD$.

  We let $n\FoldCat$ denote the big category of $n$-fold categories and $n$-fold
  functors.
\end{Definition}

We have an isomorphism $1\FoldCat\cong \Cat$.

\begin{Example}\index{h@$\rh$}\index{v@$\rv$}
  A $2$-fold category is called a \emph{double category} \cite[D\'efinition~10,
  p.~389]{Ehr}. A double category $\dC$ consists of a set $\Ob(\dC)=\dC(0,0)$ of objects,
  a set $\Hor(\dC)=\dC(1,0)$ of horizontal morphisms, a set $\Ver(\dC)=\dC(0,1)$ of vertical morphisms, and a set
  $\rSq(\dC)=\dC(1,1)$ of squares, equipped with various sources, targets,
  and associative and unital compositions. We will sometimes enumerate the elements of the set of directions \{1,2\} using the
  notation $\rh=1$ and $\rv=2$. Functoriality implies that given two squares
  that are horizontally composable
  as shown in the diagram
  \[\xymatrix{X_1\ar[r]^j \ar[d]\ar@{}[rd]|D & X_2\ar[r]^{j'}\ar[d]\ar@{}[rd]|{D'} & X_3\ar[d]\\
  Y_1\ar[r]^{i} & Y_2\ar[r]^{i'} & Y_3,}
  \]
  the upper arrow of the horizontal composition $D'\circ^{\rh} D$ is $j'\circ
  j$,
  and the lower arrow is $i'\circ i$. The same holds for vertical
  composition. The interchange law means that given four squares as shown in
  the diagram
  \[\xymatrix{X_1\ar[r] \ar[d]\ar@{}[rd]|D & X_2\ar[r]\ar[d]\ar@{}[rd]|{D'} & X_3\ar[d]\\
  Y_1\ar[d]\ar[r]\ar@{}[rd]|E & Y_2\ar[d]\ar[r]\ar@{}[rd]|{E'} & Y_3\ar[d]\\
  Z_1\ar[r] & Z_2\ar[r] & Z_3, }
  \]
  the two ways of composing them produce the same square:
  \[(E'\circ^{\rh} E) \circ^{\rv} (D'\circ^{\rh} D)=(E'\circ^{\rv} D')\circ^{\rh}(E\circ^{\rv}
  D).
  \]
\end{Example}

\begin{Remark}
  Fiore and Paoli defined the $n$-fold nerve functor, which is a
  fully faithful functor from the big category $n\FoldCat$ to
the big category of $n$-fold simplicial sets \cite[Definition 2.14,
Proposition 2.17]{FP}. As in the case $n=1$ (see for example
\cite[Proposition 1.1.2.2]{Lurie}; see also \cite[Proposition
VI.2.2.3]{Illusie}), the
  essential image of the functor consists of $n$-fold simplicial sets
  satisfying the unique right lifting property with respect to the inclusions
  \[\Lambda_k^{m_i}\boxtimes
  \left(\mathop{\boxtimes}\limits_{\substack{1\le j\le n\\j\neq i}} \Delta^{m_j}\right) \subseteq \Delta^{m_1}\boxtimes\dots
  \boxtimes \Delta^{m_n},\]
  $1\le i\le n$, $0< k< m_i$, and $m_1,\dots, m_n\ge 0$. Here $\Lambda_k^{m_i}\subseteq
  \Delta^{m_i}$ denotes the $k$-th horn in the $m_i$-simplex. Multisimplicial
  sets satisfying suitable lifting properties play an essential role in the theory
  of gluing functors between $\infty$-categories developed in \cite{LZ1}.
\end{Remark}

\begin{Notation}[External product $\dC\boxtimes \dD$ of an $m$-fold category $\dC$ and an $n$-fold category $\dD$]\index{=@$\boxtimes$}
Let $\dC$ be an $m$-fold category and let $\dD$ be an $n$-fold category. We
let $\dC \boxtimes \dD$ denote their \emph{external product}, which is the
$(m+n)$-fold category $\dE$ given by $\dE(\alpha)=\dC(\beta)\times
\dD(\gamma)$ for $\alpha=(\beta_1,\dots,\beta_m,\gamma_1,\dots,\gamma_n)$,
with $\circ^i_\dE$ for $1\le i\le m$ given by the $\circ^i_\dC$  and
$\circ^{m+i}_\dE$ for $1\le i\le n$ given by $\circ^{i}_{\dD}$. We get a
functor
\[\boxtimes \colon m\FoldCat \times n\FoldCat \to (m+n)\FoldCat.\]

In particular, given categories $\bC_1,\dots, \bC_n$, the external product
$\dC=\bC_1\boxtimes \dots \boxtimes \bC_n$ is the $n$-fold category defined
in \cite[Definition 2.9]{FP}. We have $\dC(\alpha)=\prod_{1\le i\le
    n}\bC_i(\alpha_i)$ with $\circ^i$ given by the composition in
    $\bC_i$.
\end{Notation}

For any $n$-fold category $\dD$ and any object $\alpha$ of $\bI^n$, the set
$\dD(\alpha)$ of $\alpha$-hypercubes can be identified with
$\FFun{n}(\alpha_1\boxtimes\dots \boxtimes\alpha_n,\dD)$.

\begin{Remark}\index{FoldFun@$\dFFun{n}(\dC,\dD)$}
Let $\dC$ and $\dD$ be $n$-fold categories. The $n$-fold functors $\dC\to
\dD$ can be organized into an $n$-fold category
  $\dFFun{n}(\dC,\dD)$ as follows. For any object $\alpha$ of $\bI^n$, we define $\dFFun{n}(\alpha)\coloneqq \FFun{n}((\alpha_1\boxtimes \dots \boxtimes \alpha_n)\times \dC,\dD)$.
If $\dE$ is another $n$-fold category, we have
\[\dFFun{n}(\dC,\dFFun{n}(\dD,\dE))\cong \dFFun{n}(\dC\times
  \dD,\dE).
\]
\end{Remark}

We now define the pullback functor $\phi^*$ for a map $\phi$ between two
sets of directions.

\begin{Definition}[The functor $\phi^*$]\label{d.phi*}\index{-*@$(-)^*$!$\phi^*$}
 Let $\phi\colon \{1,\dots m\}\to\{1,\dots, n\} $ be a map
    and let $\dC$ be an $n$-fold category. We define an $m$-fold category $\phi^*
    \dC$  by
\begin{equation}\label{e.phistar}
    (\phi^*\dC)(\alpha)\coloneqq \FFun{n}\left(\mathop\boxtimes_{j=1}^n \prod_{\phi(i)=j}
    \alpha_i,\dC\right)
\end{equation}
    for every object $\alpha$ of $\bI^m$. In particular, we have $(\phi^*\dC)(\epsilon_i)=\dC(\epsilon_{\phi(i)})$. For $1\le i \le m$, the composition $\circ^i$ in $\phi^*\dC$ is given by
    $\circ^{\phi(i)}$ in $\dC$. More formally, for objects $\alpha$ and $\alpha'$ of $\bI^m$ satisfying
    $\alpha'=\alpha+\epsilon_i$, the composition $\circ^i$ is given by the map
    \[(\phi^*\dC)(\alpha')\times_{(\phi^*\dC)(\alpha)}(\phi^*\dC)(\alpha')\cong
    \FFun{n}\left(\mathop\boxtimes_{j=1}^n\prod_{\phi(k)=j} \tilde\alpha_k,\dC\right) \to (\phi^*\dC) (\alpha')
    \]
    induced by $d^2_1$, where
    \[\tilde\alpha_k=\begin{cases}
      [2] & \text{if $k=i$,}\\
      \alpha_k & \text{if $k\neq i$.}
    \end{cases}\]
We obtain a functor $\phi^*\colon n\FoldCat\to m\FoldCat$.
\end{Definition}

For a sequence of maps $\{1,\dots,l\}\xrightarrow{\psi} \{1,\dots
m\}\xrightarrow{\phi}\{1,\dots, n\}$, we have $(\phi\psi)^*\cong
\psi^*\phi^*$.

\begin{Example}[Special cases of the functor $\phi^*$]\label{ex.rhos}\leavevmode
\begin{enumerate}
\item Let $\iota_i\colon \{1\}\to \{1,\dots,n\}$ be the map with image
    $\{i\}$. For an $n$-fold category $\dC$, we have
    $\Ob(\iota_i^*\dC)\cong \dC(0,\dots,0)$ and $\Mor(\iota_i^*\dC)\cong
    \dC(\epsilon_i)$. More generally, the functor $\phi^*$ for $\phi$
    strictly increasing has been studied by Fiore and Paoli \cite[Notation
    2.12]{FP}.

\item \label{ex.rho2} Let $\pr_n\colon \{1,\dots, n\}\to \{1\}$. We denote
    the functor $\pr_n^*$ by $\dK^n$\index{Kn@$\dK^n$}. For a category
    $\bC$, we have $(\dK^n\bC)(\alpha)=\Fun(\alpha_1\times\dots \times
    \alpha_n,\cC)$. In other words, $\alpha$-hypercubes of $\dK^n\bC$ are
    the commutative hypercubes of dimension $\Sigma \alpha$ in $\bC$. In
    particular, morphisms in each direction are the morphisms of $\bC$. We
    call $\dK^n\bC$ the \emph{$n$-fold category of hypercubes in $\bC$}.
    This was defined by Ehresmann \cite[end of p.~398]{Ehr} and, in the
    case $n=2$, was called the double category of quartets
    (\emph{quatuors} in French) \cite[Proposition 12, p.~394]{Ehr}.

\item \index{-t@$(-)^t$!$\dC^t$} Let $t\colon \{1,2\} \to \{1,2\}$ be the
    map swapping $1$ and $2$. Then $\dC^t\coloneqq t^*\dC$ is the
    transpose of $\dC$ in the sense that $\Ob(\dC^t)=\Ob(\dC)$,
    $\Hor(\dC^t)=\Ver(\dC)$,
  $\Ver(\dC^t)=\Hor(\dC)$, and $\rSq(\dC^t)$ is obtained from $\rSq(\dC)$
  by transposing the squares.
\end{enumerate}
\end{Example}

\begin{Remark}[The functor $\phi_!$]
For any map $\phi\colon \{1,\dots,m\}\to \{1,\dots,n\}$, the functor
$\phi^*$ admits a left adjoint $\phi_!\colon m\FoldCat\to n\FoldCat$. Let
$\dC$ be an $m$-fold category. In the case where $\phi$ is injective,
$\phi_!\dC$ can be easily described by the formula
$(\phi_!\dC)(\alpha)=\dC(\alpha\circ \phi)$ for every object $\alpha$ of
$\bI^n$, where $\alpha\circ \phi$ is the object of $\bI^m$ given by
$(\alpha\circ \phi)_i=\alpha_{\phi(i)}$, and composition in direction $i$ of
$\dC$ provides composition in direction $\phi(i)$ of $\phi_!\dC$. In
general, the assignment $\alpha\mapsto \dC(\alpha\circ \phi)$ does not
provide an $n$-fold category, since composition in direction $j$ is not
well-defined whenever $\#\phi^{-1}(j)>1$. In general, $\phi_!\dC$ is the
fundamental $n$-fold category \cite[Proposition 2.18]{FP} of the $n$-fold
simplicial set $(\alpha_1,\dots,\alpha_n)\mapsto
\FFun{n}(\alpha_{\phi(1)}\boxtimes \dots \boxtimes \alpha_{\phi(m)},\dC)$,
where each $\alpha_i$ is a (combinatorial) simplex.

For an $m$-fold category $\dC$ and an $n$-fold category $\dD$, we have
$\dC\boxtimes \dD\cong (\phi_!\dC)\times (\psi_!\dD)$, where $\phi\colon
\{1,\dots,m\}\to \{1,\dots,m+n\}$ is the inclusion and $\psi\colon
\{1,\dots,n\}\to \{1,\dots,m+n\}$ is the map given by $\psi(i)=m+i$. In
\eqref{e.phistar}, we have $\boxtimes_{j=1}^n\prod_{\phi(k)=j}\alpha_k\cong
\phi_!(\alpha_1\boxtimes \dots \boxtimes \alpha_m)$.
\end{Remark}

For the map $\pr_n\colon \{1,\dots,n\}\to \{1\}$, the category
$(\pr_n)_!\dC$ associated to an $n$-fold category $\dC$ is the fundamental
category of the simplicial set $[m]\mapsto \FFun{n}([m]^{\boxtimes n},\dC)$.
A priori this construction involves $\dC(1,\dots,1)$ hence hypercubes of
$\dC$ of all dimensions. However, we have the following simpler description.

\begin{Definition}[Category of paths $\bT_n\dC$ in an $n$-fold category
$\dC$]\label{d.cT}\index{Tn@$\bT_n$} Let $\dC$ be an $n$-fold category. We
define the \emph{category $\bT_n \dC$ of paths in $\dC$} as follows. The
objects of $\bT_n \dC$ are the objects of $\dC$. A \emph{morphism} $X\to Y$
of $\bT_n \dC$ is an equivalence class of paths
  \[X=X_0\xrightarrow{f_1} X_1\xrightarrow{f_2} \dots \xrightarrow{f_{m-1}} X_{m-1} \xrightarrow{f_m} X_m=Y, \quad m\ge 0\]
where $f_i\in \coprod_{1\le k\le n} \dC(\epsilon_k)$, $1\le i\le m$ under
the equivalence relation $\sim$ generated by the following conditions:
  \begin{enumerate}
    \item (unit) $g* f\sim g* \one^k_Y *f$, where $X\xrightarrow{f}
        Y\xrightarrow{g} Z$ is a sequence of paths, $\one^k_Y\in
        \dC(\epsilon_k)$ is identity morphism of $Y$.

    \item (composition) $g* h'* h * f\sim g* (h'h)
        * f$, where $X\xrightarrow{f} Y\xrightarrow{h} Y'
        \xrightarrow{h'} Y''\xrightarrow{g} Z$ is a sequence of paths, $h$
        and $h'$ belong to the same $\dC(\epsilon_k)$ and $h'h$ is their
        composition in $\dC$.

    \item (square) $g*i*q*f\sim g*p*j*f$, where $i,j\in \dC(\epsilon_k)$,
        $p,q\in \dC(\epsilon_{k'})$, $k<k'$, $D\in
        \dC(\epsilon_k+\epsilon_{k'})$ is of the form
\begin{equation}\label{e.dCsquare}
    \xymatrix{X\ar[r]^j\ar[d]_q\ar@{}[rd]|D & Y\ar[d]^p\\
    Z\ar[r]^i & W,}
\end{equation}
    $g\colon W\to W'$ and $f\colon X'\to X$ are paths.
  \end{enumerate}
Here $*$ denotes concatenation of paths. The identity morphism $\one_X\colon
X\to X$ in $\bT_n\dC$ is given by the path of length~$0$. Composition of
morphisms is given by concatenation of paths.

We obtain a functor $\bT_n \colon n\FoldCat\to \Cat$.
\end{Definition}

In other words, $\bT_n\dC$ is the quotient of the free category on the graph
$\coprod_{1\le k\le n}\dC(\epsilon_k)\rightrightarrows \Ob(\dC)$ by the
relations (1) through (3) above.

Note that $\bT_n\dC$ does not depend on hypercubes of $\dC$ of dimension $>
2$. We have an isomorphism $\bT_n\dC\cong (\pr_n)_!\dC$ that is the identity
on objects and carries the class of $f\in \dC(\epsilon_i)$ to the class of
the image of $f$ under the degeneracy map $\dC(\epsilon_i)\to
\dC([1],\dots,[1])$. The inverse carries the class of a hypercube $C\in
\dC(\alpha)$ to the class of any path from the initial vertex to the final
vertex of the hypercube, the class of the path being independent of the
choice of the path.

For $n$-fold categories $\dC$ and $\dD$, we have an isomorphism
$\bT_n(\dC\times\dD)\cong (\bT_n\dC)\times (\bT_n\dD)$ that is the identity
on objects and carries the class of $(f,g)\in (\dC\times \dD)(\epsilon_k)$
to $([f],[g])$, where $[f]$ and $[g]$ denote the classes of $f$ and $g$,
respectively. The inverse carries $([f_m*\dots* f_1],[g_l*\dots*g_1])$ to
the path $(f_m,\id)*\dots *(f_1,\id)*(\id, g_l)*\dots *(\id,g_1)$, which is
equivalent to the path $(\id, g_l)*\dots *(\id,g_1)*(f_m,\id)*\dots
*(f_1,\id)$ by conditions (2) and (3) above.

For an $m$-fold category $\dC$ and an $n$-fold category $\dD$, we have
$\bT_{m+n}(\dC\boxtimes \dD)\cong (\bT_m \dC)\times (\bT_n\dD)$. In
particular, for categories $\bC_1,\dots,\bC_n$, we have
$\bT_n(\bC_1\boxtimes \dots \boxtimes \bC_n)\cong \bC_1\times \dots \times
\bC_n$.

\begin{Remark}[Adjunction between $\bT_n$ and $\dK^n$]\label{r.adj}
Let $\dK^n=\pr_n^*\colon \Cat\to n\FoldCat$ be the functor carrying a
category to its $n$-fold category of hypercubes as in Example \ref{ex.rhos}
\ref{ex.rho2}. Let us describe the adjunction between $\bT_n$ and $\dK^n$
more explicitly. For every category $\bD$, we have an isomorphism of
categories $\bT_n\dK^n\bD\cong \bD$ which is the identity on objects and
carries the morphism represented by a path $f_m*\dots *f_1$, where $f_i\in
\coprod_{1\le k\le n}(\dK^n\bD)(\epsilon_k)$ to the composite $f_m\dotsm
f_1$. For every $n$-fold category $\dC$, we have an $n$-fold functor $\dC\to
\dK^n\bT_n\dC$ carrying a hypercube $C\in \dC(\alpha)$, corresponding to an
$n$-fold functor $\alpha_1\boxtimes \dots \boxtimes \alpha_n \to \dC$, to
the functor $\alpha_1\times \dots \times \alpha_n\cong
\bT_n(\alpha_1\boxtimes \dots \boxtimes \alpha_n)\to \bT_n\dC$. These
constructions exhibit $\bT_n$ as a left adjoint of $\dK^n$. Furthermore, we
have an isomorphism of $n$-fold categories
\begin{equation}\label{e.adj}
\dK^n\bFun(\bT_n\dC,\bD)\cong \dFFun{n}(\dC,\dK^n\bD),
\end{equation}
functorial in $\dC$ and $\bD$. Indeed, we have
\[\Fun(\alpha_1\times \dots \times \alpha_n,\bFun(\bT_n\dC,\bD))\cong \Fun(\bT_n((\alpha_1\boxtimes \dots \boxtimes \alpha_n)\times\dC),\bD) \cong \FFun{n}((\alpha_1\boxtimes \dots \boxtimes \alpha_n)\times\dC,\dK^n\bD).\]
\end{Remark}

The remainder of this section is devoted to $2$-categorical analogues of the
above. For the inclusion map $\iota_i\colon \{1\}\to \{1,\dots,n\}$ with
image $i$, the functor $\iota_i^*\colon n\FoldCat\to \Cat$ has the following
$2$-categorical refinement.

\begin{Definition}\index{H@$\cH$}\index{Hij@$\cH_{i,j}$}\index{V@$\cV$}
  To any double category $\dC$, one associates the \emph{underlying
  horizontal $2$-category} $\cH \dC$ and the \emph{underlying vertical
  $2$-category} $\cV \dC$. The underlying category of $\cH \dC$ is
  $\iota_1^*\dC=(\Ob(\dC),\Hor(\dC))$ and the underlying category of $\cV \dC$ is
  $\iota_2^*\dC=(\Ob(\dC),\Ver(\dC))$. A $2$-cell $\alpha\colon f\Rightarrow g$ in $\cH \dC$ is a square in $\dC$ of the boundary
\begin{equation}\label{e.sqh}
    \xymatrix{X\ar[r]^g\ar@{=}[d] & Y\ar@{=}[d]\\
  X\ar[r]^f & Y.}
\end{equation}
A $2$-cell $\alpha\colon f\Rightarrow g$ in $\cV\dC$ is a square in $\dC$ of
the boundary
\begin{equation}\label{e.sqv}
      \xymatrix{X\ar@{=}[r]\ar[d]_f & X\ar[d]^g\\
      Y\ar@{=}[r] & Y.}
\end{equation}
We have isomorphisms $\cH(\dC^t)\cong(\cV\dC)^{\twoop}$,
$\cV(\dC^t)\cong(\cH\dC)^{\twoop}$ (see Notation \ref{n.op}).

More generally, to an $n$-fold category $\dC$, and $1\le i,j \le n$, $i\neq
j$, one associates the $2$-category
\[\cH_{i,j}\dC=\begin{cases}
\cH(\iota_{i,j}^*\dC) & \text{if $i<j$,}\\
\cV(\iota_{j,i}^*\dC) & \text{if $i>j$,}
\end{cases}
\]
where $\iota_{i,j}\colon \{1,2\}\to \{1,\dots,n\}$ is the map sending $1$ to
$i$ and $2$ to $j$. We obtain a functor $\cH_{i,j}\colon n\FoldCat\to
\twoCat$.
\end{Definition}

We now proceed to give a $2$-categorical analogue of the adjunction between
$\bT_n$ and $\dK^n$. We start by $2$-categorical refinements of $\bT_n$.

\begin{Definition}[$2$-Categories of paths $\cT_n\dC$ and $\cT^\red_n \dC$ in an $n$-fold category
$\dC$]\label{d.T}\index{Tn@$\cT_n$}
  Let $\dC$ be an $n$-fold category. We define the \emph{$2$-category $\cT_n
  \dC$ of paths in
  $\dC$} as follows. The objects of $\cT_n\dC$ are the object of $\dC$.
  A \emph{morphism} $X\to Y$ of $\cT_n \dC$ is a path
  \[X=X_0\xrightarrow{f_1} X_1\to \dots \to X_{m-1} \xrightarrow{f_m} X_m=Y, \quad m\ge 0\]
  where $f_i\in \coprod_{1\le k\le n} \dC(\epsilon_k)$, $1\le i\le m$. The identity morphism $\one_X\colon X\to
  X$ in $\cT_n\dC$
  is the path of length~$0$. Composition of morphisms is given by concatenation of paths and is denoted by $*$.
  An
  \emph{atomic $2$-cell} between two paths sharing a source and a target is one of the
  following:
  \begin{enumerate}
    \item (creation of
        unit)\index{iotak@$\iota^k_{g,f}$}\index{tzhetak@$\theta^k_{g,f}$}
        $\iota^k_{g,f}\colon g* f\Rightarrow g* \one^k_Y *f$ or (deletion
        of unit) $\theta^k_{g,f}\colon g* \one^k_Y
        * f\Rightarrow g* f$, where $X\xrightarrow{f} Y\xrightarrow{g} Z$
        is a sequence of paths, $\one^k_Y\in \dC(\epsilon_k)$ is identity
        morphism of $Y$.

    \item
        (composition)\index{gzamma@$\gamma_{g,h',h,f}$}\index{delta@$\delta_{g,h',h,f}$}
        $\gamma_{g,h',h,f}\colon g* h'* h
        * f\Rightarrow g* (h'h)
        * f$ or (decomposition) $\delta_{g,h',h,f}\colon g* (h'h)
        * f\Rightarrow g* h'* h * f$, where $X\xrightarrow{f} Y\xrightarrow{h} Y'
        \xrightarrow{h'} Y''\xrightarrow{g} Z$ is a sequence of
        paths, $h$ and $h'$ belong to the same $\dC(\epsilon_k)$
        and $h'h$ is their composition in $\dC$.

    \item (square)\index{sigma@$\sigma_{g,D,f}$} $\sigma_{g,D,f}\colon
        g*i*q*f\Rightarrow g*p*j*f$, where $i,j\in \dC(\epsilon_k)$,
        $p,q\in \dC(\epsilon_{k'})$, $k<k'$, $D\in
        \dC(\epsilon_k+\epsilon_{k'})$ is of the form \eqref{e.dCsquare},
    $g\colon W\to W'$ and $f\colon X'\to X$ are paths.
  \end{enumerate}
  A \emph{pre-$2$-cell} $f\Rightarrow g$ is a sequence of atomic
  $2$-cells $f=f_0\Rightarrow f_1\Rightarrow \dots \Rightarrow f_{n-1}\Rightarrow
  f_n=g$. Vertical composition of pre-$2$-cells, denoted by $\circ$, is given by concatenation.
  Horizontal composition of pre-$2$-cells with morphisms is also given by concatenation: If $f,g\colon X\to Y$, $h\colon W\to
  X$, $h'\colon Y\to Z$ are paths and $\alpha\colon f\Rightarrow g$ is a pre-$2$-cell, then one has the
  pre-$2$-cell $h'*\alpha*h\colon h'*f*h\Rightarrow
  h'*g*h$. Consider systems which associate to every pair of paths $(f,g)$ sharing a source and a target,
  an equivalence relation on the set of pre-$2$-cells $f\Rightarrow
  g$. We say that one system $\sim$ is finer than another system
  $\sim'$ if $\alpha\sim \beta$ implies $\alpha\sim'\beta$. There is
  a finest system $\sim$ satisfying the following conditions:
  \begin{enumerate}
    \item \label{d.T1} (Stability under horizontal and vertical
        composition) If $f,g\colon X\to Y$, $h\colon W\to X$, $h'\colon
        Y\to Z$ are paths, $\alpha\sim \beta\colon f\Rightarrow g$,
        $\eta\colon f'\Rightarrow f$, $\eta'\colon g\Rightarrow g'$ are
        pre-$2$-cells, then $\eta'\circ \alpha \circ \eta \sim \eta'\circ
        \beta \circ\eta$ and $h'*\alpha*h\sim h'*\beta*h$.

    \item \label{d.T2} (Interchange law) If $f,f'\colon X\to Y$,
        $g,g'\colon Y\to Z$ are paths, $\alpha\colon f\Rightarrow f'$ and
        $\beta\colon g\Rightarrow g'$ are pre-$2$-cells, then $(f'*\beta)
        \circ (\alpha*g)\sim (\alpha*g')\circ (f*\beta)$.

    \item \label{d.T3} (Creation and deletion of unit) For $1\le k\le n$
        and every object $X$ of $\dC$,
        \[\theta^k_{\one_X,\one_X}\circ \iota^k_{\one_X,\one_X}\sim
        \one_{\one_X},\quad \iota^k_{\one_X,\one_X}\circ
        \theta^k_{\one_X,\one_X}\sim \one_{\one^k_X}.\]

    \item \label{d.T4} (Unit and composition) For $1\le k\le n$ and
        $f\colon X\to Y$ belonging to $\dC(\epsilon_k)$,
    \[\iota^k_{f,\one_X}\sim \delta_{\one_Y,f,\one^k_X,\one_X},\quad
        \iota^k_{\one_Y,f}\sim \delta_{\one_Y,\one^k_Y,f,\one_X},
        \]

    \item \label{d.T5} (Composition and decomposition) For $1\le k\le n$
        and every sequence $X\xrightarrow{f} Y\xrightarrow{g} Z$ of
        morphisms in $\dC(\epsilon_k)$,
        \[\delta_{\one_Z,g,f,\one_X}\circ
\gamma_{\one_Z,g,f,\one_X}\sim
        \one_{g*f},\quad \gamma_{\one_Z,g,f,\one_X}\circ
        \delta_{\one_Z,g,f,\one_X}\sim \one_{gf}.\]

    \item \label{d.T6} (Associativity of composition) For $1\le k\le n$
        and every sequence $X\xrightarrow{f} Y\xrightarrow{g}
        Z\xrightarrow{h} W$ of morphisms in $\dC(\epsilon_k)$,
        \[
        \gamma_{\one_W,h,gf,\one_X}\circ
        \gamma_{h,g,f,\one_X}\sim
        \gamma_{\one_W,hg,f,\one_X}\circ \gamma_{\one_W,h,g,f}.
        \]

    \item \label{d.T7} (Unit square) For $1\le k, k'\le n$ ($k\neq k'$)
        and $f\colon X\to Y$ belonging to $\dC(\epsilon_k)$, we have
        \[\begin{cases}
        \sigma_{\one_Y,D,\one_X}\circ\iota^{k'}_{f,\one_X}\sim
        \iota^{k'}_{\one_Y,f} & \text{if $k<k'$,}\\
        \sigma_{\one_Y,D,\one_X}\circ\iota^{k'}_{\one_Y,f}\sim
        \iota^{k'}_{f,\one_X} &\text{if $k'<k$,}
        \end{cases}
        \]
        where $D=s(f)\in \dC(\epsilon_k+\epsilon_{k'})$ is the
    identity square
    \begin{equation}\label{e.vidsquare}
        \xymatrix{X\ar[r]^f\ar@{=}[d] &Y\ar@{=}[d]\\
        X\ar[r]^f & Y.}
    \end{equation}

    \item \label{d.T8} (Composition of squares) For $1\le k,k'\le n$
        ($k\neq k'$) and $D''=D'\circ^{k} D$ in
        $\dC(\epsilon_k+\epsilon_{k'})$ of the form
  \begin{equation}\label{e.hsquare}
  \xymatrix{X_1\ar[r]^j \ar[d]_{p_1}\ar@{}[rd]|D & X_2\ar[r]^{j'}\ar[d]^{p_2} \ar@{}[rd]|{D'} &
  X_3\ar[d]^{p_3}\\
  Y_1\ar[r]^{i} & Y_2\ar[r]^{i'} & Y_3,}
  \end{equation}
we have
        \[\begin{cases}\gamma_{p_3,j',j,\one_{X_1}}\circ
        \sigma_{\one_{Y_3},D',j}\circ
        \sigma_{i',D,\one_{X_1}}\sim
        \sigma_{\one_{Y_3},D'',\one_{X_1}}\circ\gamma_{\one_{Y_3},i',i,p_1} & \text{if $k<k'$,}\\
\gamma_{\one_{Y_3},i',i,p_1}\circ
        \sigma_{i',D,\one_{X_1}}\circ
        \sigma_{\one_{Y_3},D',j}\sim
        \sigma_{\one_{Y_3},D'',\one_{X_1}}\circ
        \gamma_{p_3,j',j,\one_{X_1}} & \text{if $k'<k$.}
\end{cases}
        \]

\item \label{d.T9} (Commutativity of cubes) For $1\le k<k'<k''\le n$ and
    $C\in \dC(\epsilon_k+\epsilon_{k'}+\epsilon_{k''})$ of the form
     \[
    \xymatrix{X'\ar[dd]_{x}\ar[dr]^{q'}\ar[rr]^{b'} && Y'\ar[dd]^(.3){y}|\hole \ar[dr]^{p'}\\
    & Z'\ar[dd]^(.3){z}\ar[rr]^(.3){a'} && W'\ar[dd]^w \\
    X\ar[dr]_q\ar[rr]^(.3){b}|\hole && Y\ar[dr]^p\\
    & Z\ar[rr]^a && W }
     \]
    where $a,b,a',b'\in \dC(\epsilon_k)$, $p,q,p',q'\in
    \dC(\epsilon_{k'})$, $x,y,z,w\in \dC(\epsilon_{k''})$, we have
    \[\sigma_{\one_W,I,b'}\circ \sigma_{p,J',\one_{X'}}\circ \sigma_{\one_W,K,x}
    \sim \sigma_{w,K',\one_{X'}}\circ \sigma_{\one_W,J,q'}\circ \sigma_{a,I',\one_{X'}},\]
    where $I,I',J,J',K,K'$ are respectively the right, left, front, back,
    bottom, top faces of the cube.
  \end{enumerate}
  A $2$-cell $f\Rightarrow g$ of $\cT_n\dC$ is an equivalence class of pre-$2$-cells
  under this system of equivalence relations. Composition of $2$-cells
  is given by composition of pre-$2$-cells.

We define the \emph{reduced $2$-category $\cT^\red_n \dC$ of paths in $\dC$}
as follows. Objects of $\cT^\red_n \dC$ are objects of $\dC$. A
\emph{morphism} $X\to Y$ of $\cT^\red_n \dC$ is an equivalence class of
paths under the equivalence condition generated by conditions (1) and (2)
(unit, composition) of Definition \ref{d.cT}. The definition of $2$-cells in
$\cT^\red_n\dC$ is similar to the above, with only atomic $2$-cells of type
(3) (square). The atomic $2$-cells of types (1) and (2) (creation and
deletion of unit, composition and decomposition) appearing in conditions
\ref{d.T3} through \ref{d.T8} are replaced by identity pre-$2$-cells. In
particular, conditions \ref{d.T3} through \ref{d.T6} become tautological.

We obtain functors $\cT_n,\cT^\red_n\colon n\FoldCat\to \twoCat$.
\end{Definition}

As the (last) referee pointed out, the definition of $\cT_n\dC$ above is a
presentation of the $2$-category in the terminology of Street \cite[page
547]{Street}. The presentation consists of the graph $\coprod_{1\le k\le
n}\dC(\epsilon_k)\rightrightarrows \Ob(\dC)$, the derivation scheme given by
the atomic $2$-cells with $f$ and $g$ being identities, and the relation on
$2$-cells given by conditions \ref{d.T03} through \ref{d.T09}. Note that the
underlying category of $\cT_n\dC$ is a free category, so that Lemma
\ref{l.PsFun2} applies.

For a $1$-fold category $\bC$, the $2$-category $\cT_1\bC$ is $\cT\bC$ in
Definition \ref{d.cT} and $\cT^\red_1\bC$ can be identified with $\bC$. Note
that $\cT_n\dC$ and $\cT^\red_n\dC$ do not depend on hypercubes of $\dC$ of
dimension $>3$. This is justified by a theorem of Gray that we will recall
soon.

\begin{Remark}
The functors $\cT_n$ and $\cT_n^\red$ are refinements of $\bT_n$, in the
sense that we have isomorphisms $\bT_n\cong \bO\cT_n\cong \bO \cT^\red_n$,
where $\bO\colon \twoCat\to \Cat$ is the functor in Construction \ref{c.cO}.
\end{Remark}

We define the \emph{reduction $2$-functor} $R\colon \cT_n\dC\to
\cT^\red_n\dC$\index{R@$R$} as follows. On objects we take the identity. To
a morphism, we associate its equivalence class. To the class of
$\sigma_{g,D,f}$, we associate the class of $\sigma_{g,D,f}$. To the class
of atomic $2$-cells of types (1) and (2), we associate identities. We have
the following analogue of Proposition \ref{p.TC}.

\begin{Proposition}\label{p.TredC}
The reduction $2$-functor $R \colon \cT_n\dC\to \cT^\red_n\dC$ is a
bi-equivalence.
\end{Proposition}

Despite the proposition, we need both $\cT_n$ and $\cT^\red_n$, the latter
for the $2$-categorical refinement of the adjunction in Remark \ref{r.adj},
and the former for applications to pseudo functors.

\begin{proof}
We construct a strictly unital pseudo functor $G\colon \cT^\red_n\dC\to
\cT_n\dC$ satisfying $RG=\id_{\cT^{\red}_n\dC}$ and a pseudo natural
isomorphism $\eta\colon \id_{\cT_n\dC} \Rightarrow GF$ carrying $X$ to
$\id_X$ as follows. Note that the length of the target of each atomic
$2$-cell of type $\theta$ (deletion of unit) or $\gamma$ (composition) is
one less than the source. Thus, for any path $f$ in $\dC$, there exists a
sequence of atomic $2$-cells $f\Rightarrow \dots \Rightarrow f^\red$ of
types $\theta$ and $\gamma$ such that $f^\red$ is not the source of any
$\theta$ or $\gamma$. We denote the composite $2$-cell in $\cT_n\dC$ by
$\eta(f)$, which does not depend on choices by conditions \ref{d.T4} (unit
and composition) and \ref{d.T6} (associativity) of Definition \ref{d.T}. For
an atomic $2$-cell $\alpha\colon f_1\Rightarrow f_2$ of type $\theta$ or
$\gamma$, we have $\eta(f_1)=\eta(f_2)\alpha$. In particular,
$f_1^\red=f_2^\red$. Thus $f^\red$ depends only on the equivalence class of
$f$, and can be characterized as the unique path of minimal length in the
class of $f$. For any morphism $g$ of $\cT_n^{\red}\dC$, we take $G(g)$ to
be the unique path of minimal length representing $g$. For a sequence of
morphisms $X\xrightarrow{g} Y\xrightarrow{h} Z$ in $\cT^\red_n \dC$, the
composition constraint of $G$ is given by $\eta(G(h)G(g))$. An atomic
$2$-cell $\beta\colon g_1\Rightarrow g_2$ in $\cT^\red_n \dC$ can be lifted
to an atomic $2$-cell $\alpha\colon f_1\Rightarrow f_2$ in $\cT_n\dC$. We
take $G(\beta)$ to be the composite $G(f_1)\xRightarrow{\eta({f_1})^{-1}}
g_1\xRightarrow{\alpha} g_2\xRightarrow{\eta({f_2})} G(f_2)$.
\end{proof}

Our next goal is to extend the functor $\dK^n\colon \Cat\to n\FoldCat$
(Example \ref{ex.rhos} \ref{ex.rho2}) to a functor $\dQ^n\colon \twoCat\to
n\FoldCat$. In the case $n=2$, for a $2$-category $\cC$, $\dQ^2\cC$ is the
\emph{double category of up-squares} of $\cC$ defined by Bastiani and
Ehresmann \cite[2.C.1, p.~272]{BEhr}. In the case $\cC=\cCat$, the double
category $\dQ^2\cCat$ was studied earlier by Ehresmann and called the double
category of quintets \cite{Ehr2}. Objects of $\dQ^2\cC$ are objects of
$\cC$, horizontal morphisms are morphisms of $\cC$, vertical morphisms are
morphisms of~$\cC$, and squares are up-squares in $\cC$, namely, diagrams in
$\cC$ of the form
  \begin{equation}\label{e.2}
  \xymatrix{X\ar[r]^j\ar[d]_q\drtwocell\omit{^\alpha} & Y \ar[d]^p\\
  Z\ar[r]^i & W.}
  \end{equation}
In general, for $\alpha$ in $\bI^n$, we define $\alpha$-hypercubes of
$\dQ^n\cC$ to be the set of commutative hypercubes of dimension $\Sigma
\alpha$ with $2$-cells in $\cC$. The directions of the $2$-cells are
determined by the natural order on the set $\{1,\dots,n\}$.

By a theorem of Gray \cite{Gray2}, a commutative hypercube of dimension $n$
with $2$-cells is encoded by the $2$-category
$\gamma_n=\cT^\red_n([1]^{\boxtimes n})$, despite the fact that no
hypercubes of dimension $>3$ appear in the definition of $\cT^\red_n$. More
precisely, the theorem says that $\gamma_n$ locally a partially ordered set,
and hence commutative in the sense that, given any pair $(f,g)$ of
morphisms, there exists at most one $2$-cell $f\Rightarrow g$.

\begin{Notation}[The $\alpha$-hypercube $\gamma_{\alpha}$]\label{e.gamma}\index{gzammaalpha@$\gamma_\alpha$}
Let $\alpha=(\alpha_1,\dots,\alpha_n)$ be an object of $\bI^n$ (each
$\alpha_i$ being either $[0]$ or $[1]$). We define the \emph{(commutative)
$\alpha$-hypercube $\gamma_\alpha$ with $2$-cells} to be
$\cT^\red_n(\alpha_1\boxtimes \dots \boxtimes \alpha_n)$.
\end{Notation}

Of course for the $2$-category $\gamma_{\alpha}$ is isomorphic to
$\gamma_{\Sigma \alpha}$, but the above notation makes it clear that
$\gamma_\alpha$ is functorial with respect to $\alpha$.

Let us recall in our notation the explicit description of the hypercube by
Street \cite[Section 6]{Street}. We have $\Ob(\gamma_\alpha)=\alpha_1\times
\dots \times \alpha_n$. Consider a path $f_1*\dots*f_{\# J}$ from the
initial vertex to the final vertex of a $J$-hypercube, each $f_k$ in
direction $i_k\in J$. The path is uniquely determined by the total order
$i_1<_f\dots<_fi_{\#J}$ on $J$ (or equivalently, by the permutation of $J$
that carries the natural order to $f$). The partial order on the set of such
paths is precisely the weak order on permutations. Let us give a description
in terms of total orders, which is more natural. For objects $a$ and $b$ of
$\gamma_\alpha$, the category of morphisms $\gamma_\alpha(a,b)$ is given by
the partially ordered set
  \[\begin{cases}
    \text{set of total orders on $J_{a,b}=\{1\le i\le n\mid a_i\neq b_i\}$} & \text{if $a\le b$,}\\
    \emptyset & \text{otherwise.}
  \end{cases}\]
For total orders $f$ and $g$ on $J_{a,b}$, there exists a $2$-cell
$f\Rightarrow g$ if and only if $\Inv(f)\subseteq \Inv(g)$. In this case
there exists a unique $2$-cell. Here $\Inv(f)$ denote the set of pairs
$(i,j)\in J_{a,b}^2$ for which $i<j$ but $i>_f j$ ($i$ greater than $j$
under the order $f$). For $f\colon a\to b$ and $g\colon b\to
  c$, the composition $g\circ f$ is the unique order on $J_{a,c}=J_{b,c}\coprod
  J_{a,b}$ extending $f$ and $g$ such that $i<_{g\circ f}j$ for all
  $i\in J_{b,c}$ and $j\in J_{a,b}$.

\begin{Remark}\label{r.straight}
Note that $\gamma_\alpha$ is also the locally full sub-$2$-category of
$\cT_n(\alpha_1\boxtimes \dots \boxtimes \alpha_n)$ spanned by paths that do
not contain $\id^k$. Let $S$ be the set of edges of the hypercube
$\gamma_\alpha$, namely the set of morphisms $f\colon a\to b$ of
$\gamma_\alpha$ such that $b=a+\epsilon_i$ for some $i$. Then the underlying
category of $\gamma_\alpha$ is free on the graph $S\rightrightarrows
\Ob(\gamma_\alpha)$, so that Lemma \ref{l.PsFun2} applies.
\end{Remark}

\begin{Definition}[The $n$-fold category $\dQ^n\cC$ of hypercubes in a $2$-category
$\cC$]\label{d.Q}\index{Q@$\dQ^n$, $\dQ_{\cA_1,\dots,\cA_n}\cC$} Let $\cC$
be a $2$-category. We define an $n$-fold category $\dQ^n \cC$ by
$(\dQ^n\cC)(\alpha)=\rtwoFun(\gamma_\alpha,\cC)$. For $\alpha,\alpha'\in
\bI^n$ satisfying $\alpha'=\alpha+\epsilon_i$, the composition $\circ^i$ in
direction $i$ is given by the map
  \[(\dQ^n\cC)(\alpha')\times_{(\dQ^n\cC)(\alpha)}(\dQ^n\cC)(\alpha')\cong
    \rtwoFun(\cB,\cC) \to (\dQ^n\cC) (\alpha')
    \]
induced by $d^2_1$, where $\cB=\cT^{\red}_n({\tilde \alpha}_1\boxtimes \dots
\boxtimes {\tilde \alpha}_n)$ ($\tilde \alpha_j$ as in Definition
\ref{d.phi*}).

We obtain a functor $\dQ^n\colon \twoCat\to n\FoldCat$.

Now let $\cA_1,\dots, \cA_n$ be locally full sub-$2$-categories of $\cC$,
each containing all objects of $\cC$. We denote by
$\dQ_{\cA_1,\dots,\cA_n}\cC$ the largest $n$-fold subcategory of $\dQ^n\cC$
such that
$(\dQ_{\cA_1,\dots,\cA_n}\cC)(\epsilon_i)=\rtwoFun(\gamma_{\epsilon_i},\cA_i)$.
In other words, $\alpha$-hypercubes of $\dQ_{\cA,\dots,\cA_n}\cC$ are
commutative hypercubes of dimension $\Sigma \alpha$ with $2$-cells in $\cC$,
in the directions $i$ satisfying $\alpha_i=1$, such that every edge in
direction $i$ is in $\cA_i$.
\end{Definition}

For a category $\bC$, we have $\dQ^n\bC\cong \dK^n\bC$. For $n=1$, the
$1$-fold category $\dQ^1\cC$ is the underlying category of $\cC$. By
definition, $\dQ^n\cC=\dQ_{\cC,\dots,\cC}\cC$. For the map $\iota_i\colon
\{1\}\to \{1,\dots,n\}$ with image of $i$, the category
$\iota_i^*\dQ_{\cA_1,\dots,\cA_n}\cC$ is isomorphic to the underlying
category of $\cA_i$. Moreover, for $i\neq j$, we have an isomorphism of
$2$-categories $\cH_{i,j}(\dQ_{\cA_1,\dots,\cA_n}\cC)\cong \cA_i$.

\begin{Example}\label{ex.up}
In the case $n=2$, let $\cA$ and $\cB$ be locally full sub-$2$-categories of
$\cC$, each containing all objects of $\cC$. The objects of the double
category $\dQ_{\cA,\cB}\cC$ are the objects of $\cC$. Horizontal morphisms
are morphisms of $\cA$, vertical morphisms are morphisms of~$\cB$, and
squares are \emph{$(\cA,\cB)$-squares in $\cC$}, namely, diagrams in $\cC$
of the form \eqref{e.2} where the horizontal morphisms $i,j$ are in $\cA$
and the vertical morphisms $p,q$ are in $\cB$. We have isomorphisms
$\cH(\dQ_{\cA,\cB}\cC)\cong \cA$ and $\cV(\dQ_{\cA,\cB}\cC)\cong \cB$.
\end{Example}

\begin{Construction}[Adjunction between $\cT^\red_n$ and $\dQ^n$]
Let $\dC$ be an $n$-fold category. We construct an $n$-fold functor
\begin{equation}\label{e.QT}
\dC\to \dQ^n\cT^\red_n\dC
\end{equation}
as follows. To a hypercube $C\in \dC(\alpha)$, corresponding to an $n$-fold
functor $\alpha_1\boxtimes \dots \boxtimes \alpha_n\to \dC$, we associate
the $2$-functor $\gamma_\alpha\to \cT^\red_n\dC$ obtained by applying
$\cT^\red_n$.

Let $\cD$ be a $2$-category. We construct a $2$-functor
\begin{equation}\label{e.TQ}\index{Enred@$E^\red_n$}
E^\red_n\colon \cT^\red_n\dQ^n\cD\to \cD
\end{equation}
as follows. On objects we take the identity. To a path $f_m*\dots* f_1$,
where $f_i\in \coprod_{1\le k\le n} (\dQ^n \cD)(\epsilon_k)$, we associate
the composite $f_m\dotsm f_1$. We take $F(\sigma_{g,D,f})$ to be the
$2$-cell induced by $D$.
\end{Construction}

As in Remark \ref{r.adj}, it is easy to check the following.

\begin{Proposition}\label{p.adj}
The above construction exhibits $\cT^\red_n\colon n\FoldCat\to \twoCat$ as a
left adjoint of $\dQ^n\colon \twoCat\to n\FoldCat$. In other words, we get a
bijection
\[\rtwoFun(\cT^\red_n\dC,\cD)\cong \FFun{n}(\dC,\dQ^n\cD),
\]
functorial in $\dC$ and $\cD$.
\end{Proposition}

For $n=1$, this adjunction reduces to the one mentioned at the end of
Construction \ref{c.cO}.

\begin{Remark}\label{r.Gray}
In the case where $\dC=\bC_1\boxtimes \dots \boxtimes\bC_n$ is the external
product of categories $\bC_1, \dots, \bC_n$, this adjunction was studied by
Gray \cite[Section I.4]{Gray}. In this case, the $2$-category
$\cT^\red_n(\bC_1\boxtimes \dots \boxtimes \bC_n)$ is Gray's tensor product
$\bC_1\otimes \dots \otimes \bC_n$,\footnote{The identification with Gray's
tensor product follows from the definition for $n=2$ and from Proposition
\ref{p.adj} for $n\ge 3$.} and $n$-fold functors $\bC_1\boxtimes \dots
\boxtimes\bC_n\to \dQ^n\cD$ are precisely Gray's ``quasi-functors'' of
several variables $\bC_1\times \dots \times \bC_n \to \cD$. We emphasize
that this is only a special case of Gray's theory, which defines more
generally tensor products of $2$-categories in relation to the $2$-category
of $2$-functors, \emph{lax} natural transformations, and modifications.
\end{Remark}

The functor $\dK^n$ is fully faithful and we have $\bT_n\dK^n\bD\cong \bD$.
For $n\ge 2$, the $2$-functor $E_n^\red$ is a bijection on objects and a
surjection on morphisms and $2$-cells, so that the functor $\dQ^n$ is
faithful. On the other hand, $E_n^\red$ is seldom an isomorphism and $\dQ^n$
is never fully faithful. Nonetheless, we have the following $2$-categorical
analogue.

\begin{Proposition}\label{p.Fequiv}
Let $\cD$ be a $2$-category. For $n\ge 2$, the $2$-functor $E^\red_n\colon
\cT^\red_n\dQ^n\cD\to \cD$ is a
  bi-equivalence. More generally, if $\cA_1,\dots,\cA_n$ are locally full
  sub-$2$-categories of $\cD$ each of which contains all objects of $\cD$ and
  $\cA_i=\cD$ for some $1\le i\le n$, then the $2$-functor $\cT^\red_n\dQ_{\cA_1,\dots,\cA_n}\cD\to
  \cD$ induced by $E^\red_n$ is a bi-equivalence.
\end{Proposition}

For $n=1$, $E_1^\red$ is the inclusion of the underlying category of $\cD$
into $\cD$, which is seldom a bi-equivalence. The functor $\dQ^1$ is not
faithful.

\begin{proof}
We begin with a construction for an $n$-fold category $\dC$. For $1\le i,j
\le n$ and $i\neq j$, we define a $2$-functor
\begin{equation}\label{e.HtoT}
  \cH_{i,j}\dC\to \cT_n^\red\dC
\end{equation}
as follows. On objects, we take the identity. To $f\in \dC(\epsilon_i)$, we
associate $f$. If $i<j$ (resp.\ $i>j$), to any $2$-cell $\alpha\colon
f\Rightarrow g$ in $\cH_{i,j}\dC$, where $f,g\colon X\to Y$,
  we associate the $2$-cell
\begin{equation}\label{e.Fequiv2}
  f=f*\one^j_X \xRightarrow{\sigma_{\one_Y,D,\one_X}}
  \one^j_Y*g = g\quad \text{(resp.\ $f=\one^j_Y*f
  \xRightarrow{\sigma_{\one_Y,D,\one_X}}
  g*\one^j_X = g$)},
\end{equation}
  where $D\in \dC(\epsilon_i+\epsilon_j)$ is the square induced by $\alpha$ of the boundary \eqref{e.sqh} (resp.\
  \eqref{e.sqv}). Taking $\dC=\dQ^n\cD$, we get a $2$-functor
\begin{equation}\label{e.TQ2}
G_i^\red\colon \cD\cong \cH_{i,j}\dQ^n\cD \to \cT_n^\red\dQ^n\cD.
\end{equation}
Note that in this case the $2$-cell \eqref{e.Fequiv2} does not depend on the
choice of $j$ by conditions \ref{d.T7} (unit square) and \ref{d.T9} (cube)
of Definition \ref{d.T}.

We still use $E^\red_n$ and $G^\red_i$ to denote the induced $2$-functors
between $\cT^\red_n\dQ_{\cA_1,\dots,\cA_n}\cD$ and $\cD$. We have $E^\red
G^\red_i=\one_\cD$. We define a pseudo natural isomorphism $\epsilon\colon
G^\red_iE^\red\Rightarrow \one_{\cT_n\dQ_{\cA_1,\dots,\cA_n}\cD}$ as
follows.
  For any object $X$ of $\cD$, we take $\epsilon(X)=\one_X$. For any
  morphism $f\colon X\to Y$ of $\cD$, we denote by $f^k\in (\dQ^n\cD)(\epsilon_k)$ the image of $f$. We take
  $\epsilon(f^i)=\one_{f^i}$, and take $\epsilon(f^k)$, $i<k$ (resp.\ $i>k$) to be the
  composition
  \[f^i=
  f^i*\one^k_Y \xRightarrow{\sigma_{\one_Y,D,\one_X}}
  f^k*\one^i_X = f^k \quad\text{(resp.\ $f^i=
  \one^k_Y*f^i \xRightarrow{\sigma_{\one_Y,E,\one_X}}
  \one^i_X*f^k = f^k$)},\]
  where $D,E\in (\dQ^n \cD)(\epsilon_i+\epsilon_k)$ are
  the squares
  \[\xymatrix{X\ar@{=}[r]\ar@{=}[d]\ar@{}[rd]|D & X\ar[d]^{f} & X\ar[r]^f\ar[d]_f\ar@{}[rd]|E &Y\ar@{=}[d]\\
  X\ar[r]^f & Y & Y\ar@{=}[r] & Y}\]
  induced by $\one_f$. It is straightforward to check that this definition is compatible with $2$-cells of $\cT_n^\red\dQ_{\cA_1,\dots,\cA_n}\cD$.
\end{proof}

Since our goal is to study pseudo functors between $2$-categories, it would
be nice to have a version of the above adjunction for pseudo functors.
Between double categories, a notion of double pseudo functors, weakly
compatible with compositions in both directions, was defined by Shulman
\cite[Definition 6.1]{Shulman}\footnote{This is more general than the notion
of pseudo double functors by Fiore \cite[Definition 6.4]{Fiore}, which are
weakly compatible with horizontal composition and strictly compatible with
vertical composition.}. For $n\ge 3$, however, to define $n$-fold pseudo
functors between $n$-fold categories, one would need, for each direction
$i$, to choose an auxiliary direction $j\neq i$ that provides the
pseudoness. Fortunately, we will only need $n$-fold pseudo functors from an
$n$-fold category $\dC$ to $\dQ^n \cD$, where $\cD$ is a $2$-category. In
this case, no choice of directions is necessary. Taking our clue from
Propositions \ref{l.PsFun} and \ref{p.adj}, we regard
$\twoFunps(\cT_n\dC,\cD)$ as the $2$-category of \emph{$n$-fold pseudo
functors} from $\dC$ to $\dQ^n \cD$.

\begin{Remark}\label{d.mixedPsFun}
As in Proposition \ref{p.adj}, it is possible to describe
$\twoFunps(\cT_n\dC,\cD)$ in terms of $\dC$ and $\dQ^n\cD$, as follows.
\emph{Objects} are pairs $(F,(F_i)_{1 \le i \le n})$ consisting of a natural
transformation $F\colon \dC\to \dQ^n\cD$, where $\dC$ and $\dQ^n\cD$ are
viewed as functors $(\bI_+^{\op})^n\to \Set$, and pseudo functors $F_i\colon
\iota_i^*\dC\to \cD$ extending $F_{\bzero}$ and $F_{\epsilon_i}$. Here
$\bI_+$ is the subcategory of $\bI$ (Notation \ref{n.cI}) spanned by the
strictly increasing maps. For all $1\le k, k'\le n$ ($k\neq k'$), and $D\in
\dC(\epsilon_k+\epsilon_{k'})$, we denote by $G_D$ the $2$-cell in
      $F_{\epsilon_k+\epsilon_{k'}}(D)$. The pair $(F,(F_i))$ is subject to the following
conditions:
\begin{itemize}
  \item[(a)] For $f\colon X\to Y$ belonging to $\dC(\epsilon_k)$,
      the following triangle (resp.\ with the horizontal arrow
      reversed)
      \[\xymatrix{F_k(f)\ar@=[rd] \ar@=[d] \\
      F_k(f) F_{k'}(\one^{k'}_X)\ar@=[r]^{G_D} & F_{k'}(\one^{k'}_Y)F_k(f) }
      \]
      commutes if $k<k'$ (resp.\ $k'<k$), where $D=s(f)\in
      \dC(\epsilon_k+\epsilon_{k'})$ is the identity square
      \eqref{e.vidsquare}.

  \item[(b)] For $D''=D'\circ^{k} D$ in
      $\dC(\epsilon_k+\epsilon_{k'})$ of the form
      \eqref{e.hsquare}, the following pentagon (resp.\ with the
      horizontal arrows reversed)
\[\xymatrix{F_k(i')F_{k}(i)F_{k'}(p_1)\ar@=[r]^{G_D}\ar@=[d]
& F_k(i')F_{k'}(p_2)F_k(j) \ar@=[r]^{G_{D'}}
& F_{k'}(p_3)F_k(j')F_k(j)\ar@=[d]\\
F_k(i'i)F_{k'}(p_1)\ar@=[rr]^{G_{D''}}
&& F_{k'}(p_3)F_k(j'j)}
\]
      commutes if $k<k'$ (resp.\ $k'<k$).
\end{itemize}

A \emph{morphism} $\mu\colon F\to F'$ is a collection $(\mu_i)_{1\le i\le
n}$ of morphisms (i.e.\ pseudo natural transformations) $\mu_i\colon F_i\to
F'_i$ of $\PsFun(\iota_i^*\dC,\cD)$ satisfying $\lvert \mu_1\rvert =\dots
=\lvert \mu_n\rvert$ (Notation \ref{d.bar} \ref{d.bar2}) and such that for
every square \eqref{e.dCsquare} in $\dC$, the cube
\[    \xymatrix{FX\ar[dd]_{\alpha X}\ar[dr]^{F_{k'}(q)}\ar[rr]^{F_k(j)} && FY\ar[dd]^(.3){\alpha Y}|\hole \ar[dr]^{F_{k'}(p)}\\
    & FZ\ar[dd]^(.3){\alpha Z}\ar[rr]^(.3){F_k(i)} && FW\ar[dd]^{\alpha W} \\
    F'X\ar[dr]_{F'_{k'}(q)}\ar[rr]^(.3){F'_k(j)}|\hole && F'Y\ar[dr]_{F'_{k'}(p)}\\
    & F'Z\ar[rr]_{F'_k(i)} && F'W }
\]
is commutative. Here the top, bottom, front, back, right, and left faces are
respectively given by $F_{\epsilon_k+\epsilon_{k'}}(D)$,
$F'_{\epsilon_k+\epsilon_{k'}}(D)$, $\mu_k(i)$, $\mu_k(j)$, $\mu_{k'}(p)$,
and $\mu_{k'}(q)$.

A \emph{$2$-cell} $\Xi\colon \mu\Rightarrow \nu$
  consists of a collection $(\Xi_i)_{1\le i\le n}$ of $2$-cells (i.e.\ modifications)
  $\Xi_i\colon \mu_i\Rightarrow  \nu_i$ of $\PsFun(\iota_i^*\dC,\cD)$ such that $\lvert \Xi_1 \rvert = \dots=\lvert
  \Xi_n \rvert$. In other words, a $2$-cell $\mu\Rightarrow \nu$ of $\twoFunps(\dC,\cD)$ is a
  function from $\Ob(\dC)$ to
  the set of $2$-cells of $\cD$ which is a modification $\mu_i\Rightarrow \nu_i$ for all $1\le i\le n$ at the same time.

\emph{Composition} is given by composition in $\cD$.

The $2$-category $\twoFunps(\cL\cT_n\dC,\cD)$ can be identified with the
$2$-full sub-$2$-category spanned by objects for which all the $G_D$ are
invertible, where $\cL\colon \twoCat\to \twooneCat$ is the functor in
Construction \ref{c.twoone}.
\end{Remark}

\begin{Construction}[The $2$-functor $E_n\colon \cT_n\dQ^n\cC\to
\cT\cC$]\label{c.E}\index{En@$E_n$} Let $\cC$ be a $2$-category. We
construct a variant
\begin{equation}\label{e.ET}
E_n\colon \cT_n\dQ^n\cC\to \cT\cC
\end{equation}
of the adjunction $2$-functor \eqref{e.TQ} as follows. On objects we take
the identity. To a path $f_m*\dots*f_1$ in $\dQ^n\cC$, we associate the path
$f_m*\dots*f_1$ in $\cC$. To the $2$-cells $\iota^k$ and $\theta^k$
(creation and deletion of unit), we associate $\iota$ and $\theta$. To the
$2$-cells $\gamma$ and $\delta$ (composition and decomposition), we
associate $\gamma$ and $\delta$. Finally, to the $2$-cell $\sigma_{g,D,f}$
(square), we associate the composite
\[(E_ng)*i*q*(E_nf) \xRightarrow{\gamma_{E_ng,i,q,E_nf}} (E_ng)*iq*(E_nf) \xRightarrow{\sigma_{E_ng,\alpha,E_nf}} E_ng*pj*E_nf\xRightarrow{\delta_{E_ng,p,j,E_nf}} E_ng*p*j*E_nf,\]
where $\alpha\colon iq\Rightarrow pj$ is the $2$-cell in $D$. We have a
commutative square
\[\xymatrix{\cT_n\dQ^n\cC\ar[r]^-{E_n}\ar[d]_{R} &\cT\cC\ar[d]^{R_\cC}\\
\cT^\red_n\dQ^n\cC\ar[r]^-{E^\red_n} & \cC,}
\]
where the vertical arrows are the reduction $2$-functors and are
bi-equivalences by Propositions \ref{p.TC} and \ref{p.TredC}. For $n\ge 2$,
the lower horizontal arrow is a bi-equivalence by Proposition
\ref{p.Fequiv}, so that $E_n\colon \cT_n\dQ^n\cC\to \cT\cC$ is a
bi-equivalence as well.

This construction is related to pseudo functors as follows. For any
$2$-category $\cD$, the $2$-functor $E_n$ induces a
$\cD^{\Ob(\cC)}$-$2$-functor
\[ \PsFun(\cC,\cD)\cong \twoFunps(\cT\cC,\cD)\to \twoFunps(\cT_n\dQ^n\cC,\cD),\]
which is a $\cD^{\Ob(\cC)}$-$2$-equivalence for $n\ge 2$ by Proposition
\ref{p.equiv} and Lemma \ref{l.PsFun2}.

For reference in the next section, for $n\ge 2$, we define $2$-functors
\begin{equation}\label{e.Gi}
G_i\colon \cT\cC\to \cT_n\dQ^n\cC,
\end{equation}
pseudo inverses of $E_n$, as follows. The definition is similar to
\eqref{e.TQ2}. On objects we take the identity. To a path $f_m*\dots*f_1$ in
$\cC$, we associate the path $f_m^i*\dots*f_1^i$ in $\dQ^n\cC$, where
$f_j^i\in (\dQ^n\cC)(\epsilon_i)$ is $f_j$ considered as a morphism in
direction $i$. We take $G_i(\iota_{g,f})=\iota_{G_ig,G_if}^i$ and
$G_i(\gamma_{g,h',h,f})=\gamma_{G_i g, h'^i ,h^i,G_if}$. Choose $1\le j\le
n$ with $j\neq i$. We take $G_i(\sigma_{g,\alpha,f})$ to be the $2$-cell
\[G_i(g)*h^i*G_i(f)\xRightarrow{\iota^j}G_i(g)*h^i*\id^j_Y*G_i(f)\xRightarrow{\sigma_{G_ig,D,G_if}}G_i(g)*\id^j_Z*h'^i*G_i(f)\xRightarrow{\theta^j}G_i(g)*h'^i*G_i(f)\]
for $i<j$ and the $2$-cell
\[G_i(g)*h^i*G_i(f)\xRightarrow{\iota^j}G_i(g)*\id^j_Y*h^i*G_i(f)\xRightarrow{\sigma_{G_ig,D,G_if}}G_i(g)*h'^i*\id^j_Z*G_i(f)\xRightarrow{\theta^j}G_i(g)*h'^i*G_i(f)\]
for $i>j$. Here $D$ denotes the square induced by $\alpha$. The $2$-cell
does not depend on the choice of $j$.
\end{Construction}

\begin{Construction}[The descent $2$-functor $Q^n_\cD$]\index{Qn@$Q^n$}
Let $\cC$ be a $(2,1)$-category. Consider the $2$-functor
\[
\cL E_n\colon \cL\cT_n\dQ^n\cC\to \cL\cT\cC\cong \cT\cC
\]
obtained from \eqref{e.ET} by applying the functor $\cL$ (Construction
\ref{c.twoone}). This induces, for every $2$-category $\cD$, a
$\cD^{\Ob(\cC)}$-$2$-functor
\begin{equation}\label{e.preQ}
 \PsFun(\cC,\cD)\cong \twoFunps(\cT\cC,\cD)\to \twoFunps(\cL\cT_n\dQ^n\cC,\cD),
\end{equation}
For $n\ge 2$, the $2$-functor $\cL E_n$ is a bi-equivalence and
\eqref{e.preQ} is a $\cD^{\Ob(\cC)}$-$2$-equivalence.

One main goal of this article is to study more generally, for locally full
sub-$2$-categories $\cA_1,\dots,\cA_n\subseteq \cC$ each of which contains
all objects of $\cC$, the \emph{descent $2$-functor}
\begin{equation}\label{e.Q}
Q^n= Q^n_{\cD}\colon \PsFun(\cC,\cD)\cong \twoFunps(\cT\cC,\cD)\to \twoFunps(\cL\cT_n\dQ_{\cA_1,\dots,\cA_n}\cC,\cD)
\end{equation}
induced by the $2$-functor
\[
\cL E_n\colon \cL\cT_n\dQ_{\cA_1,\dots,\cA_n}\cC\to \cT\cC
\]
obtained by restriction via the inclusion
$\dQ_{\cA_1,\dots,\cA_n}\cC\subseteq \dQ^n\cC$.
\end{Construction}

\begin{Definition}[$2$-Category of gluing data]\label{d.gd}\index{GD@$\GD_{\cA,\cB}(\cC,\cD)$}\index{GD@$\GD_{\cA_1,\dots,\cA_n}(\cC,\cD)$}
For $n\ge 2$, we call
\[\GD_{\cA_1,\dots,\cA_n}(\cC,\cD)\coloneqq \twoFunps(\cL\cT_n\dQ_{\cA_1,\dots,\cA_n}
\cC, \cD)
\]
the $2$-category of \emph{gluing data} from $\cC$ to~$\cD$, relative to
$\cA_1,\dots,\cA_n$. For $n=1$, we put
$\GD_{\cA_1}(\cC,\cD)\coloneqq\PsFun(\cA_1,\cD)$.
\end{Definition}

We will give more explicit descriptions of the $2$-category of gluing data
later (see Remark \ref{r.gd} for the case $n=2$ and Remarks \ref{r.gdn},
\ref{r.gdna} for the general case), which also justifies our convention for
the case $n=1$. In order to treat the cases $n=1$ and $n\ge 2$ uniformly, we
adopt the following notation for $n\ge
2$:\index{TnQ@$\overline{\cT_n\dQ}_{\cA_1,\dots,\cA_n}\cC$,
$\overline{\cT_n^\red\dQ}_{\cA_1,\dots,\cA_n}\cC$}
\begin{gather*}
\overline{\cT_n\dQ}_{\cA_1,\dots,\cA_n}\cC\coloneqq\cT_n\dQ_{\cA_1,\dots,\cA_n}\cC,\quad \overline{\cT_1\dQ}_{\cA_1}\cC\coloneqq\cT\cA_1,\\
\overline{\cT_n^\red\dQ}_{\cA_1,\dots,\cA_n}\cC\coloneqq\cT^\red_n\dQ_{\cA_1,\dots,\cA_n}\cC,\quad \overline{\cT_1^\red\dQ}_{\cA_1}\cC\coloneqq\cA_1.
\end{gather*}

By Proposition \ref{p.equiv} and Lemma \ref{l.PsFun2}, we have the following
criterion.

\begin{Proposition}\label{p.GD}
Let $\cL E^\red_n\colon \cL\cT^\red_n\dQ_{\cA_1,\dots,\cA_n}\cC\to \cC$ be
the $2$-functor induced by \eqref{e.TQ}. Consider the descent $2$-functor
$Q^n_\cD$ \eqref{e.Q}.
\begin{enumerate}
\item \label{p.GD1} If $ \cL E^\red_n$ is locally essentially surjective,
    then $Q^n_\cD$ is $2$-faithful for every $2$-category $\cD$.
\item \label{p.GD2} The following conditions are equivalent:
\begin{enumerate}
\item The $2$-functor $\cL E^\red_n$ is a bi-equivalence.
\item The $2$-functor $Q^n_\cD$ is a bi-equivalence for every
    $2$-category $\cD$.
\item The $2$-functor $Q^n_\cD$ is a $\cD^{\Ob(\cC)}$-$2$-equivalence
    for every $2$-category $\cD$.
\end{enumerate}
\end{enumerate}
\end{Proposition}

Combining with Proposition \ref{p.Fequiv}, we obtain the following.

\begin{Corollary}\label{c.QD}
Assume $n\ge 2$ and $\cA_i=\cC$ for some $1\le i\le n$. Then the descent
$2$-functor
\[Q^n_\cD\colon \PsFun(\cC,\cD)\to \GD_{\cA_1,\dots,\cA_n}(\cC,\cD)\]
is a $\cD^{\Ob(\cC)}$-$2$-equivalence for every $2$-category $\cD$.
\end{Corollary}

The corollary does not produce \emph{new} pseudo functors from $\cC$ to
$\cD$, since a pseudo functor from $\cC$ to $\cD$ is part of the gluing
datum by the assumption $\cA_i=\cC$. In Section \ref{s.two}, we will give a
more substantial application of Proposition \ref{p.GD} for the gluing of two
pseudo functors (case $n=2$) defined on \emph{proper} sub-$2$-categories.
For the gluing of more than two functors (case $n>2$) defined on proper
sub-$2$-categories, we will do so indirectly, by relating the gluing problem
for $n$ pseudo functors to the gluing problem for $n-1$ pseudo functors via
the constructions of the following section.

\section{Functoriality of gluing data with respect to the index set $\{1,\dots,n\}$}\label{s.3}

In this section, we record some functorial properties, with respect to the
index set $\{1,\dots, n\}$, of gluing data and of the operations $\cT_n$ and
$\dQ^n$ defining them. Such properties will play an essential role in
relating gluing problems for different numbers of pseudo functors in
Section~\ref{s.finiteglue}. Section \ref{s.two} does not depend on the
construction of this section.

We start with the case where the map $\phi\colon \{1,\dots, m\} \to
\{1,\dots,n\}$ between index sets is nondecreasing.

\begin{Construction}[Functoriality of $\cT_n$ for $\phi$ nondecreasing]\label{c.Tfun}
  Let $\phi\colon \{1,\dots, m\} \to \{1,\dots,n\}$ be a nondecreasing
  map and let $\dC$ be an $n$-fold category. We define $2$-functors
  \begin{equation}\label{e.Tphi}
  \cT_m \phi^* \dC\to \cT_n \dC,\quad \cT^\red_m \phi^* \dC\to \cT^\red_n \dC
  \end{equation}
as follows. To an object $X$, we associate $X$. To a morphism $f^k$, where
$f\in (\phi^*\dC)(\epsilon_k)= \dC(\epsilon_{\phi(k)})$, we associate
$f^{\phi(k)}$. To an atomic $2$-cell in $\cT_m \phi^*\dC$ or $\cT_m^\red
\phi^*\dC$ of type $\iota$, $\theta$ (unit), $\gamma$, $\delta$
((de)composition), or $\sigma_{g,D,f}$ with $D\in
(\phi^*\dC)(\epsilon_k+\epsilon_{k'})=
\dC(\epsilon_{\phi(k)}+\epsilon_{\phi(k')})$ satisfying $\phi(k)\neq
\phi(k')$, we associate the corresponding atomic $2$-cell in $\cT_n \dC$ or
$\cT_n^\red \dC$. An element $D\in \dC(\epsilon_k+\epsilon_{k'})$ with
$k<k'$ and $\phi(k)=\phi(k')=l$ corresponds to a commutative square in
$\iota_l^*\dC$, and we take $\cT_n(\sigma_{g,D,f})$ to be the composite
\[g*i*q*f\xRightarrow{\gamma_{g,i,q,f}} g*iq*f=g*pj*f \xRightarrow{\delta_{g,p,j,f}} g*p*j*f\]
and take $\cT^\red_n(\sigma_{g,D,f})$ to be the identity. The two
$2$-functors are compatible with
  reduction $\cT_n\to \cT^\red_n$.

For a sequence of nondecreasing maps
\begin{equation}\label{e.lmn}
  \{1,\dots, l\} \xrightarrow{\psi} \{1,\dots, m\} \xrightarrow{\phi}
  \{1,\dots,n\},
\end{equation}
the composition
  \[\cT_l(\phi\psi)^*\dC\cong\cT_l{\psi}^*\phi^*\dC\to \cT_m\phi^*\dC\to
  \cT_n\dC\]
equals the $2$-functor induced by $\phi\psi$. The same holds for
$\cT^\red_n$.

Consider the natural transformation induced by \eqref{e.Tphi}
\begin{equation}\label{e.Tphis}
\cT_m^\red\to \cT_m^\red\phi^*\phi_! \to \cT^\red_n\phi_!,
\end{equation}
which is a natural isomorphism when $\phi$ is injective (hence strictly
increasing).
\end{Construction}

\begin{Construction}[Functoriality of $\dQ^n$ for $\phi$ nondecreasing]\label{c.Qfun}
By adjunction \eqref{e.Tphis} induces a natural transformation
\begin{equation}\label{e.Qfun}
\phi^*\dQ^n\cC\to \dQ^m\cC.
\end{equation}
When $\phi$ is injective (hence strictly increasing), \eqref{e.Qfun} is a
natural isomorphism, and we have
\[
\phi^*\dQ_{\cA_1,\dots, \cA_n} \cC\cong \dQ_{\cA_{\phi(1)},\dots ,\cA_{\phi(m)}}\cC
\]
for any $2$-category $\cC$ and locally full sub-$2$-categories
$\cA_1,\dots,\cA_n$ of $\cC$ each containing all objects of $\cC$.

For a sequence of nondecreasing maps \eqref{e.lmn}, the composition
\[(\phi\psi)^*\dQ^n \cong{\psi}^*\phi^*\dQ^n \to {\psi}^*\dQ^m\to
  \dQ^l\]
equals the natural transformation induced by $\phi\psi$.
\end{Construction}

\begin{Construction}[Functoriality of $\overline{\cT_n\dQ^n}$ for $\phi$
nondecreasing]\label{c.TQfun}\index{Ephi@$E_\phi$, $E^\red_\phi$} Let
$\phi\colon \{1,\dots, m\}\to \{1,\dots,n\}$ be a nondecreasing map and let
$\cC$ be a $2$-category. We define $2$-functors
\[E_\phi\colon \overline{\cT_m\dQ^m}\cC\to \overline{\cT_n\dQ^n}\cC, \quad E_\phi^\red\colon \overline{\cT_m^\red\dQ^m}\cC\to\overline{\cT_n^\red\dQ^m}\cC,\]
compatible with reduction, as follows.

For $m=n=1$, we take $E_\phi\colon \cT\cC\to \cT\cC$ and $E_\phi^\red\colon
\cC\to \cC$ to be the identities. For $m=1$ and $n\ge 2$, we take
\[E_\phi\colon \overline{\cT_1\dQ^1}\cC=\cT\cC\to \cT_n\dQ^n\cC,\quad E^\red_\phi\colon \overline{\cT^\red_1\dQ^1}\cC=\cC\to \cT^\red_n\dQ^n\cC\]
to be the $2$-functors $G_{\phi(1)}$ \eqref{e.Gi} and $G^\red_{\phi(1)}$
\eqref{e.TQ2}.  For $m\ge 2$ and $n=1$, we take
\[E_\phi=E_m\colon \cT_m\dQ^m\cC\to \cT\cC,\quad E_\phi^\red=E_m^\red\colon
\cT_m^\red\dQ^m\cC\to \cC.
\]
For $m,n\ge 2$, we construct $E_\phi$ as follows. To an object $X$, we
associate $X$. To a morphism $f_a*\dots *f_1$, where $f_i\in \coprod_{1\le
k\le m}(\dQ^m \cC)(\epsilon_{k})$ for $1\le i\le a$, we associate
$f_a*\dots* f_1$. To the atomic $2$-cell $\sigma_{g,D,f}$ of
$\cT_m\dQ^m\cC$, where $D\in \dQ^m\cC(\epsilon_k+\epsilon_{k'})$ with
$\phi(k)=\phi(k')=l$, $k<k'$, we associate the composition
\[g*i*q*f\xRightarrow{\gamma_{g,i,q,f}} g*iq*f\xRightarrow{g*\alpha*f} g*pj*f \xRightarrow{\delta_{g,p,j,f}} g*p*j*f,\]
where $\alpha\colon iq\Rightarrow pj$ is the image of the $2$-cell in $D$
under the $2$-functor $G_l\colon \cT\cC\to \cT_n\dQ^n\cC$ (the assumption
$n\ge 2$ is used in the definition of $G_l$). To other atomic $2$-cells in
$\cT_m \dQ^m\cC$, we associate the corresponding atomic $2$-cells in $\cT_n
\dQ^n\cC$. The $2$-functor $E^\red_\phi$ is defined similarly, with $G_l$
replaced by $G_l^\red$.

If $\cA_1,\dots,\cA_m,\cB_1,\dots,\cB_n$ are locally full sub-$2$-categories
of $\cC$, each containing all objects of $\cC$, such that $\cA_i\subseteq
\cB_{\phi(i)}$ for all $1\le i\le m$, then $E_{\phi}$ restricts to a
$2$-functor
\[\overline{\cT_m\dQ}_{\cA_1,\dots,\cA_m}\cC \to \overline{\cT_n\dQ}_{\cB_1,\dots,\cB_n} \cC\]
and similarly for $E^\red_\phi$.

For a sequence of nondecreasing maps $\{1,\dots, l\} \xrightarrow{\psi}
\{1,\dots, m\} \xrightarrow{\phi} \{1,\dots,n\}$, the composition
\[\overline{\cT_l\dQ^l}\cC \xrightarrow{E_{
\psi}} \overline{\cT_m\dQ^m}\cC \xrightarrow{E_{\phi}} \overline{\cT_n\dQ^n}\cC
\]
equals $E_{\phi\psi}$. Similarly $E_{\phi\psi}^\red=E_\phi^\red
E_\psi^\red$.
\end{Construction}

The following is a generalization of Proposition \ref{p.Fequiv}.

\begin{Proposition}\label{p.Fphiequiv}
Let $\phi\colon \{1,\dots, m\}\to \{1,\dots,n\}$ be a nondecreasing map with
$n\ge 2$ and let $\psi$ be a section of $\phi$. Let $\cA_1,\dots,\cA_m$ be
locally full sub-$2$-categories of $\cC$, each containing all objects of
$\cC$, such that $\cA_i\subseteq \cA_{\psi\phi(i)}$ for all $1\le i \le m$.
Then
\[
E_\phi\colon\overline{\cT_m\dQ}_{\cA_1,\dots,\cA_m} \cC\to \overline{\cT_n\dQ}_{\cA_{\psi(1)},\dots,\cA_{\psi(n)}}\cC,\quad
E_{\psi}\colon \overline{\cT_n\dQ}_{\cA_{\psi(1)},\dots,\cA_{\psi(n)}}\cC \to \overline{\cT_m\dQ}_{\cA_1,\dots,\cA_m} \cC
\]
are pseudo inverses of each other. The same holds for $E^\red_\phi$ and
$E^\red_{\psi}$.
\end{Proposition}

\begin{proof}
We have $E_{\phi}E_{\psi}=E_{\one}=\one$. Furthermore, one can construct a
pseudo natural isomorphism $E_{\psi}E_{\phi}\Rightarrow \one$ as in the
proof of Proposition \ref{p.Fequiv}.
\end{proof}

\begin{Remark}\label{r.t}
Let $t\colon \{1,\dots, n\} \to \{1,\dots,n\}$ be the map sending $i$ to
$n+1-i$. For an $n$-fold category $\dC$ and $2$-categories $\cC$ and $\cD$,
we have isomorphisms
\[
\cT_n (t^* \dC) \cong (\cT_n\dC)^{\twoop},\quad\cT_n^\red (t^* \dC) \cong (\cT_n^\red\dC)^{\twoop},\qquad
  t^*(\dQ_{\cA_1,\dots,\cA_n}\cC)\cong\dQ_{\cA_1^\twoop,\dots,\cA_n^\twoop}\cC^\twoop.
\]
\end{Remark}

When $\phi$ is not nondecreasing, the functoriality involves inverting
$2$-cells.

\begin{Construction}[Functoriality of $\cL\cT$ and $\dQ^n$ for $\phi$ arbitrary]
Let $\phi\colon \{1,\dots,m\}\to \{1,\dots,n\}$ be an arbitrary map. We have
the following analogues of Constructions \ref{c.Tfun} and \ref{c.Qfun}. For
any $n$-fold category $\dC$, we have $2$-functors
\begin{equation}\label{e.LTphi}
  \cL\cT_m\phi^* \dC \to \cL\cT_n \dC, \quad \cL\cT_m^\red\phi^* \dC \to \cL\cT_n^\red \dC
\end{equation}
sending $\sigma_{g,D,f}$, with $D\in \dC(\epsilon_i+\epsilon_j)$, $i < j$,
$\phi(i)>\phi(j)$, to $\sigma^{-1}_{g,D,f}$. Consider the composite natural
transformation
\[\cL\cT_m \to \cL\cT_m \phi^*\phi_!\to \cL\cT_n\phi_!.\]
By adjunction, we obtain, for any $(2,1)$-category $\cC$, an $m$-fold
functor
\[
  \phi^*\dQ^n\cC \to \dQ^m\cC,
\]
which is an isomorphism when $\phi$ is injective.
\end{Construction}

\begin{Construction}[Functoriality of $\cL\overline{\cT_n\dQ^n}$ and of gluing
data]\label{c.Qphi}\index{LEphi@$\cL E_\phi$, $\cL E^\red_\phi$} Let
$\phi\colon \{1,\dots, m\}\to \{1,\dots, n\}$ be an arbitrary map and let
$\cC$ be a $(2,1)$-category. Similarly to Construction \ref{c.TQfun}, we
have $2$-functors
\begin{equation}\label{e.Ephi}
  \cL E_{\phi}\colon \cL\overline{\cT_m\dQ^m}\cC\to \cL\overline{\cT_n\dQ^n}\cC,\quad \cL E_{\phi}^\red\colon \cL\overline{\cT_m^\red\dQ^m}\cC\to \cL\overline{\cT_n^\red\dQ^n}\cC
\end{equation}
sending $\sigma_{g,D,f}$ to $\sigma^{-1}_{g,D^*,f}$, where $D\in
(\dQ^m\cC)(\epsilon_i+\epsilon_j)$, $i<j$, $\phi(i)>\phi(j)$, and $D^*$ is
the transpose of $D$ obtained by inverting the $2$-cell.

Now let $\cA_1,\dots, \cA_n,\cB_1,\dots,\cB_m$ be locally full
sub-$2$-categories of $\cC$, each containing all objects of $\cC$,
satisfying $\cB_i\subseteq \cA_{\phi(i)}$ for all $1\le i\le m$. By
restriction we obtain $2$-functors
\begin{equation}\label{e.LEphi}
  \cL E_{\phi}\colon \cL\overline{\cT_m\dQ}_{\cB_1,\dots,\cB_m}\cC\to \cL\overline{\cT_n\dQ}_{\cA_1,\dots,\cA_n}\cC,\quad \cL E_{\phi}^\red\colon \cL\overline{\cT_m^\red\dQ^m}\cC\to \cL\overline{\cT_n^\red\dQ^n}\cC
\end{equation}
Our extended goal is to study the $2$-functor
\begin{equation}\label{e.Qphi}\index{Qphi@$Q^\phi$}
Q^\phi=Q^\phi_\cD\colon \GD_{\cA_1,\dots,\cA_n}(\cC,\cD)\to \GD_{\cB_1,\dots,\cB_m}(\cC,\cD)
\end{equation}
induced by $\cL E_{\phi}$. For $n=1$ and $m\ge 2$, the $2$-functor $Q^\phi$
is the descent $2$-functor $Q^m$ \eqref{e.Q}.

For a sequence of arbitrary maps $\{1,\dots, l\} \xrightarrow{\psi}
\{1,\dots, m\} \xrightarrow{\phi} \{1,\dots,n\}$ and locally full
sub-$2$-categories $\cC_1,\dots,\cC_l$ of $\cC$, each containing all objects
of $\cC$, the composite
\[\GD_{\cA_1,\dots,\cA_n}(\cC,\cD)\xrightarrow{Q^\phi} \GD_{\cB_1,\dots,\cB_m}(\cC,\cD) \xrightarrow{Q^\psi}\GD_{\cC_1,\dots,\cC_l}(\cC,\cD)\]
equals $Q^{\phi\psi}$. In particular, for $m,n\ge 2$, the composite
\[\PsFun(\cC,\cD)\xrightarrow{Q^n} \GD_{\cA_1,\dots,\cA_n}(\cC,\cD) \xrightarrow{Q^\phi}\GD_{\cB_1,\dots,\cB_m}(\cC,\cD)\]
equals $Q^m$. If $\phi$ is a bijection and $\cB_i=\cA_{\phi(i)}$ for all
$1\le i\le m$, then $Q^\phi$ is a bijection.
\end{Construction}

By Proposition \ref{p.equiv} and Lemma \ref{l.PsFun2}, we have the following
criterion.

\begin{Proposition}\label{p.GDn}
Let $\cL E^\red_\phi\colon \cL\cT^\red_\phi\dQ_{\cA_1,\dots,\cA_n}\cC\to
\cC$ be the $2$-functor \eqref{e.LEphi}.
\begin{enumerate}
\item \label{p.GDn1} If $ \cL E^\red_\phi$ is locally essentially
    surjective, then $Q^\phi_\cD$ is $2$-faithful for every $2$-category
    $\cD$.
\item \label{p.GDn2} The following conditions are equivalent:
\begin{enumerate}
\item The $2$-functor $\cL E^\red_\phi$ is a bi-equivalence.
\item The $2$-functor $Q^\phi_\cD$ is a bi-equivalence for every
    $2$-category $\cD$.
\item The $2$-functor $Q^\phi_\cD$ is a $\cD^{\Ob(\cC)}$-$2$-equivalence
    for every $2$-category $\cD$.
\end{enumerate}
\end{enumerate}
\end{Proposition}

Later we will give an explicit (and uniform) description of the $2$-functor
$Q^\phi$ (see Remark \ref{r.QD}). However, the above construction via $\cL
E_\phi$ will be essential in proving Theorem \ref{t.main}.

\section{Gluing two pseudo functors}\label{s.two}
In this section, we study the gluing of two pseudo functors and prove
Theorem \ref{p.Del}. Throughout this section, $\cC$ is a $(2,1)$-category,
and $\cA$ and $\cB$ are locally full sub-$2$-categories that each contain
all objects of $\cC$. In particular, we have $\Ob(\cA)=\Ob(\cB)=\Ob(\cC)$.
The theorem gives sufficient conditions for the descent $2$-functor
\eqref{e.Q}
\[Q_\cD= Q_\cD^2\colon \PsFun(\cC,\cD)\to\GD_{\cA,\cB}(\cC,\cD),\]
where the target is the $2$-category of gluing data, to be a
$2$-equivalence. It is possible to state the theorem in concrete terms,
without reference to the constructions of Section~\ref{s.2}. In fact, we
will give an explicit description of the $2$-category of gluing data (Remark
\ref{r.gd}) and the descent $2$-functor (Remark \ref{r.QD}).

In Section~\ref{s.2}, the descent $2$-functor was defined more conceptually
using the operations $\cT_2$ ($2$-category of paths) and $\dQ^2$ (double
category of squares). Up to $2$-equivalences, the descent $2$-functor is
induced by the $2$-functor
\begin{equation}\label{e.E2}
E= \cL E^\red_2\colon \cL \cT^\red_2 \dQ_{\cA,\cB} \cC \to \cC.
\end{equation}
This conceptual interpretation is an essential ingredient in the proof of
Theorem \ref{p.Del}. Indeed, we will deduce Theorem \ref{p.Del} from a
criterion involving fundamental groups in the $2$-category
$\cL\cT^\red_2\dQ_{\cA,\cB}\cC$ (Theorem \ref{t.TQ}).

\begin{Remark}[Explicit description of gluing
data]\label{r.gd}\index{GD@$\GD_{\cA,\cB}(\cC,\cD)$}\index{GD@$G_D$} Let us
explicitly describe the $\cD^{\Ob(\cC)}$-$2$-category
$\GD_{\cA,\cB}(\cC,\cD)$. An \emph{object} of it is a triple
$\left(F_\cA,F_\cB,(G_D)\right)$ consisting of an object $F_\cA$ of
$\PsFun(\cA,\cD)$, an object $F_\cB$ of $\PsFun(\cB,\cD)$ satisfying $\lvert
F_\cA \rvert = \lvert F_\cB \rvert$ (recall from Notation \ref{d.bar} that
$\lvert -\rvert$ denotes the restriction to $\Ob(\cC)$), and a family of
invertible $2$-cells of~$\cD$
  \[G_D\colon F_\cA(i)F_\cB(q) \Rightarrow F_\cB(p)F_\cA(j),\]
$D$ running over $(\cA,\cB)$-squares in $\cC$ of the form \eqref{e.2}. The
triple is subject to the following conditions:
  \begin{description}
    \item[(a)] For any square $D$ of the form
  \[
  \xymatrix{X\ar@{=}[d]\ar[r]^j\drtwocell\omit{^\alpha} & Y \ar@{=}[d]\\
  X\ar[r]^i & Y,}
  \]
  the following square commutes
  \[\xymatrix{F_\cA(i)\ar@=[r]^{F_\cA(\alpha)}\ar@=[d] & F_\cA(j)\ar@=[d]\\
  F_\cA(i)F_\cB(\one_X)\ar@=[r]^{G_D} & F_\cB(\one_Y)F_\cA(j).}\]

    \item[(\aprime)] For any square $D$ of the form
  \[\xymatrix{X\ar@{=}[r]\ar[d]_q\drtwocell\omit{^\alpha} & X \ar[d]^p\\
  Y\ar@{=}[r] & Y,}\]
  the following square commutes
  \[\xymatrix{F_\cB(q)\ar@=[r]^{F_\cB(\alpha)}\ar@=[d] & F_\cB(p)\ar@=[d]\\
  F_\cA(\one_Y)F_\cB(q)\ar@=[r]^{G_D} & F_\cB(p)F_\cA(\one_X).}\]

\item[(b)] If $D$, $D'$, $D''$ are respectively the upper, lower and outer
    squares of the diagram
  \[\xymatrix{X_1\ar[r]^{i_1}\ar[d]_{q}\drtwocell\omit{^\alpha} &  Y_1\ar[d]^{p}\\
  X_2\ar[r]^{i_2}\drtwocell\omit{^\alpha'}\ar[d]_{q'} & Y_2\ar[d]^{p'}\\
  X_3\ar[r]^{i_3} & Y_3,}\]
    then the following pentagon commutes
    \[\xymatrix{F_{\cA}(i_3)F_\cB(q')F_\cB(q)\ar@=[r]^{G_{D'}}\ar@=[d]& F_\cB(p')F_\cA(i_2)F_\cB(q)\ar@=[r]^{G_{D}}& F_\cB(p')F_\cB(p)F_\cA(i_1)\ar@=[d]\\
    F_\cA(i_3)F_\cB(q'q)\ar@=[rr]^{G_{D''}} & & F_\cB(p'p)F_\cA(i_1).}\]

\item[(\bprime)] If $D$, $D'$, $D''$ are respectively the left, right and
    outer squares of the diagram
  \[\xymatrix{X_1\ar[r]^j\ar[d]_{p_1}\drtwocell\omit{^\alpha} &  X_2\ar[d]^{p_2}\ar[r]^{j'}\drtwocell\omit{^\alpha'} & X_3\ar[d]^{p_3}\\
  Y_1\ar[r]^i & Y_2\ar[r]^{i'} &Y_3,}\]
    then the following pentagon commutes
    \[\xymatrix{F_{\cA}(i')F_\cA(i)F_\cB(p_1)\ar@=[r]^{G_D}\ar@=[d]& F_\cA(i')F_\cB(p_2)F_\cA(j)\ar@=[r]^{G_{D'}}& F_\cB(p_3)F_\cA(j')F_\cA(j)\ar@=[d]\\
    F_\cA(i'i)F_\cB(p_1)\ar@=[rr]^{G_{D''}} & & F_\cB(p_3)F_\cA(j'j).}\]
  \end{description}

A \emph{morphism} $(F_\cA,F_\cB,G)\to (F'_\cA,F'_\cB,G')$ of
$\GD_{\cA,\cB}(\cC,\cD)$ is a pair $(\alpha_\cA,\alpha_\cB)$ consisting of a
morphism $\alpha_\cA\colon F_\cA \to F'_\cA$ of $\PsFun(\cA,\cD)$ and a
morphism $\alpha_\cB\colon F_\cB \to F'_\cB$ of $\PsFun(\cB,\cD)$ such that
$\lvert \alpha_\cA\rvert = \lvert \alpha_\cB \rvert$  and satisfying the
following condition
\begin{description}
\item[(m)] For any $(\cA,\cB)$-square $D$ \eqref{e.2}, the following
    hexagon commutes
  \[
  \xymatrix{\alpha_0(W)F_\cA(i)F_\cB(q)\ar@=[d]_{G_D}\ar@=[r]^{\alpha_\cA(i)}
  & F'_\cA(i)\alpha_0(Z)F_\cB(q)\ar@=[r]^{\alpha_\cB(q)}
  & F'_\cA(i)F'_\cB(q)\alpha_0(X)\ar@=[d]^{G'_D}\\
  \alpha_0(W)F_\cB(p)F_\cA(j)\ar@=[r]^{\alpha_\cB(p)}
  & F'_\cB(p)\alpha_0(Y)F_\cA(j)\ar@=[r]^{\alpha_\cA(j)}
  & F'_\cB(p)F'_\cA(j)\alpha_0(X).}
  \]
  Here $\alpha_0=\lvert \alpha_\cA\rvert = \lvert \alpha_\cB \rvert$.
\end{description}

  A \emph{$2$-cell} of $\GD_{\cA,\cB}(\cC,\cD)$ is a pair $(\Xi_\cA,\Xi_\cB)\colon (\alpha_\cA,\alpha_\cB)\Rightarrow (\alpha'_\cA,\alpha'_\cB)$ consisting of a $2$-cell $\Xi_\cA
  \colon \alpha_\cA\Rightarrow \alpha'_\cA$ of $\PsFun(\cA,\cD)$ and a $2$-cell $\Xi_\cB\colon \alpha_\cB\Rightarrow \alpha'_\cB$ of $\PsFun(\cB,\cD)$ such that $\lvert \Xi_\cA \rvert = \lvert \Xi_\cB \rvert$.

  The $\cD^{\Ob(\cC)}$-$2$-category structure of $\GD_{\cA,\cB}(\cC,\cD)$ is given by the $2$-functor defined by
  \[
  (F_\cA,F_\cB,G)\mapsto \lvert F_\cA \rvert=\lvert F_\cB \rvert,\quad
  (\alpha_\cA,\alpha_\cB)\mapsto \lvert \alpha_\cA\rvert = \lvert \alpha_\cB \rvert,\quad (\Xi_\cA,\Xi_\cB)\mapsto \lvert \Xi_\cA \rvert = \lvert \Xi_\cB\rvert.
  \]
\end{Remark}

In an object $(F_\cA,F_\cB,G)$ of $\GD_{\cA,\cB}(\cC,\cD)$, the family $G$
of $2$-cells expresses compatibility between the pseudo functors $F_\cA$ and
$F_\cB$. The necessity of such compatibility for gluing $F_\cA$ and $F_\cB$
into a pseudo functor $\cC\to \cD$ is also clear from the following remark.

\begin{Remark}[Explicit description of the descent $2$-functor]\label{r.QD}
Let us explicitly describe the descent $2$-functor
  \[
  Q_\cD\colon \PsFun(\cC,\cD)\to \GD_{\cA,\cB}(\cC,\cD).
  \]
For an object $F$ of $\PsFun(\cC,\cD)$, we have
$Q_\cD(F)=(F\res\cA,F\res\cB,G)$. Here, $F\res \cA$ and $F\res \cB$ denote
restrictions of $F$. For an $(\cA,\cB)$-square $D$ \eqref{e.2}, $G_D$ is the
composition
  \[\xymatrix{F(i)F(q)\ar@=[r]& F(iq)\ar@=[r]^{F(\alpha)} & F(pj)\ar@=[r] & F(p) F(j).}\]
For a morphism $\alpha\colon F\to F'$ of $\PsFun(\cC,\cD)$,
$Q_\cD(\alpha)$ is
  \[(\alpha\res\cA,\alpha\res\cB)\colon(F\res\cA,F\res\cB,G)\to (F'\res\cA,F'\res\cB,G').\]
For a $2$-cell $\Xi\colon \alpha\Rightarrow \alpha'$ of
$\PsFun(\cA,\cB)$, $Q_\cD(\Xi)$ is
  \[(\Xi\res\cA,\Xi\res\cB)\colon (\alpha\res\cA,\alpha\res\cB)\Rightarrow (\alpha'\res\cA,\alpha'\res\cB).\]
\end{Remark}

\begin{Remark}\label{s.gd}
The isomorphism of $\cD^{\Ob(\cC)}$-$2$-categories \eqref{e.Qphi}
  \[Q^t\colon \GD_{\cA,\cB}(\cC,\cD) \cong \GD_{\cB,\cA}(\cC,\cD)\]
associated to the map $t\colon \{1,2\}\to \{1,2\}$ swapping $1$ and $2$ can
be described as follows. To an object $(F_\cA,F_\cB,G)$ of
$\GD_{\cA,\cB}(\cC,\cD)$, we associate $(F_\cB,F_\cA,G^*)$. Here, for every
$(\cA,\cB)$-square $D$ \eqref{e.2}, $G^*_{D^*}=G_D^{-1}$, where $D^*$ is the
square obtained from $D$ by inverting $\alpha$. To a morphism
$(\alpha_\cA,\alpha_\cB)\colon (F_\cA,F_\cB,G)\to (F'_\cA,F'_\cB,G')$ of
$\GD_{\cA,\cB}(\cC,\cD)$, we associate
  \[(\alpha_\cB,\alpha_\cA)\colon (F_\cB,F_\cA,G^*)\to (F'_\cB,F'_\cA,G'{}^*).\]
To a $2$-cell $(\Xi_\cA,\Xi_\cB)\colon (\alpha_\cA,\alpha_\cB)\Rightarrow
(\alpha'_\cA,\alpha'_\cB)$ of $\GD_{\cA,\cB}(\cC,\cD)$, we associate
  \[(\Xi_\cB,\Xi_\cA)\colon (\alpha_\cB,\alpha_\cA)\Rightarrow (\alpha'_\cB,\alpha'_\cA).\]
\end{Remark}

\begin{Definition}\label{s.gen}
We say that $\cA$ and $\cB$ \emph{generate} $\cC$ if for every morphism $f$
of $\cC$, there exist morphisms $i_1,\dots, i_n$ of $\cA$ and $p_1,\dots,
p_n$ of $\cB$ and a $2$-cell $p_1 i_1 \dotsb p_n i_n\Rightarrow f$ of $\cC$.
\end{Definition}

\begin{Remark}
We have $\cA$ and $\cB$ generate $\cC$ if and only if $E\colon
\cL\cT_2^\red\dQ_{\cA,\cB}\cC\to \cC$ is locally essentially surjective. In
particular, if $\cA$ and $\cB$ generate $\cC$, then $Q_\cD$ is $2$-faithful
by Proposition \ref{p.GD} \ref{p.GD1}.
\end{Remark}

\begin{Remark}\label{r.DA}
  Let us recall the results of Deligne \cite[Proposition 3.3.2]{Deligne} and
  Ayoub \cite[Theorem~1.3.1]{Ayoub}. Both assume $\cC=\bC$ to be an (ordinary) category. Let $\cA=\bA$ and $\cB=\bB$ be subcategories of
  $\bC$, each containing all objects of $\bC$. Consider the
  conditions \ref{p.Del1} and \ref{p.Del2} of Theorem \ref{p.Del}, as well as the following
  conditions:
  \begin{itemize}
    \item[(2a)] Fiber products exist in $\bB$ and are fiber products in
        $\bC$.

    \item[(2b)] $\bC$ admits fiber products and morphisms of $\bA$ and
        $\bB$ are stable under base change by morphisms of $\bC$. (In
        particular, all isomorphisms in $\bC$ are in $\bA\cap \bB$.)
        Moreover, the diagonal of every morphism in $\bC$ is in $\bA$.
  \end{itemize}
Deligne assumes \ref{p.Del1} and (2a), while Ayoub assumes \ref{p.Del1} and
(2b). In our language, their conclusions can be stated as saying that
$Q_\cD$ induces a bijection between equivalence classes of $\PsFun(\bC,\cD)$
and $\GD_{\bA,\bB}(\bC,\cD)$. There are no implications between (2a) and
(2b). Moreover, each of (2a) and (2b) is stronger than \ref{p.Del2}. Thus
Theorem \ref{p.Del} is a common generalization of the results of Deligne and
Ayoub. Even under the assumptions of Deligne and Ayoub, the conclusion of
Theorem \ref{p.Del} is more precise in the sense that it establishes a
$2$-equivalence of $2$-categories. As we have remarked in the Introduction,
this precision is useful in the construction of pseudo natural
transformations.
\end{Remark}

By Proposition \ref{p.GD} \ref{p.GD2}, the conclusion of Theorem \ref{p.Del}
is equivalent to saying that $E\colon \cL\cT^\red_2\dQ_{\cA,\cB}\cC\to \cC$
is a bi-equivalence. We will construct a pseudo inverse $F$ as follows. The
quadruple $(Z,i,p,\alpha)$ in Theorem \ref{p.Del} \ref{p.Del1} corresponding
to the diagram
\[\xymatrix{X\ar[r]^i\ar[dr]_f \drtwocell\omit{<-1.5>\alpha}& Z\ar[d]^p\\
&Y }
\]
is a called a \emph{compactification} of $f$ (relative to $\cA$, $\cB$).
Given a choice of compactification, $p*i$ is a candidate for $Ff$. To
express the dependence on choices, we need to organize the set of
compactifications $\rComp_{\cA,\cB}(f)$ into $2$-categories.

\begin{Definition}[$2$-Category $\Comp^\cP_{\cA,\cB}(f)$ of compactifications of a morphism
$f$]\label{d.comp2}\index{Compf@$\Comp^\cP_{\cA,\cB}(f)$} Let $f\colon X\to
Y$ be a morphism of $\cC$. Let $\cP$ be a locally full sub-$2$-category of
$\cC$ containing all objects of $\cC$. We define the \emph{$2$-category of
compactifications of $f$} (relative to $\cA$, $\cB$, and $\cP$), denoted
$\Comp^\cP(f)=\Comp^\cP_{\cA,\cB}(f)$, as follows. The \emph{objects} are
compactifications of $f$ (relative to $\cA$, $\cB$, independent of $\cP$). A
\emph{morphism} $(Z,i,p,\alpha)\to (W,j,q,\beta)$ is a triple $(r, \gamma,
\delta)$ consisting of a morphism $r\colon Z\to W$ of~$\cP$, and $2$-cells
$\gamma\colon ri\Rightarrow j$ and $\delta\colon qr\Rightarrow p$ of $\cC$,
fitting in the diagram
\[\xymatrix{&Z\ar[d]^r\dduppertwocell^p{\delta^{-1}}\\
X\rtwocell\omit{<-1.5>\gamma}\drtwocell\omit{<-1.5>\beta} \ar[r]_j\ar[rd]_f\ar[ur]^i & W\ar[d]^q\\
&Y}
\]
where the outer triangle is $\alpha$. A \emph{$2$-cell}
$(r,\gamma,\delta)\Rightarrow (r',\gamma',\delta')$ is a $2$-cell $\epsilon
\colon r\Rightarrow r'$ of $\cP$ such that $\gamma=\gamma'\circ (\epsilon
i)$ and $\delta=\delta'\circ(q\epsilon)$. We omit $\cA$ and $\cB$ from the
notation when no confusion arises.
\end{Definition}

The $2$-category $\Comp^\cP(f)$ is a $(2,1)$-category, and is a locally full
sub-$2$-category of $\Comp^\cC(f)$ containing all objects of $\Comp^\cC(f)$.
We will mainly work with the category $\bO\Comp^\cP(f)$ (Construction
\ref{c.cO}). A compactification $(Z,i,p,\alpha)$ of $f$ gives a morphism
$p*i\colon X\to Y$ of $\cT^\red_2\dQ_{\cA,\cB}\cC$. A morphism
$(r,\gamma,\delta)$ of $\Comp^\cB(f)$ as above gives a $2$-cell
\[q*j= q*j*\vid_X \xRightarrow{\sigma_{q,D,\one_X}} q*r*i = qr*i
\Rightarrow p*i
\]
of $\cT^\red_2\dQ_{\cA,\cB}\cC$, where the last $2$-cell is given by the
image of $\delta$ under the pseudo functor $\cB\to
\cT^\red_2\dQ_{\cA,\cB}\cC$ induced by $G_\rv$ \eqref{e.TQ2}, and $D$ is the
square
\[\xymatrix{X\ar[r]^i\ar@{=}[d] & Z\ar[d]^r\\
X\ar[r]^j & W}
\]
given by $\gamma^{-1}$. This defines a functor
\[\bO\Comp^\cB(f)\to (\cT_2^\red\dQ_{\cA,\cB}\cC)(X,Y)^{\op}.\]
Composing with the functor inverting $2$-cells in
$\cE=\cL\cT_2^\red\dQ_{\cA,\cB}\cC$, we get a functor
\begin{equation}\label{e.ComptoE}
G_f\colon \bO\Comp^\cB(f)\to \cE(X,Y).
\end{equation}

To investigate the compatibility of $F$ with composition, we also need to
consider compactifications of sequences of morphisms. Recall that, for $n\ge
0$, $[n]$ denotes the totally ordered set $\{0,\dots, n\}$. We define the
\emph{$2$-category of $n$-simplices of $\cC$} to be the $2$-category
$\UPsFun([n],\cC)$ of strictly unital pseudo functors $[n]\to \cC$ (Notation
\ref{n.U}). A sequence of morphisms $X_0\xrightarrow{f_1} \dotso
\xrightarrow{f_n}X_n$ defines a $2$-functor $[n]\to \cC$, which we denote by
$(f_n,\dots,f_1)$. Let $\nabla_n$ be the partially ordered set
\[\{(k,l)\in[n]\times [n]\mid k\ge l\}.\]
The diagonal embedding $[n]\to \nabla_n$ induces a $2$-functor
$\UPsFun(\nabla_n,\cC)\to \UPsFun([n],\cC)$. The following generalization of
Definition \ref{d.comp2} slightly generalizes \cite[D\'efinition
3.2.5]{Deligne} and \cite[Section 1.3.1]{Ayoub}.

\begin{Definition}[$2$-Category $\Comp^\cP_{\cA,\cB}(\sigma)$ of compactifications of an $n$-simplex $\sigma$]\index{Compsigma@$\Comp^\cP_{\cA,\cB}(\sigma)$}
Let $\sigma$ be an $n$-simplex of $\cC$ and let $\cP$ be a locally full
sub-$2$-category of $\cC$ containing all objects of $\cC$. The
\emph{$2$-category of compactifications of $\sigma$} (relative to $\cA$,
$\cB$, and $\cP$), denoted $\Comp^\cP(\sigma)=\Comp^\cP_{\cA,\cB}(\sigma)$,
is the locally full sub-$2$-category of the strict fiber product
  \[\UPsFun(\nabla_n,\cC)\times_{\UPsFun([n],\cC)}\{\sigma\}\]
spanned by pseudo functors $F$ such that $F(k\to k',l)$ is a morphism of
$\cA$ and $F(k',l\to l')$ is a morphism of $\cB$, for all elements $(k,l)$
and $(k',l')$ of $\nabla_n$, and pseudo natural transformations $\alpha$
such that $\alpha(k,l)$ is a morphism in $\cP$ for every element $(k,l)$ of
$\nabla_n$. We denote the set of objects by $\rComp_{\cA,\cB}(\sigma)$,
which does not depend on $\cP$.
\end{Definition}

The $2$-category $\Comp^\cP(\sigma)$ is a $(2,1)$-category. A
compactification of $(f_n,\dots,f_1)$ can be represented by a diagram of the
form
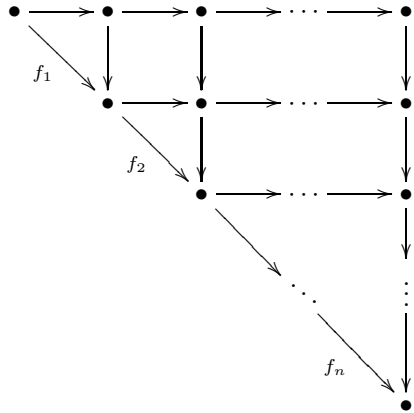
\[\xymatrix{\bullet \ar[r]\ar[rd]_{f_1} & \bullet \ar[r]\ar[d] & \bullet\ar[r]\ar[d] & \cdots\ar[r] & \bullet\ar[d]
\\
&\bullet\ar[rd]_{f_2}\ar[r] & \bullet \ar[r]\ar[d] & \cdots \ar[r] & \bullet \ar[d]\\
&&\bullet\ar[rd]\ar[r] & \cdots\ar[r] & \bullet\ar[d]\\
&&&\ddots\ar[rd]_{f_n} & \vdots\ar[d]\\
&&&&\bullet}
\]
with horizontal arrows in $\cA$, vertical arrows in $\cB$, and suitable
$2$-cells.

Let $d\colon [m]\to [n]$ be a nondecreasing map and let $\sigma$ be an
$n$-simplex. We have a $2$-functor
\begin{equation}\label{e.Compd}
\Comp^\cP(\sigma)\to \Comp^\cP(\sigma\circ d)
\end{equation}
sending $F\colon \nabla_n\to\cC$ to the composition $dF\colon
\nabla_m\xrightarrow{d\times d} \nabla_n\xrightarrow{F} \cC$.

We will deduce Theorem \ref{p.Del} from the following criterion.

\begin{Theorem}\label{t.TQ}
Let $\cC$ be a $(2,1)$-category and let $\cA$ and $\cB$ be locally full
sub-$2$-categories of $\cC$, each containing all objects of $\cC$. Let
$\cE=\cL\cT_2^\red\dQ_{\cA,\cB}\cC$. Assume the following:
  \begin{enumerate}
    \item \label{t.TQ1} For any morphism $f\colon X\to Y$ of $\cC$, there
        exists a compactification $\kappa=(Z,i,p,\alpha)$ of $f$ relative
        to $\cA$, $\cB$ such that the homomorphism of fundamental groups
    \[\pi_1(\bO\Comp^\cB(f),\kappa)\to \pi_1(\cE(X,Y), p*i)\]
        induced by $G_f$ \eqref{e.ComptoE} has trivial image.
    \item \label{t.TQ2} For any sequence of two morphisms
        $X\xrightarrow{f} Y\xrightarrow{g} Z$, $\Comp^\cB(g,f)$ is
        connected.
  \end{enumerate}
  Then the $2$-functor $E\colon \cE\to \cC$ \eqref{e.E2} is a bi-equivalence.
\end{Theorem}

Note that \ref{t.TQ2} implies that $\Comp^\cB(f)$ is connected, so that
\ref{t.TQ1} holds for every compactification $\kappa$ of~$f$.

\begin{Remark}
The proof of Theorem \ref{t.TQ} makes use of the following generalization of
\eqref{e.ComptoE}. There exists a unique way to associate, for every $n\ge
0$ and every $n$-simplex $\sigma$ of~$\cC$, a functor
\begin{equation}\label{e.ComptoEn}
G_\sigma\colon\bO\Comp^\cB(\sigma)\to \UPsFun([n],\cE),
\end{equation}
satisfying the following conditions:
\begin{itemize}
  \item[(0)] For every $0$-simplex (i.e.\ object) $X$ of $\cC$, the image
      of $G_X$ is $X$.

  \item[(1)] For every $1$-simplex (i.e.\ morphism) $f$ of $\cC$, the
      functor $G_f$ is \eqref{e.ComptoE}.

  \item[(2)] For every $2$-simplex $\sigma$ of $\cC$, the functor
      $G_\sigma$ sends a compactification as partly
shown by the diagram
\[\xymatrix{X\ar[r]_i\rruppertwocell^k{^\alpha} & X'\ar[r]_\ell\ar[d]_{p}\ar@{}[rd]|D& X''\dduppertwocell^r{^\beta}\ar[d]_{s}\\
& Y\ar[r]_j & Y'\ar[d]_{q}\\
&&Z}\]
to the $2$-simplex of $\cE$
\[\xymatrix{X\ar[r]_{p*i}\rruppertwocell^{r*k}{^} &Y\ar[r]_{q*j} & Z}\]
where the $2$-cell is the composition
\[q*j*p*i\xRightarrow{\sigma_{q,D,i}} q*s*\ell*i
=(qs)*(\ell i)
\xRightarrow{\beta*\alpha} r*k.
\]

\item[(3)] For any non-decreasing map $d\colon [m]\to [n]$, the diagram
    \[\xymatrix{\bO\Comp^\cB(\sigma)\ar[r]^{G_\sigma}\ar[d]_{\eqref{e.Compd}} & \UPsFun([n],\cE)\ar[d]^{\UPsFun(d,\cE)}\\
    \bO\Comp^\cB(\sigma\circ d) \ar[r]^{G_{\sigma\circ d}} & \UPsFun([m],\cE)}\]
    commutes.
\end{itemize}
It follows from (0) and (3) that the image of $G_\sigma$ lies in the
category $\UPsFun([n],\cE)\times_{\cE^{\Ob([n])}}\{\lvert \sigma\rvert\}$.
\end{Remark}

\begin{Lemma}\label{l.empty}
Let $\cC$ be a $(2,1)$-category and let $\cA$ and $\cB$ be locally full
sub-$2$-categories of $\cC$, each containing all objects of $\cC$, such that
$\rComp_{\cA,\cB}(f)$ is nonempty for every morphism $f$ of $\cC$. Then
$\rComp_{\cA,\cB}(\sigma)$ is nonempty for every $n$-simplex $\sigma$.
\end{Lemma}

Recall that $\rComp_{\cA,\cB}(\sigma)$ denotes the set of objects of
$\Comp_{\cA,\cB}^\cP(\sigma)$, which is independent of $\cP$.

\begin{proof}
Consider the subset
  \[\nabla_{n,j}=\{(k,l)\in [n]\times [n]\mid 0\le k-l\le j\}\subseteq \nabla_n\]
for $0\le j\le n$. We show by induction on $j$ that there exists a strictly
unital pseudo functor $F\colon \nabla_{n,j}\to \cC$ such that $F$ composed
with the diagonal embedding $[n]\to \nabla_{n,j}$ equals $\sigma$ and that
$F(k\to k',\ell)$ is in $\cA$, $F(k',\ell\to \ell')$ is in $\cB$. The case
$j=0$ is trivial as the diagonal embedding $[n]\to \nabla_{n,0}$ is an
isomorphism. For $j\ge 1$, assume that $F\colon \nabla_{n,j-1}\to \cC$ has
been constructed. Applying the hypothesis, we get a decomposition
$F((k+j-1,k)\to (k+j,k+1))\cong pi$ for $0\le k\le n-j$. We take $F(k+j-1\to
k+j,k)=i$ and $F(k+j,k\to k+1)=p$. In general, for $P\le Q$ in
$\nabla_{n,j}$ with $P\neq Q$, we take
\[F(P\to Q)=F(Q'\to Q) F(P'\to Q') F(P\to P')\]
where $P',Q'\in \nabla_{n,j-1}$ are given by
\[P'=\begin{cases}
  P &\text{if $P\in \nabla_{n,j-1}$,}\\
  P+(0,1)&\text{if $P\not\in \nabla_{n,j-1}$,}
\end{cases}\qquad
Q'=\begin{cases}
  Q &\text{if $Q\in \nabla_{n,j-1}$,}\\
  Q-(1,0)&\text{if $Q\not\in \nabla_{n,j-1}$.}
\end{cases}\]
The coherence constraint of $F$ is given by the obvious $2$-cells. For
$j=n$, we have $\nabla_{n,n}=\nabla_n$ so that $\rComp_{\cA,\cB}(\sigma)$ is
nonempty.
\end{proof}

\begin{proof}[Proof of Theorem \ref{t.TQ}]
For any set $S$, we denote by $S\sptilde$ the contractible groupoid of
underlying set $S$. In other words, we have $\Ob(S\sptilde)=S$ and for all
elements $X$ and $Y$ of $S$, there exists a unique morphism in $S\sptilde$
from $X$ to $Y$. By assumption, for any morphism $f\colon X\to Y$ of~$\cC$,
$G_f$ \eqref{e.ComptoE} factorizes through a functor $G_f\sptilde\colon
\rComp(f)\sptilde\to\cE(X,Y)$. For every $f$, choose an element $\kappa_f$
of $\rComp(f)$ such that $\kappa_{\one_X}=(X,\one_X,\one_X,\one_{\one_X})$
for every object $X$ of $\cC$. We denote by $d^n_i\colon [n-1]\to [n]$ and
$s^n_i\colon [n+1]\to [n]$, $0\le i\le n$, the face and degeneracy maps,
respectively.

We construct a strictly unital pseudo functor $F\colon \cC\to \cE$ as
follows. On objects we let $F$ be the identity. For any morphism $f$ of
$\cC$, we take $Ff=G_f(\kappa_f)$. In particular, $F(\id_X)=\id_X$. A
$2$-cell $\alpha\colon f\Rightarrow g$ of $\cC$ induces an isomorphism of
$2$-categories $H_\alpha\colon \Comp^\cB(f)\to \Comp^\cB(g)$ such that $G_g
H_\alpha = G_f$. We take $F\alpha=G_g\sptilde(H_\alpha(\kappa_f)\to
\kappa_g)$.

We construct the coherence constraint of $F$ as follows. Let
$X\xrightarrow{f} Y \xrightarrow{g} Z$ be a sequence of morphisms of $\cC$.
We consider the $2$-simplex $(g,f)$ of $\cC$ and the $2$-functor
$G_{g,f}\coloneqq G_{(g,f)}$ \eqref{e.ComptoEn}. For any object $\lambda$ of
$\Comp^\cB(g,f)$, the image $G_{g,f}(\lambda)$ is a $2$-simplex of $\cE$
with edges $G_f(d^2_2 \lambda)$, $G_{gf}(d^2_1\lambda)$,
$G_{g}(d^2_0\lambda)$. Applying Lemma \ref{l.ps} to $G_{g,f}(\lambda)$, we
get a unique pair $(F_\lambda,\phi_\lambda)$, where $F_\lambda$ is a
$2$-simplex of $\cE$ with edges $F(f)$, $F(gf)$, $F(g)$, and
$\phi_\lambda\colon G_{g,f}(\lambda)\to F_\lambda$ is a morphism of
$\UPsFun([2],\cE)\times_{\cE^{\Ob([2])}}\{\lvert (g,f) \rvert\}$ satisfying
\[
  \phi_\lambda(0\to 1)=G_f\sptilde(d^2_2\lambda\to \kappa_f),\quad \phi_\lambda(0\to 2)=G_{gf}\sptilde(d^2_1\lambda\to \kappa_{gf}),\quad \phi_\lambda(1\to 2)=G_g\sptilde(d^2_0\lambda\to \kappa_g).
  \]
By the uniqueness of the pair, for any morphism $\psi\colon \lambda\to
\lambda'$ of $\Comp^\cB(g,f)$, we have $F_\lambda=F_{\lambda'}$ and
$\phi_\lambda=\phi_{\lambda'}\circ G_{g,f}(\psi)$. It then follows from the
connectedness of $\Comp^\cB(g,f)$ that $F_\lambda$ does not depend on the
choice of $\lambda$ and we denote the corresponding $2$-cell of $\cE$ by
\[F_{g,f}\colon
F(g)F(f)\Rightarrow F(gf).
\]
A $2$-cell $\alpha\colon f\Rightarrow f'$ of $\cC$ induces an isomorphism of
$2$-categories $H_{g,\alpha}\colon \Comp^\cB(g,f)\to \Comp^\cB(g,f')$,
compatible with $H_\alpha$ and $H_{g\alpha}$ and such that
$G_{g,f'}H_{g,\alpha}=G_{g,f}$. Thus $F_{g,f}$ is functorial in $f$.
Similarly $F_{g,f}$ is functorial in $g$. By construction, $F_{\one_X,f}$ is
given by $G_{\one_X,f}(s^1_0 \kappa_f)=s^1_0G_f(\kappa_f)$, so that
$F_{\one_X,f}=\id_{F(f)}$. Similarly, $F_{f,\one_Y}=\id_{F(f)}$ since it is
given by $G_{f,\one_Y}(s^1_1\kappa_f)$.

Let $X\xrightarrow{f} Y \xrightarrow{g} Z\xrightarrow{h} W$ be a sequence of
morphisms of $\cC$. We consider the $3$-simplex $\sigma=(h,g,f)$ of $\cC$
and the $2$-functor $G_\sigma$ \eqref{e.ComptoEn}. For any object $\lambda$
of $\Comp^\cB(\sigma)$, the image $G_{\sigma}(\lambda)$ is a $3$-simplex of
$\cE$ with faces $G_{h,g}(d^3_0\lambda)$, $G_{h,gf}(d^3_1\lambda)$,
$G_{hg,f}(d^3_2\lambda)$, $G_{g,f}(d^3_2\lambda)$. Applying Lemma \ref{l.ps}
to $G_{\sigma}(\lambda)$, we get a pair $(F_\lambda,\phi_\lambda)$, where
$F_\lambda$ is a $3$-simplex of $\cE$ of edges $\kappa_f$, $\kappa_g$,
$\kappa_h$, $\kappa_{gf}$, $\kappa_{hg}$, $\kappa_{hgf}$, and
$\phi_\lambda\colon G_\sigma(\lambda)\to F_\lambda$ is a morphism of
$\UPsFun([3],\cE)\times_{\cE^{\Ob([3])}}\{\lvert (h,g,f) \rvert\}$ such that
\[\phi_\lambda(e)=G_{\sigma(e)}\sptilde(d_e\lambda\to \kappa_{\sigma(e)})\]
for all edges $e$ of $[3]$. Here $d_e\colon [1]\to [3]$ denotes
the map determined by $e$. By construction, $d^{3}_i
F_{\lambda}=F_{d^{3}_i\lambda}$ for $0\le i\le 3$. Thus $F_{\lambda}$
implies that the diagram
\[\xymatrix{F(h)F(g)F(f)\ar@=[r]^{F_{g,f}}\ar@=[d]_{F_{h,g}} & F(h)F(gf)\ar@=[d]^{F_{h,gf}}\\
F(hg)F(f)\ar@=[r]^{F_{hg,f}} & F(hgf)}
\]
commutes, which proves the composition axiom. This finishes the construction
of $F$.

Let $\tilde F\colon \cT\cC\to \cE$ be the $2$-functor induced by $F$.
Consider the commutative square
\[\xymatrix{\cL\cT_2 \dQ_{\cA,\cB}\cC\ar[d]_{R_\cE}\ar[r]^{\tilde E} & \cT\cC\ar[d]^{R_\cC}\\
\cE\ar[r]^E & \cC,}
\]
where $\tilde E=\cL E_2$, and the vertical arrows $R_\cC$ and $R_\cE$ are
the reduction $2$-functors. Since $R_\cC$ and $R_\cE$ are bi-equivalences,
it suffices to construct pseudo natural equivalences $E\tilde F\Rightarrow
R_\cC$ and $R_\cE\Rightarrow \tilde F \tilde E$.

We define a pseudo natural isomorphism $\epsilon\colon E\tilde F\Rightarrow
R_\cC$ sending $X$ to $\one_X$ as follows. To every morphism $f$ of $\cC$,
we associate the $2$-cell $\alpha\colon pi\Rightarrow f$ in
$\kappa_f=(Z,i,p,\alpha)$.

We define a pseudo natural isomorphism $\eta\colon R_{\cE}\Rightarrow \tilde
F\tilde E$ sending $X$ to $\one_X$ as follows. To every morphism $f$ of
$\cA$ (resp.\ $\cB$), we associate the $2$-cell
\begin{equation}\label{e.pnequiv}
f= f*\one_X^\rv
\xRightarrow{\sigma_{\one_Y,D,\one_X}} p*i \quad \text{(resp.\
$f= \one^\rh_Y*f
\xRightarrow{\sigma_{\one_Y,D',\one_X}} p*i$)},
\end{equation}
where $D$ (resp.\ $D'$) is the square
\[\xymatrix{X\ar[r]^i\ar@{=}[d]\ar@{}[rd]|D & Z\ar[d]^p & \text{(resp.}&X\ar[r]^i\ar[d]_f\ar@{}[rd]|{D'} & Z\text{)}\ar[d]^p\\
X\ar[r]^f & Y && Y\ar@{=}[r] & Y
}\]
induced by $\kappa_f$.
\end{proof}

\begin{Remark}
Similarly to \eqref{e.ComptoE}, we have a functor
\begin{equation}\label{e.CompAtoE}
G'_f\colon \bO\Comp^\cA(f)\to (\cT_2^\red\dQ_{\cA,\cB}\cC)(X,Y)\to\cE(X,Y).
\end{equation}
It sends every morphism $(r,\gamma,\delta)\colon (Z,i,p,\alpha)\to
(W,j,q,\beta)$ of $\Comp^\cA(f)$ to the composition
\[p*i= \hid_Y*p*i\xRightarrow{\sigma_{\one_Y,D,i}}
q*r*i= q*ri
\xRightarrow{\gamma} q*j,
\]
where $D$ is the square
\[\xymatrix{Z\ar[r]^r\ar[d]_p & W\ar[d]^q\\
Y\ar@{=}[r] & Y}
\]
given by $\delta^{-1}$. The analogue of Theorem \ref{t.TQ} holds with
$\Comp^\cB$ replaced by $\Comp^\cA$.

Moreover, \eqref{e.ComptoE} and \eqref{e.CompAtoE} extend to a functor
\begin{equation}\label{e.dComptoE}
\bT_2\dComp(f)\to
\cE(X,Y),
\end{equation}
where
\[
\dComp(f)=\dQ_{\bO\Comp^\cA(f),\bO\Comp^\cB(f)}\bO\Comp^\cC(f)
\]
is the double subcategory of $\dK^2(\bO\Comp^\cC(f))$.
\end{Remark}

By Proposition \ref{p.GD} \ref{p.GD2}, Theorem \ref{p.Del} follows from the
following.

\begin{Proposition}
Under the assumptions of Theorem \ref{p.Del},  for every $n$-simplex
$\sigma$ of $\cC$, the $2$-category $\Comp^\cB(\sigma)^{\coop}$ is
  directed. Moreover, the conditions of Theorem \ref{t.TQ} are satisfied.
\end{Proposition}

Here we say that a $2$-category $\cD$ is \emph{directed} if its underlying
category is directed, namely, if it is nonempty and if for every pair of
objects $X$ and $Y$ of $\cD$, there exist an object $Z$ of $\cD$ and
morphisms $X\to Z$ and $Y\to Z$ of $\cD$. A directed $2$-category is
connected.

\begin{proof}
By Lemma \ref{l.empty}, assumption \ref{p.Del1} of Theorem \ref{p.Del1}
implies that $\Comp^\cB(\sigma)$ is nonempty. Next we show the following:
\begin{itemize}
\item[($*$)] Every morphism of $n$-simplices $\sigma\to \tau$ in $\cC$,
    where $\tau$ is an $n$-simplex of $\cB$, is isomorphic in
    $\UPsFun([n],\cC)$ to the composition $\sigma\xrightarrow{i}
    \tau'\xrightarrow{p} \tau$, where $p$ is a morphism of $n$-simplices
    of $\cB$ and $i(j)$ is a morphism of $\cA$ for all $0\le j\le n$.
\end{itemize}
We proceed by induction on $n$. The case $n=0$ is assumption \ref{p.Del1} of
Theorem \ref{p.Del}. For $n\ge
    0$, induction hypothesis provides the restrictions $p\res\{1,\dots, n\}$ and
    $i\res\{1,\dots,n\}$. Applying  assumption \ref{p.Del2} of Theorem \ref{p.Del}, we obtain a commutative
    diagram with $2$-cells in $\cC$
\[\xymatrix{\sigma(0)\ar[r]^f\ar[d]&X\ar[d]\ar[r]&\tau(0)\ar[d]\\
\sigma(1)\ar[r]^{i(1)} & \tau'(1)\ar[r]^{p(1)} & \tau(1)}
\]
where the square on the right is in $\cB$. Applying  assumption \ref{p.Del1}
of Theorem \ref{p.Del} to $f$, we get $f\cong qh$. Replacing $f$ by $h$, we
may assume $f$ is a morphism of $\cA$. Then it suffices to take $\tau'(0)=X$
with the restrictions $p\res\{0,1\}$ and $i\res\{0,1\}$ given by the
diagram.

Now let $F$ and $F'$ be two objects of $\Comp^\cB_{\cA,\cB}(\sigma)$.
Applying  assumption \ref{p.Del2} of Theorem \ref{p.Del}, we obtain
morphisms $F''\to F$ and $F''\to F'$ of $\Comp^\cB_{\cC,\cB}(\sigma)$, where
$F''(k,l)$ is a pseudo fiber product $F(k,l)\times_{\sigma(k)} F'(k,l)$ in
$\cC$. To show that $\Comp^\cB_{\cA,\cB}(\sigma)$ is directed, it then
suffices to show that, for every object $F''$ of
$\Comp^\cB_{\cC,\cB}(\sigma)$, there exists a morphism $F'''\to F''$ of
$\Comp^\cB_{\cC,\cB}(\sigma)$ such that $F'''$ is an object of
$\Comp^\cB_{\cA,\cB}(\sigma)$. We will construct a sequence
\[F'''=F_n\to \dots \to F_1\to F_0=F''\]
of morphisms of $\Comp^\cB_{\cC,\cB}(\sigma)$ such that $F_j\res{\nabla_j}$
is an object of $\Comp^\cB_{\cA,\cB}(\sigma\res[j])$, $0\le j\le n$. For
$j\ge 1$, assume $F_{j-1}$ constructed. Applying ($*$), we get a
decomposition $F_{j-1}\res\{j-1\}\times [j-1]\to \tau \to
F_{j-1}\res\{j\}\times [j-1]$. We take
$F_j\res{\nabla_{j-1}}=F_{j-1}\res{\nabla_{j-1}}$, $F_j\res\{j\}\times [j-1]
=\tau$, and $F_j\res(\nabla_n-\nabla_j)=F_{j-1}\res(\nabla_n-\nabla_j)$.
This finishes the proof of the fact that
$\Comp^\cB_{\cA,\cB}(\sigma)^{\coop}$ is directed. In particular, we have
verified condition \ref{t.TQ2} of Theorem \ref{t.TQ}.

The verification of condition \ref{t.TQ1} of Theorem \ref{t.TQ} is
reminiscent of the proof of \cite[Lemme 1.3.8]{Ayoub}. Let $f\colon X\to Y$
be a morphism of $\cC$. Endow $S=\rComp(f)$ with the following preorder:
$\kappa\le \lambda$ if and only if there exists a morphism $\lambda\to
\kappa$ in $\Comp^\cB(f)$. Since $S$ is directed, it is simply connected.
Thus it suffices to show that for every pair of morphisms $(r_1,r_2)\colon
\lambda=(W,j,q,\tau)\rightrightarrows \kappa=(Z,i,p,\sigma)$ of
$\Comp^\cB(f)$, we have $G_f(r_1)=G_f(r_2)$, so that $G_f$ factorizes
through $S$. Applying assumption \ref{p.Del2} of Theorem \ref{p.Del}, we
obtain a diagram in $\cB$
\[\xymatrix{W\times_Y W \ar[r]\ar[d] & Z\times_Y W\ar[d] \ar[r] & W\ar[d]^{r_2}\\
W\times_Y Z\ar[r]\ar[d] & Z\times_Y Z \ar[r]^{p_2}\ar[d]^{p_1} & Z\ar[d]^p\\
W\ar[r]^{r_1} & Z \ar[r]^p & Y}
\]
where all the squares are Cartesian in $\cC$. Let $q_\alpha\colon W\times_Y
W\to W$, $\alpha=1,2$ be the projections. Applying ($*$), we obtain a
decomposition
\[\xymatrix{X\ar@{=}[d]\ar[r] & W'\ar[r]^t\ar[d]_u & W\times_Y W\ar[d]^{r_1\times_Y r_2}\\
X\ar[r] & Z'\ar[r]^s &Z\times_Y Z,}
\]
where $s$, $t$ and $u$ are in $\cB$ and the horizontal arrows on the left
are in $\cA$. This induces commutative squares with $2$-cells
\[\xymatrix{W'\ar[r]^{q_\alpha t}\ar[d]_u & W\ar[d]^{r_\alpha}\\
Z'\ar[r]^{p_\alpha s} & Z}
\]
in $\Comp^\cB(f)$. Since $\cE(X,Y)$ is a groupoid, it suffices to show
$G_f(p_1 s)=G_f(p_2 s)$ (and $G_f(q_1 t)=G_f(q_2 t)$). Applying ($*$) again,
we obtain a decomposition
\begin{equation}\label{e.Del}
\xymatrix{X\ar[r]\ar[d]_i & Z'''\ar[r]^v\ar[d]_{v'} & Z'\ar[d]^s\\
Z\ar[r]^h & Z''\ar[r]^{s'} & Z\times_Y Z,}
\end{equation}
where $s'$, $v$, $v'$ are in $\cB$, the left horizontal arrows are in $\cA$,
and the composition $hs'$ of the second line is isomorphic to the diagonal.
The diagram \eqref{e.Del} induces commutative squares with $2$-cells
\[\xymatrix{Z'''\ar[r]^v\ar[d]_{v'} & Z'\ar[d]^{p_\alpha s}\\
Z''\ar[r]^{p_\alpha s'} & Z}
\]
in $\Comp^\cB(f)$. Thus it suffices to show $G_f(p_1 s')=G_f(p_2 s')$.
Consider the $2$-functor \eqref{e.dComptoE}. Note that $h$ induces a
morphism in $\Comp^{\cA}(f)$. Since $p_\alpha s'*h= \one_{(Z,i,p,\sigma)}$
in $\bT_2\dComp(f)$, its image under \eqref{e.dComptoE} is the identity, so
that $G_f(p_1 s')=G_f(p_2 s')$. One can also check more directly that
$G_f(p_\alpha s')^{-1}=G'_f(h)$, where $G'_f$ is defined in
\eqref{e.CompAtoE}. Indeed, both $G_f(p_\alpha s')^{-1}$ and $G'_f(h)$ are
equal to the morphism $p*i\Rightarrow q*j$ induced by the diagram
\[\xymatrix{X\ar[r]^i\ar@{=}[d] & Z\ar@{=}[d]\ar[r]^h& Z''\ar[d]^{p_\alpha s'}\\
X\ar[r]^i & Z\ar[d]_p\ar@{=}[r] & Z\ar[d]^p\\
&Y\ar@{=}[r] & Y.}
\]
\end{proof}

This finishes the proof of Theorem \ref{p.Del}, which relates pseudo
functors to gluing data for two sub-$2$-categories. In the next section we
relate gluing data for two sub-$2$-categories to gluing data for more
sub-$2$-categories.

\section{Gluing finitely many pseudo functors}\label{s.finiteglue}
In this section, we study the gluing of finitely many pseudo functors in
general. The main result of this section is Theorem \ref{c.glue'}, which is
the precise form of Theorem \ref{t.main}. Although the constructions of
Sections \ref{s.2} and \ref{s.3} can be avoided in the statement of the
theorem (see Remarks \ref{r.gdn}, \ref{r.gdna}, and \ref{r.Qphi}), they
allow us to give a more conceptual interpretation and are used in the proof.
We deduce Theorem \ref{c.glue'} from a criterion (Theorem \ref{p.glue'})
involving $2$-functors of type \eqref{e.E2} studies in the previous section.

Throughout this section, $\cC$ is a $(2,1)$-category, $\cA_1,\dots, \cA_n$,
$n\ge 1$, are locally full sub-$2$-categories of $\cC$, each containing all
objects of $\cC$. Let $\cD$ be a $2$-category.

\begin{Remark}[Explicit description of gluing data]\label{r.gdn}\index{GD@$\GD_{\cA_1,\dots,\cA_n}(\cC,\cD)$}
Let us explicitly describe the $\cD^{\Ob(\cC)}$-$2$-category
$\GD_{\cA_1,\dots,\cA_n}(\cC,\cD)$. An \emph{object} is a pair
  \[\left((F_i)_{1\le i\le n},(G_{ij})_{1\le i<j\le n}\right),\]
where $F_i\colon \cA_i\to \cD$ is an object of $\PsFun(\cA_i,\cD)$, and
$(F_i,F_j,G_{ij})$ is an object of $\GD_{\cA_i,\cA_j}(\cC,\cD)$, as
described in Remark \ref{r.gd}, satisfying the following condition:
  \begin{description}
     \item[(D)] For $1\le i<j<k\le n$ and any commutative cube with
         $2$-cells of the form
     \begin{equation}\label{e.3}
    \xymatrix{X'\ar[dd]_{x}\ar[dr]^{q'}\ar[rr]^{b'} && Y'\ar[dd]^(.3){y}|\hole \ar[dr]^{p'}\\
    & Z'\ar[dd]^(.3){z}\ar[rr]^(.3){a'} && W'\ar[dd]^w \\
    X\ar[dr]_q\ar[rr]^(.3){b}|\hole && Y\ar[dr]^p\\
    & Z\ar[rr]^a && W }
     \end{equation}
    where $a,b,a',b'$ are morphisms of $\cA_i$, $p,q,p',q'$ are morphisms
    of $\cA_j$, and $x,y,z,w$ are morphisms of $\cA_k$, the following
    hexagon commutes
    \[\xymatrix{F_i(a)F_j(q)F_k(x)\ar@=[r]^{G_{jkI'}}\ar@=[d]_{G_{ijK}} & F_i(a)F_k(z)F_j(q')\ar@=[r]^{G_{ikJ}}
    & F_k(w)F_i(a')F_j(q')\ar@=[d]^{G_{ijK'}}\\
    F_j(p)F_i(b)F_k(x)\ar@=[r]^{G_{ikJ'}} & F_j(p)F_k(y)F_i(b')\ar@=[r]^{G_{jkI}} & F_k(w)F_j(p')F_i(b')}\]
    where $I,I',J,J',K,K'$ are respectively the right, left,
    front, back, bottom, top faces of the cube.
  \end{description}

  A \emph{morphism} $\left((F_i),(G_{ij})\right)\to ((F'_i),(G'_{ij}))$ of $\GD_{\cA_1,\dots,\cA_n}(\cC,\cD)$ is a collection
  $(\alpha_i)_{1\le i\le n}$
  of morphisms $\alpha_i\colon F_i\to F'_i$ of $\PsFun(\cA_i,\cD)$, such that for $1\le i<j\le n$,
  \[(\alpha_i,\alpha_j)\colon (F_i,F_j,G_{ij})\to(F'_i,F'_j,G'_{ij})\]
  is a morphism of $\GD_{\cA_i,\cA_j}(\cC,\cD)$.

A \emph{$2$-cell} of $\GD_{\cA_1,\dots,\cA_n}(\cC,\cD)$ is a collection
$(\Xi_i)_{1\le i \le n}\colon (\alpha_i)_{1\le i\le n}\Rightarrow
(\alpha'_i)_{1\le i \le n}$ of $2$-cells $\Xi_i\colon \alpha_i\Rightarrow
\alpha'_i$ of $\PsFun(\cA_i,\cD)$ such that $\lvert \Xi_i \rvert = \dots
=\lvert \Xi_n \rvert$.

  The $\cD^{\Ob(\cC)}$-$2$-category structure of $\GD_{\cA_1,\dots,\cA_n}(\cC,\cD)$ is given by the $2$-functor defined by
  \[
((F_i),G)\mapsto \lvert F_1\rvert = \dots =\lvert F_n \rvert,\quad
    (\alpha_i)\mapsto \lvert \alpha_1 \rvert = \dots=\lvert \alpha_n \rvert, \quad (\Xi_i)\mapsto \lvert \Xi_1 \rvert = \dots = \lvert \Xi_n \rvert.
  \]
\end{Remark}

\begin{Remark}[Alternative description of gluing data]\label{r.gdna}
Let us give an alternative description of the objects and morphisms of
$\GD_{\cA_1,\dots,\cA_n}(\cC,\cD)$. An \emph{object} of it is a pair
$((F_i)_{1\le i\le n},(G_{ijD})_{1\le i,j\le n})$ (here we do \emph{not}
assume $i<j$), where $F_i\colon \cA_i\to \cD$ is an object of
$\PsFun(\cA_i,\cD)$ such that $\lvert F_1\rvert = \dots = \lvert F_n
\rvert$, and
  \[G_{ijD}\colon F_i(a)F_j(q) \Rightarrow F_j(p)F_i(b)\]
is an invertible $2$-cell of $\cD$, where $D$ runs over
$(\cA_i,\cA_j)$-squares in $\cC$ (Example \ref{ex.up}) of the form
  \begin{equation}\label{e.2'}
  \xymatrix{X\ar[r]^b\ar[d]_q\drtwocell\omit{^\alpha} & Y \ar[d]^p\\
  Z\ar[r]^a & W,}
  \end{equation}
satisfying condition (D) of Remark \ref{r.gdn} for all $1\le i,j,k\le n$
(again here we do \emph{not} assume $i<j<k$) and the following conditions:
  \begin{description}
    \item[(A)] For $1\le i,j\le n$ and any square $D$ of the form
  \[
  \xymatrix{X\ar[r]^b\ar@{=}[d]\drtwocell\omit{^\alpha} & Y \ar@{=}[d]\\
  X\ar[r]^a & Y}
  \]
  where $a$ and $b$ are morphisms of $\cA_i$, the following square
  commutes
  \[\xymatrix{F_i(a)\ar@=[r]^{F_i(\alpha)}\ar@=[d] &F_i(b)\ar@=[d]\\
  F_i(a)F_j(\one_X)\ar@=[r]^{G_{ijD}} &F_j(\one_Y)F_i(b).}\]

    \item[(O)] For any square $D$ \eqref{e.2'} with $i=j$, the
        following square commutes
    \[\xymatrix{F_i(a)F_i(q)\ar@=[r]^{G_{iiD}}\ar@=[d] & F_i(p)F_i(b)\ar@=[d]\\
    F_i(aq)\ar@=[r]^{F_i(\alpha)} & F_i(pb).}\]
  \end{description}
In fact, given $(G_{ij})_{1\le i < j\le n}$, it suffices to take (O) as a
definition of $G_{ii}$, and to put $G_{ji}=G^*_{ij}$ (Remark \ref{s.gd}) for
$1\le i<j\le n$. Condition (D) for the triple $(i,i,j)$ follows from axiom
(b) of Remark \ref{r.gd} for $G_{ij}$.

A  \emph{morphism} $((F_i),(G_{ij}))\to ((F'_i),(G'_{ij}))$ of
$\GD_{\cA_1,\dots,\cA_n}(\cC,\cD)$ is a collection $(\alpha_i)_{1\le i\le
n}$  of morphisms $\alpha_i\colon F_i\to F'_i$ of $\PsFun(\cA_i,\cD)$ such
that for every $(\cA_i,\cA_j)$-square $D$ \eqref{e.2'}, the following
hexagon commutes
  \[\xymatrix{\alpha_0(W)F_i(a)F_j(q)\ar@=[d]^{G_{ijD}}\ar@=[r]^{\alpha_i(a)} & F'_i(a)\alpha_0(Z)F_j(q)\ar@=[r]^{\alpha_j(q)}
  &F_i'(a)F'_j(q)\alpha_0(X) \ar@=[d]^{G'_{ijD}} \\
  \alpha_0(W)F_j(p)F_i(b)\ar@=[r]^{\alpha_j(p)} & F'_j(p)\alpha_0(Y)F_i(b)\ar@=[r]^{\alpha_i(b)} &F'_j(p)F'_i(b)\alpha_0(X).}
  \]
Here $\alpha_0=\lvert \alpha_1\rvert = \dots=\lvert \alpha_n \rvert$.
\end{Remark}

\begin{Remark}[Explicit description of the $2$-functor
$Q^\phi$]\label{r.Qphi}\index{Qphi@$Q^\phi$} Let $\phi\colon
\{1,\dots,n\}\to \{1,\dots,m\}$ be a map and let
$\cA_1,\dots,\cA_n,\cB_1,\dots,\cB_m$ be locally full sub-$2$-categories of
$\cC$, each containing all objects of $\cC$, such that $\cA_i\subseteq
\cB_{\phi(i)}$ for $1\le i\le n$.  In Construction \ref{c.Qphi} we defined a
$\cD^{\Ob(\cC)}$-$2$-functor
\[Q^{\phi}=Q^{\phi}_{\cD}\colon \GD_{\cB_1,\dots,\cB_m}(\cC,\cD)\to \GD_{\cA_1,\dots,\cA_n}(\cC,\cD)\]
via a $2$-functor $E_\phi$. The definition of $E_\phi$ in the cases $m=1$
and $m=1$ needs special treatment. The description in Remark \ref{r.gdna}
allows us to explicitly and uniformly describe $Q^\phi$ as
  \[((F_i),(G_{ij}))\mapsto ((F_{\phi(i)}),(G_{\phi(i)\phi(j)})),\quad (\alpha_i)\mapsto (\alpha_{\phi(i)}), \quad (\Xi_i)\mapsto (\Xi_{\phi(i)}).\]
It follows immediately from this description that for a sequence of maps
$\{1,\dots,n\}\xrightarrow{\phi} \{1,\dots, m\} \xrightarrow{\psi}
\{1,\dots,l\}$, we have $Q^{\psi\phi}=Q^{\phi}Q^{\psi}$.
\end{Remark}

We have the following generalization of Corollary \ref{c.QD}.

\begin{Proposition}\label{p.gluet}
Let $\phi\colon \{1,\dots, n\}\to \{1,\dots,m\}$ be a map with a section
$\psi$. Assume $\cA_i\subseteq \cA_{\psi\phi(i)}$ for all $1\le i \le n$.
Then
\[Q^\phi\colon \GD_{\cA_{\psi(1)},\dots,
\cA_{\psi(m)}}(\cC,\cD)\to \GD_{\cA_1,\dots ,\cA_n}(\cC,\cD)
\]
and $Q^{\psi}$ are $\cD^{\Ob(\cC)}$-$2$-quasi-inverses to each other.
\end{Proposition}

\begin{proof}
This follows from Proposition \ref{p.Fphiequiv}. The above description of
$Q^\phi$ and $Q^\psi$ allows us to give a more direct proof as follows. We
have $Q^{\psi}Q^{\phi}=\one$. We construct a $\cD^{\Ob(\cC)}$-$2$-natural
isomorphism $\epsilon\colon Q^{\phi}Q^{\psi}\Rightarrow \one$ as follows.
For any object $((F_i),(G_{ij}))$ of $\GD_{\cA_1,\dots,\cA_n}(\cC,\cD)$, we
associate the morphism $Q^\phi Q^{\psi}((F_i),(G_{ij})) \to
((F_i),(G_{ij}))$ given by $\rho_{i,\psi\phi (i)}\colon F_{\psi\phi
(i)}\res\cA_i\to F_i$. See Remark \ref{s.gdp'} for the definition of
$\rho_{i,\psi\phi(i)}$.
\end{proof}

Similarly to Corollary \ref{c.QD}, Proposition \ref{p.gluet} does not
produce essentially new gluing data (see however Remark \ref{r.app} for an
application of the proposition). The goal of this section is to prove a more
substantial criterion for $Q^\phi$ to be a $2$-equivalence, where
$\phi\colon \{1,\dots,n\}\to\{1,\dots,m\}$ is a surjection with $m=n-1$. In
the case $n=2$, such a criterion was given in Theorem \ref{p.Del}. Here we
assume $n\ge 3$. Without loss of generality, we may assume that $\phi$ is
nondecreasing and $\phi(1)=\phi(2)$, so that $\phi(i)=\max\{1,i-1\}$ for
$1\le i \le n$. We assume $\cB_i=\cA_{i+1}$ for $2\le i\le m$. We simplify
notation as follows. We put $\cB=\cB_1$ and consider
\begin{gather}
\notag Q_\cD=Q_\cD^\phi\colon
\GD_{\cB,\cA_3,\dots,\cA_n}(\cC,\cD) \to \GD_{\cA_1,\dots,\cA_n}(\cC,\cD),\\
\label{e.E5} E=\cL E_\phi^\red\colon \cL\cT^\red_n\dQ_{\cA_1,\dots,\cA_n}\cC \to \cL\cT^\red_{n-1}\dQ_{\cB,\cA_3,\dots,\cA_n} \cC.
\end{gather}

\begin{Remark}
If $\cA_1$ and $\cA_2$ generate $\cB$ (Definition \ref{s.gen}), then $E$ is
locally essentially surjective, so that $Q_\cD$ is $2$-faithful by
Proposition \ref{p.GDn} \ref{p.GDn1}.
\end{Remark}

To state the main result of this section, we need to introduce some
terminology.

\begin{Definition}\label{s.square}
Let $\cC$ be a $(2,1)$-category and let $\cA,\cB,\cA_i,\cA_j,\cA_k$ be
locally full sub-$2$-categories of $\cC$, each containing all objects of
$\cC$.
\begin{enumerate}
\item We say that $(\cA,\cB)$ is \emph{squarable in $\cC$} if every pair
    of morphisms $i\colon Z\to W$ in $\cA$ and $p\colon Y\to W$ in $\cB$
    with the same target can be completed into an $(\cA,\cB)$-square
    \eqref{e.2}, Cartesian in $\cC$.

\item We say that $(\cA, \cB)$ is \emph{squaring in $\cC$} if every
    $(\cA,\cB)$-square \eqref{e.2} can be decomposed as
  \begin{equation}\label{e.square}
  \xymatrix{X\ar[rd]^f\xuppertwocell[rrd]{}^j{^\delta}\xlowertwocell[rdd]{}_q{^\gamma}\\
  &X'\ar[r]^k\ar[d]^r\drtwocell\omit{^\beta} & Y\ar[d]^p\\
  &Z\ar[r]^i & W}
  \end{equation}
where $k$ is a morphism of~$\cA$, and $r$ is a morphism of $\cB$, and $f$
is a morphism of $\cA\cap \cB$, and the inner square is Cartesian in
$\cC$.

\item We say that $(\cA_i,\cA_j)$ is \emph{$\cA_k$-squaring in $\cC$} if
    every cube with $2$-cells \eqref{e.3} can be decomposed as
  \begin{equation}\label{e.cubing}
  \xymatrix{X'\ar[rrd]^{f'} \xlowertwocell[rrrdd]{}_{q'}{^\gamma'}\xuppertwocell[rrrrd]{}^{b'}{^\delta'}\ar[dd]_x \\
  &&V'\ar[rr]^(.3){c'}\ar[rd]^{r'}\ar[dd]^(.25)v && Y'\ar[rd]^{p'}\ar[dd]^(.3)y|\hole\\
  X\xlowertwocell[rrrdd]{}_{q}{^\gamma}\ar[rrd]_{f}\ar@/^/[rrrrd]^(.3){b} \xtwocell[rrrd]{}\omit{^\delta} &&& Z'\ar[rr]^(.3){a'}\ar[dd]^(.7)z && W'\ar[dd]^w\\
  && V\ar[rr]^(.3){c}|\hole \ar[dr]^r&& Y\ar[dr]^p\\
  &&& Z\ar[rr]^a && W}
  \end{equation}
where $c,c'$ are morphisms of $\cA_i$, and $r,r'$ are morphisms of
$\cA_j$, and $f,f'$ are morphisms of $\cA_i\cap \cA_j$, and $v$ is a
morphism of $\cA_k$, and the bottom face $L$ and the top face $L'$ of the
inner cube are Cartesian in $\cC$. Note that if the right face $I$ (resp.\
the front face $J$) of the inner cube is Cartesian in $\cC$, so is the
left face $I''$ (resp.\ the back face $J''$) by \cite[Corollary
3.11]{Illusie-Zheng}.
\end{enumerate}
\end{Definition}

In this terminology, condition \ref{p.Del2} of Theorem \ref{p.Del} says that
$(\cB,\cB)$ is squarable. If pseudo fiber products exist in $\cB$ and are
pseudo fiber products in $\cC$ (cf.\ condition (2a) of Remark \ref{r.DA}),
then $(\cB, \cB)$ is squarable and squaring in $\cC$. One sufficient
condition for $(\cA, \cB)$ to be squaring in $\cC$ is that pseudo fiber
products exist in $\cC$, and $\cA$ and $\cB$ are stable under base change in
$\cC$ and taking diagonals in $\cC$.

One sufficient condition for $(\cA_i,\cA_j)$ to be $\cA_k$-squaring in $\cC$
is that $\cC$ admits pseudo fiber products, and $\cA_i,\cA_j,\cA_k$ are
stable under base change in $\cC$ and taking diagonals in $\cC$.

The following is the main result of this section.

\begin{Theorem}\label{c.glue'}
Let $\cC$ be a $(2,1)$-category. Let $n\ge 3$ and let $\cA_1,\dots,
\cA_n,\cB$ be locally full sub-$2$-categories of $\cC$, each containing all
objects of $\cC$. Assume the following:
\begin{enumerate}
  \item \label{c.glue1a} For every morphism $f$ of $\cB$, there exist a
      morphism $a$ of $\cA_1$, a morphism $p$ of $\cA_2$, and a $2$-cell
      $pa \Rightarrow f$ of $\cC$.

  \item \label{c.glue1b} For $3\le i\le n$ and every morphism $f$ of
      $\cB\cap \cA_i$, there exist a morphism $a$ of $\cA_1\cap \cA_i$, a
      morphism $p$ of $\cA_2\cap \cA_i$, and a $2$-cell $pa \Rightarrow f$
      of $\cC$.

  \item \label{c.glue2} For $3\le i<j\le n$, the sub-$2$-categories
      $\cA_1\cap \cA_i\cap \cA_j$ and $\cA_2\cap \cA_i\cap \cA_j$ generate
      $\cB\cap \cA_i\cap \cA_j$ (Definition \ref{s.gen}).

  \item \label{c.glue3} $(\cA_2,\cA_2)$ is squarable in $\cB$ and every
      Cartesian $(\cA_2,\cA_2)$-square in $\cB$ is Cartesian in $\cC$. For
      $3\le i\le n$, $(\cA_2\cap \cA_i, \cA_2\cap \cA_i)$ is squarable in
      $\cB\cap \cA_i$.

  \item \label{c.glue5} For $3\le i\le n$, the sub-$2$-categories
      $\cA_1\cap \cA_i$ and $\cA_2\cap \cA_i$ are stable under base change
      in $\cC$ by morphisms of $\cA_1$ whenever such base change exists.

  \item \label{c.glue4} For $3\le i\le n$, the sub-$2$-category $\cA_1$ is
      stable under base change in $\cC$ by morphisms of $\cA_i$ whenever
      such base change exists, $(\cA_2,\cA_i)$ is squarable in $\cC$, and
      $(\cB, \cA_i)$ is squaring in $\cC$.

  \item \label{c.glue4b} For $3\le i,j\le n$ with $i\neq j$, the
      sub-$2$-category $\cA_1\cap \cA_j$ is stable under base change in
      $\cC$ by morphisms of $\cA_i$ whenever such base change exists,
      $(\cA_2\cap \cA_j,\cA_i)$ is squarable in $\cC$, and $(\cB\cap
      \cA_j,\cA_i)$ is squaring in $\cC$.

  \item \label{c.glue6} For $3\le i,j\le n$ with $i\neq j$, the pair
      $(\cB,\cA_i)$ is $\cA_j$-squaring.
\end{enumerate}
Then
\[Q_\cD\colon \GD_{\cB,\cA_3,\dots,\cA_n}(\cC,\cD) \to
\GD_{\cA_1,\dots,\cA_n}(\cC,\cD)
\]
is a $\cD^{\Ob(\cC)}$-$2$-equivalence for every $2$-category $\cD$.
\end{Theorem}

Note that assumptions \ref{c.glue3} through and \ref{c.glue6} of Theorem
\ref{c.glue'} are all satisfied if $\cC$ admits pseudo fiber products,
$\cA_1,\dots,\cA_n,\cB$ are stable under $2$-base change in $\cC$, and
$\cA_3,\dots, \cA_n,\cB$ are stable under taking diagonals in $\cC$. For
$n=3$, assumptions \ref{c.glue2}, \ref{c.glue4b}, and \ref{c.glue6}
trivially hold.

\begin{Remark}\label{r.app}
Since every surjection $\phi\colon \{1,\dots,n\}\to \{1,\dots,m\}$ (for
arbitrary $m$, $n$) is a composition of surjections satisfying $m=n-1$,
Theorem \ref{p.Del} for $n=2$ and Theorem \ref{c.glue'} for $n\ge 3$ produce
results for $Q^\phi$ for every surjection $\phi$, and in particular for the
descent $2$-functor $Q^n$ for $n\ge 2$. In the case $n=3$, this is
illustrated by Corollary \ref{c.main}. We note that there are other ways to
combine results of this paper can be combined. In fact, we deduce in
\cite[Proposition 1.5]{sixop} from Theorem \ref{p.Del}, Theorem
\ref{c.glue'} (case $n=3$), and Proposition \ref{p.gluet}, a case where
$Q^2$ is a $2$-equivalence but for which condition \ref{p.Del1} of Theorem
\ref{p.Del} is not satisfied.
\end{Remark}

We will deduce Theorem \ref{c.glue'} from a general criterion. The idea is
to consider not only the gluing of the two sub-$2$-categories $\cA_1$ and
$\cA_2$ of $\cB$, but also similar gluing problems of sub-$2$-categories
with $\cB$ replaced by the $2$-category $\Ar(\cA_i;\cB)$ whose objects are
arrows in $\cA_i$ and whose morphisms are $(\cA_i,\cB)$-squares, for all
$3\le i\le n$. To state this formally, we introduce the following.

\begin{Notation}
Let $\cS,\cS',\cR$ be locally full sub-$2$-categories of $\cC$. We denote by
$\Ar(\cS;\cR)\subseteq \UPsFun([1],\cC)$ the locally full sub-$2$-category
spanned by strictly unital pseudo functors that factor through~$\cS$ and
pseudo natural transformations $\alpha$ such that $\alpha_0$ and $\alpha_1$
are both morphisms of $\cR$. We denote by $\Tr(\cS;\cR)\subseteq
\UPsFun([2],\cC)$ the locally full sub-$2$-category spanned by strictly
unital pseudo functors that factor through $\cS$ and pseudo natural
transformations $\alpha$ such that $\alpha_0,\alpha_1,\alpha_2$ are
morphisms of $\cR$. We denote by $\Sq(\cS,\cS';\cR)\subseteq
\UPsFun([1]\times [1],\cC)$ the locally full sub-$2$-category spanned by
strictly unital pseudo functors $F$ such that $F(0\to 1,0)$ and $F(0\to
1,1)$ are morphisms of $\cS$ and $F(0,0\to 1)$ and $F(1,0\to 1)$ are
morphisms of $\cS'$ and by pseudo natural transformations $\alpha$ such that
$\alpha(i,j)$, $0\le i,j\le 1$ are all morphisms of $\cT$.
\end{Notation}

Note that if $(\cA_i,\cA_j)$ is $\cA_k$-squaring in $\cC$, then
$(\cA_i,\cA_j)$ is squaring in $\cC$ and
$(\Ar(\cA_k;\cA_i),\Ar(\cA_k;\cA_j))$ is squaring in $\Ar(\cA_k;\cC)$.

For $3\le i \le n$, we consider the $2$-functor
\[G_i\colon \cG_i\coloneqq \cL\cT^\red_2\dQ_{\Ar(\cA_i;\cA_1),\Ar(\cA_i;\cA_2)}\Ar(\cA_i,\cB)\to \Ar(\cA_i;\cB)\]
induced by \eqref{e.E2}. For a morphism $f\colon x\to y$ of
$\Ar(\cA_i,\cB)$, we let $G_i^{-1}(f)$ denote the groupoid of pairs
$(g,\pi)$, where $g\colon x\to y$ is a morphism of $\cG_i$ and $\pi\colon
G_i(g)\Rightarrow f$ is a $2$-cell of $\Ar(\cA_i;\cB)$.

\begin{Theorem}\label{p.glue'}
Let $\cC$ be a $(2,1)$-category. Let $n\ge 3$ and let $\cA_1,\dots,
\cA_n,\cB$ be locally full sub-$2$-categories of $\cC$ such that
$\cA_1,\cA_2\subseteq \cB$. Assume the following:
  \begin{enumerate}
\item \label{p.glue1} The $2$-functor $G\colon \cG\coloneqq
    \cL\cT^\red_2\dQ_{\cA_1,\cA_2}\cB\to \cB$ \eqref{e.E2} is a
    bi-equivalence.

\item \label{p.glue2} For every morphism $f\colon x\to y$ of
    $\Ar(\cA_i;\cB)$, $3\le i\le n$, the groupoid $G_i^{-1}(f)$ is
    connected.

\item \label{p.glue3} For $3\le i \le n$, the sub-$2$-categories
    $\Tr(\cA_i;\cA_1)$ and $\Tr(\cA_i;\cA_2)$ generate $\Tr(\cA_i;\cB)$.

\item \label{p.glue4} For $3\le i<j\le n$, $\Sq(\cA_i,\cA_j;\cA_1)$ and
    $\Sq(\cA_i,\cA_j;\cA_2)$ generate $\Sq(\cA_i,\cA_j;\cB)$.
  \end{enumerate}
Then the $2$-functor $E\colon \cL \cT^\red_n \dQ_{\cA_1,\dots,\cA_n}\cC\to
\cL\cT^\red_{n-1}\dQ_{\cB,\cA_3,\dots,\cA_n}\cC$ \eqref{e.E5} is a
bi-equivalence.
\end{Theorem}

For $3\le i\le n$, assumption \ref{p.glue3} implies that $\Ar(\cA_i;\cA_1)$
and $\Ar(\cA_i;\cA_2)$ generate $\Ar(\cA_i;\cB)$, so that the groupoid
$G_i^{-1}(f)$ in \ref{p.glue2} is nonempty. Note also that for $n=3$,
assumption \ref{p.glue4} trivially holds.

\begin{proof}
For $3\le i\le n$, let $\cB_i=\Ar(\cA_i;\cB)$. We have source and target
$2$-functors $\tau_0, \tau_1\colon \cB_i\to \cB$, sending an object $f\colon
X\to Y$ of $\cB_i$ to $X$ and $Y$, respectively. These $2$-functors induce
$2$-functors $\cG_i\to \cG$, which we still denote by $\tau_0$ and $\tau_1$.
We have $\tau_0 G_i=G \tau_0$ and $\tau_1 G_i=G \tau_1$.

By Proposition \ref{p.equiv}, as $G$ is a bi-equivalence (assumption
\ref{p.glue1} of the theorem), there exist a pseudo functor $H\colon \cB\to
\cG$ and pseudo natural isomorphisms $\eta\colon \one_\cG\Rightarrow HG$ and
$\epsilon\colon GH\Rightarrow \one_\cB$ such that $H(X)=X$,
$\eta(X)=\one_X$, and $\epsilon(X)=\one_X$ for every object $X$ of $\cB$. We
may assume that $H$ is strictly unital.

Consider the commutative square
\[\xymatrix{\tilde \cE\ar[d]_{R_{\cE}} \ar[r]^{\tilde E} & \tilde
\cD\ar[d]^{R_{\cD}}\\
\cE\ar[r]^{E} & \cD,}
\]
where $\tilde E=\cL E_\phi\colon \cL \cT_n \dQ_{\cA_1,\dots,\cA_n}\cC\to
\cL\cT_{n-1}\dQ_{\cB,\cA_3,\dots,\cA_n}\cC$ and the vertical arrows $R_\cD$
and $R_\cE$ are the reduction $2$-functors.

We construct a $2$-functor $F\colon \tilde\cD\to \cE$ as follows. For an
object $X$ of $\tilde \cD$, we take $FX=X$. For a morphism of $\tilde \cD$
of length $1$, we take $F(f^k)=f^{k+1}$ for $2\le k\le n-1$ and $f$ in
$\cA_{k+1}$ and take $F(f^1)=H(f)$ for $f$ in $\cB$. We thus obtain a
functor from the underlying category of $\tilde \cD$ to the underlying
category of $\cC$.

Next we define the effects of $F$ on the pre-$2$-cells of
$\cT_{n-1}\dQ_{\cB,\cA_3,\dots,\cA_n}\cC$ (Definition \ref{d.T}). We take
$F(\iota^k_{g,f})=\id$ and $F(\theta^k_{g,f})=\id$ for all $1\le k\le n-1$.
We take $F(\gamma_{g,{h'}^k,h^k,f})=\id$ for $2\le k \le n-1$ and $h',h$ in
$\cA_{k+1}$ and take $F(\gamma_{g,{h'}^1,h^1,f})\colon
F(g)*H(h')*H(h)*F(f)\Rightarrow F(g)*H(h'h)*F(f)$ to be the $2$-cell induced
by the coherence constraint $H(h')*H(h)\Rightarrow H(h'h)$ of $H$ for
morphisms $h',h$ of $\cB$. In both cases, we take
$F(\delta_{g,{h'}^k,h^k,f})=F(\gamma_{g,{h'}^k,h^k,f})^{-1}$. Now let $D\in
(\dQ_{\cB,\cA_3,\dots,\cA_n}\cC)(\epsilon_k+\epsilon_{k'})$, $k<k'$. For
$k\ge 2$, we take $F(\sigma_{g,D,f})=\sigma_{F(g), D, F(f)}$. For $k=1$, we
view $D$ as a morphism $x\to y$ of $\cB_{k'+1}=\Ar(\cA_{k'+1};\cB)$. For an
object $\rho=(D',\pi)$ of $G_{k'+1}^{-1}(D)$, where $D'\colon x\to y$ is a
morphism of $\cG_{k'+1}$ and $\pi\colon G_{k'+1}(D')\Rightarrow D$ is a
$2$-cell of $\cB_{k'+1}$, we consider the $2$-cell
\[\alpha_\rho\colon H\tau_1 D*x \xRightarrow{\pi^{-1}} HG\tau_1 D'*x \xRightarrow{\eta^{-1}}
\tau_1 D'*x \xRightarrow{\sigma_{\one,D',\one}} y*\tau_0 D' \xRightarrow{\eta} y* HG\tau_0 D'
\xRightarrow{\pi} y*H\tau_0 D.
\]
It is straightforward to check that for a morphism $\rho\to \rho'$ of
$G_{k'+1}^{-1}(D)$, we have $\alpha_\rho=\alpha_{\rho'}$. Thus by the
connectedness of $G_{k'+1}^{-1}(D)$ (assumption \ref{p.glue2} of the
theorem), $\alpha_\rho$ does not depend on the choice of $\rho$, and we take
$F(\sigma_{g,D,f})$ to be the $2$-cell induced by $\alpha_\rho$.

To show that $F$ factors through a $2$-functor $\tilde\cD\to \cE$, it
suffices to check that the equivalence system defined by $\alpha\sim \beta$
if and only if $F(\alpha)=F(\beta)$ satisfies conditions \ref{d.T1} through
\ref{d.T9} of Definition \ref{d.T}. Conditions \ref{d.T1} through \ref{d.T5}
of Definition \ref{d.T} follow immediately from the construction. For the
other conditions, the cases that only concerns the directions
$\cA_3,\dots,\cA_n$ are also trivial. Let us check the nontrivial cases. The
case $k=1$ of condition \ref{d.T6} (associativity of composition) of
Definition \ref{d.T} follows from the composition axiom of the coherence
constraint of $H$. For the case $k'=1$ of condition \ref{d.T7} (unit square
in directions $k$ and $k'$, with units in direction $k'$) of Definition
\ref{d.T}, it suffices to take $\rho$ to be the object of $G_{k+1}^{-1}(D)$
given by the identity square \eqref{e.vidsquare}. For the case $k=1$ of
condition \ref{d.T7} of Definition \ref{d.T}, it suffices to take $\rho$ to
be the object of $G_{k'+1}^{-1}(D)$ given by the identity on an object of
the category $G^{-1}(f)$ of pairs $(g,\pi)$, where $g\colon X\to Y$ is a
morphism of $\cG$ and $\pi\colon g\Rightarrow f$ is a $2$-cell of~$\cB$. For
the case $k'=1$ of condition \ref{d.T8} (composition in direction $k$ of
squares in directions $k$ and $k'$) of Definition \ref{d.T}, it suffices to
take, for objects $\rho$ of $G_{k+1}^{-1}(D)$ and $\rho'$ of
$G_{k+1}^{-1}(D')$, their composite $\rho'\circ \rho$ in
$G_{k+1}^{-1}(D'')$. For the case $k=1$ of condition \ref{d.T8} of
Definition \ref{d.T}, the diagram defines a morphism of $\Tr(A_{k'+1};\cB)$,
and it suffices to take objects of $G_{k'+1}^{-1}(D)$, $G_{k'+1}^{-1}(D')$,
and $G_{k'+1}^{-1}(D'')$ given by assumption \ref{p.glue3} of the theorem.
For the case $k=1$ of condition \ref{d.T9} (cube in directions $k$, $k'$,
and $k''$) of Definition \ref{d.T}, the diagram defines a morphism in
$\Sq(\cA_{k'+1},\cA_{k''+1};\cB)$, and it suffices to take objects of
$G_{k'+1}^{-1}(K)$, $G_{k'+1}^{-1}(K')$, $G_{k''+1}^{-1}(J)$,
$G_{k''+1}^{-1}(J')$ given by assumption \ref{p.glue4} of the theorem. This
finishes the construction of the $2$-functor $F\colon \tilde\cD\to \cE$.

We define a pseudo natural isomorphism $\tilde\eta\colon R_\cE\to F\tilde E$
sending $X$ to $\one_X$ as follows. To every morphism $f$ of $\cA_i$ with
$3\le i\le n$, we associate $\one_f$. To every morphism $f$ of $\cA_1$
(resp.\ $\cA_2$), we associate $\eta(f)$.

We define a pseudo natural isomorphism $\tilde \epsilon\colon EF\Rightarrow
R_{\cD}$ sending $X$ to $\one_X$ as follows. To every morphism $f$ of
$\cA_i$ with $3\le i\le n$, we associate $\one_f$. To every morphism $f$ of
$\cB$ such that $H(f)=f_k*\dots *f_1$, we associate the composition
\[f_k*\dots*f_1 \Rightarrow f_k\circ \dots \circ f_1 \xRightarrow{\epsilon(f)} f,\]
where the first $2$-cell is given by $\gamma$.
\end{proof}

\begin{proof}[Proof of Theorem \ref{c.glue'}]
By Proposition \ref{p.GDn} \ref{p.GDn2}, it suffices to show that the
assumptions of Theorem \ref{c.glue'} imply the conditions of Theorem
\ref{p.glue'}. Let us first note that by assumption \ref{c.glue5} of Theorem
\ref{c.glue'}, each of $\cA_1,\dots,\cA_n$ contains all equivalences in
$\cC$.

To check condition \ref{p.glue1} of Theorem \ref{p.glue'}, it suffices to
apply Theorem \ref{p.Del} to the sub-$2$-categories $\cA_1$ and $\cA_2$ of
$\cB$. Here we used assumption \ref{c.glue1a} and the assumption that
$(\cA_2,\cA_2)$ is squarable in $\cB$ (assumption \ref{c.glue3}) of Theorem
\ref{c.glue'}.

Next we check condition \ref{p.glue2} of Theorem \ref{p.glue'}. To simplify
the notation, we put $\cB_i=\Ar(\cA_i;\cB)$, $\cA_{ki}=\Ar(\cA_i;\cA_k)$,
$k=1,2$.  Let $f$ be a morphism of $\cB_i$. Since $(\cB,\cA_i)$ is squaring
in $\cC$ (assumption \ref{c.glue4}), $f$ can be decomposed as follows
\[\xymatrix{X\ar[r]\ar[d] & X'\ar[r]\ar[d] & Y\ar[d]\\
Z\ar@{=}[r] & Z\ar[r] & W,}
\]
where the upper left horizontal arrow is in $\cB\cap \cA_i$ and the square
on the right is a Cartesian $(\cB,\cA_i)$-square in $\cC$. We fix some
notation in order to better analyze this decomposition. We let
$\tau_0,\tau_1\colon \cB_i\to \cB$ denote the source and target
$2$-functors. We let $\cB'_i\subseteq \Ar(\cA_i; \cB\cap \cA_i)$ denote the
locally full sub-$2$-category spanned by morphisms $f'$ such that $\tau_1
f'$ is an identity. We let $\cB''_i\subseteq \cB_{i}$ denote the locally
full sub-$2$-category spanned by morphisms corresponding to Cartesian
$(\cB,\cA_i)$-squares in $\cC$. The decomposition above provides morphisms
$f'$ in $\cB'_i$ and $f''$ in $\cB''_i$, and a $2$-cell $f''f'\Rightarrow f$
of $\cB_i$. We put $\cA'_{ki}=\cB'_i\cap \cA_{ki}$, $\cA''_{ki}=\cB''_i\cap
\cA_{ki}$.  By Theorem \ref{p.Del}, the $2$-functors
\[G'_i\colon \cL\cT_2^\red\dQ_{\cA'_{1i},\cA'_{2i}}\cB'_i\to \cB'_i,\quad
G''_i\colon \cL\cT_2^\red\dQ_{\cA''_{1i},\cA''_{2i}}\cB''_i\to \cB''_i
\]
are bi-equivalences. Here for $G'_i$, we used assumption \ref{c.glue1b} and
the second sentence of assumption \ref{c.glue3} (which implies that
$(\cA'_{2i},\cA'_{2i})$ is squarable in $\cB'_i$) of Theorem \ref{c.glue'}.
For $G''_i$, we used the assumption that $(\cA_2,\cA_i)$ is squarable in
$\cC$ and that $\cA_1$ is stable under base change in $\cC$ by $\cA_i$
whenever such base change exists (assumption \ref{c.glue4} of Theorem
\ref{c.glue'}) for condition \ref{p.Del1} of Theorem \ref{p.Del} and the
first sentence of assumption \ref{c.glue3} of Theorem \ref{c.glue'} (as well
as the fact that $\cA_2$ contains all equivalences in $\cC$) for the
condition that $(\cA''_{2i},\cA''_{2i})$ is squarable in $\cB''_i$
(condition \ref{p.Del2} of Theorem \ref{p.Del}). The $2$-cell
$f''f'\Rightarrow f$ induces a functor
\[F\colon G_i'^{-1}(f')\times
G_i''^{-1}(f'')\to G_i^{-1}(f),
\]
where $G_i'^{-1}(f')$ and $G_i''^{-1}(f'')$ are groupoids defined similarly
to $G_i^{-1}(f)$ and are contractible. Thus it suffices to show that every
object $\rho=(f_m*\dots*f_1,\pi)$ of $G_i^{-1}(f)$ is in the essential image
of $F$. Here each $f_s$, $1\le s\le m$ is either a morphism of $\cA_{1i}$ or
a morphism of $\cA_{2i}$. Note that, for every morphism $g$ of $\cA_{ki}$,
there exist $g'$ in $\cA'_{ki}$, $g''$ in $\cA''_{ki}$, and a $2$-cell
$g''g'\Rightarrow g$ in $\cA_{ki}$. Thus, up to replacing $m$ by $2m$, we
may assume that each $f_s$ is in $\cA'_{1i}$, or $\cA''_{1i}$, or
$\cA'_{2i}$, or $\cA''_{2i}$. Next we show that for $l=1,2$, morphisms of
$\cA''_{li}$ can be moved to the left. Let $g''\colon x\to y$ be a morphism
of $\cA''_{li}$, and let $g'\colon y\to z$ be a morphism of $\cA'_{ki}$.
Applying the assumption that $(\cB,\cA_i)$ is squaring, we obtain a $2$-cell
$g'g''\Rightarrow h''h'$, where $h'$ is in $\cB'_i$, and $h''$ is in
$\cA''_{li}$ by the assumption \ref{c.glue4} of Theorem \ref{c.glue'} (for
$l=2$ we also need the fact that $\cB\cap \cA_i$ contains all equivalences
in $\cC$). If $l=1$, then $h'$ is in $\cA'_{ki}$ by assumption \ref{c.glue5}
of Theorem \ref{c.glue'}, since $\tau_0 h'$ is a base change of $\tau_0 g'$.
For $l=2$, we may decompose $h'$ into $h'_2h'_1$ with $h'_1$ in $\cA'_{1i}$
and $h'_2$ in $\cA'_{2i}$, and we get a $2$-cell $g'g''\Rightarrow
(h''h'_2)h'_1$ in $\cG_i$. Thus we may assume that $\rho=(f''_{m''}*\dots
* f''_1*f'_{m'}*\dots*f'_1,\pi)$, where each $f'_s$ is in $\cA'_{1i}$ or
$\cA'_{2i}$ and each $f''_s$ is in $\cA''_{1i}$ or $\cA''_{2i}$. We may
further assume that the source of $f''_1$ is the source of $f''$. In this
case $\rho$ is in the image of $F$.

Next we check condition \ref{p.glue3} of Theorem \ref{p.glue'}. Let $f$ be a
morphism of $\Tr(\cA_i;\cB)$. There is a $2$-cell $f\Rightarrow g''g'g$ with
$g''g'g$ of the form
\[\xymatrix{X_0\ar[r]^{a_0}\ar[d] & Y_0\ar[d]\ar[r]\ar@{}[rd]|{D'} & Z_0\ar[r]\ar[d]\ar@{}[rd]|{D''} &W_0\ar[d]\\
Y_1\ar@{=}[r]\ar[d]\ar@{}[rd]|= & Y_1\ar[r]^{a'_1}\ar[d] & Z_1\ar[r]\ar[d]\ar@{}[rd]|{E''} & W_1\ar[d]\\
Z_2\ar@{=}[r] & Z_2\ar@{=}[r] & Z_2\ar[r]^{a''_2} & W_2,}
\]
where $E''$, $D''$, and $D'$ are Cartesian $(\cB,\cA_i)$-squares in $\cC$,
and the arrows $a_0$ and $a'_1$ are in $\cB\cap \cA_i$. The three Cartesian
squares are constructed successively using the assumption that $(\cB,\cA_i)$
is squaring. We decompose $a''_2$ by assumption \ref{c.glue1a} of Theorem
\ref{c.glue'} and decompose $a_0$ and $a'_1$ by assumption \ref{c.glue1b} of
Theorem \ref{c.glue'}. We get a $2$-cell $g\Rightarrow g_2g_1$, where $g_k$
is in $\Tr(\cA_i;\cA_k\cap \cA_i)$ for $k=1,2$. Applying the first two parts
of assumption \ref{c.glue4} of Theorem \ref{c.glue'}, we get $2$-cells
$g'\Rightarrow g'_2g'_1$ and $g''\Rightarrow g''_2g''_1$, where $g'_k$ and
$g''_k$ are in $\Tr(\cA_i;\cA_k)$ for $k=1,2$.

Finally we check condition \ref{p.glue4} of Theorem \ref{p.glue'}. Let $f$
be a morphism of $\Sq(\cA_i,\cA_j;\cB)$. We let
$\tau_{\to,0},\tau_{\to,1}\colon \Sq(\cA_i,\cA_j;\cB)\to \cB_i$ and
$\tau_{0,\to},\tau_{1,\to}\colon \Sq(\cA_i,\cA_j;\cB)\to \cB_j$ denote the
restriction $2$-functors. Applying the assumption that $(\cB,\cA_j)$ is
$\cA_i$-squaring in $\cC$ (assumption \ref{c.glue6} of Theorem
\ref{c.glue'}), we obtain a $2$-cell $f\Rightarrow f'g$, where $f'$ is a
morphism of $\Sq(\cA_i,\cA_j;\cB)$ such that $\tau_{0,\to}f'$ and
$\tau_{1,\to}f'$ are given by Cartesian $(\cB,\cA_j)$-squares in $\cC$, and
$g$ is a morphism in $\Sq(\cA_i,\cA_j;\cB\cap \cA_j)$ such that $\tau_{\to,
1}g$ is an identity. Applying the assumption that $(\cB,\cA_i)$ is
$\cA_j$-squaring in $\cC$, we obtain a $2$-cell $f'\Rightarrow f''g'$, where
$f''$ is a morphism of $\Sq(\cA_i,\cA_j;\cB)$ whose image under all four
restriction $2$-functors are given by Cartesian squares, and $g'$ is a
morphism in $\Sq(\cA_i,\cA_j;\cB\cap \cA_i)$ such that $\tau_{1,\to}g'$ is
an identity and $\tau_{0,\to}g'$ is given by a Cartesian $(\cA_j,\cB\cap
\cA_i)$-square in $\cC$. Applying assumption \ref{c.glue1a} and the first
two parts of assumption \ref{c.glue4} of Theorem \ref{c.glue'}, we get a
$2$-cell $f''\Rightarrow f''_2f''_1$, where $f''_k$ is in
$\Sq(\cA_i,\cA_j;\cA_k)$ for $k=1,2$. Applying assumption \ref{c.glue1b} and
the first two parts of assumption \ref{c.glue4b} of Theorem \ref{c.glue'},
we get a $2$-cell $g'\Rightarrow g'_2g'_1$, where $g'_k$ is in
$\Sq(\cA_i,\cA_j;\cA_k\cap \cA_i)$ for $k=1,2$. Applying assumptions
\ref{c.glue1b} and \ref{c.glue4b} of Theorem \ref{c.glue'}, we get a
$2$-cell $g\Rightarrow g_2g_1h$, where $g_k$ is in
$\Sq(\cA_i,\cA_j;\cA_k\cap \cA_j)$ for $k=1,2$ and $h$ is a morphism in
$\Sq(\cA_i,\cA_j;\cB\cap \cA_i\cap\cA_j)$ such that $\tau_{1,\to}h$ and
$\tau_{\to,1}h$ are identities. Finally applying assumption \ref{c.glue2} of
Theorem \ref{c.glue'}, we decompose $h$ into morphisms of
$\Sq(\cA_i,\cA_j;\cA_k\cap \cA_i\cap\cA_j)$ for $k=1,2$.
\end{proof}

This finishes the first part of the article, which relates pseudo functors
to gluing data. In the next part of the article (Sections \ref{s.6}
through~\ref{s.adjoint}), we develop several tools for constructing gluing
data. In Sections~\ref{s.6} and \ref{s.7}, we introduce Cartesian gluing
data, which are usually easier to construct in applications.

\section{Cartesian gluing data for two pseudo functors}\label{s.6}
Let $\cC$ be a $(2,1)$-category and let $\cA$ and $\cB$ be two locally full
sub-$2$-categories of $\cC$, each containing all objects of $\cC$. Let $\cD$
be a $2$-category. We studied the $2$-category $\GD_{\cA,\cB}(\cC,\cD)$ of
gluing data in Section \ref{s.two}. One way to construct such data is by
taking adjoints in base change isomorphisms (see Section \ref{s.adjoint}).
In many applications, these isomorphisms only exist for Cartesian squares.
In this section, we introduce a variant $\GD^\Cart_{\cA,\cB}(\cC,\cD)$ of
$\GD_{\cA,\cB}(\cC,\cD)$, whose objects only make use of $G_D$ for Cartesian
squares $D$. The main result of this section is a criterion for $\GD^\Cart$
and $\GD$ to be isomorphic (Theorem \ref{p.gdp}). This is used in the
construction of $Rf_!$ for Deligne-Mumford stacks in \cite{sixop} to produce
the desired gluing data.

The idea of using Cartesian squares as an intermediary step to construct
gluing data was already used by Deligne \cite[5.1.5]{Deligne} and Ayoub
\cite[Section 1.6.5]{Ayoub}. In \cite[Section 5.1]{Deligne} it is possible
to avoid this intermediary step by taking $\cA$ to be spanned by
\emph{dominant} open immersions and $\cB$ by proper morphisms, so that every
$(\cA,\cB)$-square is Cartesian. However, the intermediary step is necessary
in other applications.

\begin{Construction}[Pseudo natural isomorphism $\rho$ from a gluing
datum]\label{s.gdp}\index{rho@$\rho$} Let $(F_\cA,F_\cB,G)$ be an object of
$\GD_{\cA,\cB}(\cC,\cD)$, and let $F_0=\lvert F_\cA\rvert = \lvert F_\cB
\rvert$. Between the restrictions to $\cA\cap \cB$, we define an isomorphism
$\rho\colon F_\cB\res\cA\cap \cB \to F_\cA\res\cA\cap \cB$ in
$\PsFun(\cA\cap\cB,\cD)$ with $\lvert \rho \rvert = \one_{F_0}$ as follows.
For any morphism $f\colon X\to Y$ of $\cA\cap \cB$, let $\rho(f)$ be the
composition
  \[\xymatrix{F_\cB(f)\ar@=[r] & F_\cA(\one_Y) F_\cB(f)\ar@=[r]^{G_D} & F_\cB(\one_Y) F_\cA(f)\ar@=[r]& F_\cA(f)},\]
  where $D$ is the left square in the diagram
  \[\xymatrix{X\ar@{=}[r]\drtwocell\omit{=}\ar@{=}[d] &X\ar[d]^f\ar[r]^f\drtwocell\omit{=} & Y\ar@{=}[d]\\
  X\ar[r]^f & Y\ar@{=}[r] & Y.}\]
Denote the right square by $D'$. Applying axiom (a) of Remark \ref{r.gd} to
the outer square and axiom (\bprime) of Remark \ref{r.gd} to the above
diagram, one sees that $\rho(f)$ is the inverse of the composition
  \[\xymatrix{F_\cA(f)\ar@=[r] & F_\cA(f)F_\cB(\one_X) \ar@=[r]^{G_{D'}} & F_\cB(f) F_\cA(\one_X)\ar@=[r] & F_\cB(f).}\]
For any object $X$ of $\cC$, applying axiom (a) of Remark \ref{r.gd} to the
constant square $[1]\times [1]\to \cC$ of value $X$, one finds that the
following diagram commutes
  \[\xymatrix{\one_{F_0 X}\ar@=[d]\ar@=[dr]\\
  F_\cB(\one_X)\ar@=[r]^{\rho(\one_X)} & F_\cA(\one_X).}\]
For any sequence of morphisms $X\xrightarrow{f} Y\xrightarrow{g} Z$,
applying axioms (a), (b), (\bprime) of Remark \ref{r.gd} to
  \[\xymatrix{X\ar[r]^f\ar[d]_f\drtwocell\omit{=} & Y\ar[r]^g\ar@{=}[d]\drtwocell\omit{=} & Z\ar@{=}[d]\\
  Y\ar[d]_g\ar@{=}[r]\xtwocell[rrd]{}\omit{=} & Y\ar[r]^{g} & Z\ar@{=}[d]\\
  Z\ar@{=}[rr] && Z,}\]
  one finds that the following diagram commutes
  \[\xymatrix{F_\cB(g)F_\cB(f)\ar@=[r]^{\rho(g)\rho(f)}\ar@=[d] & F_\cA(g)F_\cA(f)\ar@=[d]\\
  F_\cB(gf)\ar@=[r]^{\rho(gf)} & F_\cA(gf).}\]
  Therefore, $\rho$ is a pseudo natural isomorphism.
\end{Construction}

\begin{Remark}\label{r.cgd}
The pseudo natural isomorphism $\rho$ of Construction \ref{s.gdp} has the
following properties:
    \begin{description}
    \item[(c)] If $D$ is an $(\cA,\cB)$-square \eqref{e.2} such that $p,q$
        are morphisms of $\cA\cap \cB$, then the following hexagon
        commutes
    \[\xymatrix{F_\cA(i)F_\cB(q)\ar@=[d]_{G_D} \ar@=[r]^{\rho(q)} & F_\cA(i)F_\cA(q)\ar@=[r] & F_\cA(iq)\ar@=[d]^{F_\cA(\alpha)}\\
    F_\cB(p)F_\cA(j) \ar@=[r]^{\rho(p)} & F_\cA(p)F_\cA(j)\ar@=[r] &F_\cA(pj).}\]

    \item[(\cprime)] If $D$ is an $(\cA,\cB)$-square \eqref{e.2} such that
        $i,j$ are morphisms of $\cA\cap \cB$, then the following hexagon
        commutes
    \[\xymatrix{F_\cA(i)F_\cB(q)\ar@=[d]_{G_D} \ar@=[r]^{\rho(i)^{-1}} & F_\cB(i)F_\cB(q)\ar@=[r] & F_\cB(iq)\ar@=[d]^{F_\cB(\alpha)}\\
    F_\cB(p)F_\cA(j) \ar@=[r]^{\rho(j)^{-1}} & F_\cB(p)F_\cB(j)\ar@=[r] &F_\cB(pj).}\]
  \end{description}
In fact, any square $D$ in (c) can be decomposed as
  \[\xymatrix{X\ar[r]^j\ar@{=}[d]\drtwocell\omit{=} & Y\ar@{=}[r]\ar@{=}[d]\drtwocell\omit{=} & Y\ar[d]^p\\
  X\ar[r]^j\ar@{=}[d]\xtwocell[rrd]{}\omit{^\alpha} & Y\ar[r]^p & W\ar@{=}[d]\\
  X\ar[r]^q\ar[d]_q\drtwocell\omit{=} & Z\ar[r]^i\ar@{=}[d]\drtwocell\omit{=} & W\ar@{=}[d]\\
  Z\ar@{=}[r] & Z\ar[r]^i & W.}\]
Denote the upper left, upper right, middle, lower left, lower right squares
by $D_1$, $D_2$, $D_3$, $D_4$ and $D_5$, respectively. Then $G_{D_2}$ can be
identified with $\rho(p)^{-1}$ and $G_{D_4}$ can be identified with
$\rho(q)$. By axiom (a) of Remark \ref{r.gd}, $G_{D_1}$ and $G_{D_5}$ can be
identified with identities and $G_{D_3}$ can be identified with
$F_\cA(\alpha)$.  Hence axioms (b) and (\bprime) of Remark \ref{r.gd} imply
that the hexagon in (c) commutes. Similarly, axioms (\aprime), (b) and
(\bprime) of Remark \ref{r.gd} imply (\cprime).
\end{Remark}

\begin{Definition}[Cartesian gluing data]\label{d.gdp}\index{GDCart@$\GD^\Cart_{\cA,\cB}(\cC,\cD)$}
Define the $2$-category $\GD^\Cart_{\cA,\cB}(\cC,\cD)$ of Cartesian gluing
data as follows. An \emph{object} of this category is a quadruple
$\left(F_\cA,F_\cB,(G_D), \rho\right)$ consisting of an object  $F_\cA$ of
$\PsFun(\cA,\cD)$, an object $F_\cB$ of $\PsFun(\cB,\cD)$, a family of
invertible $2$-cells of $\cD$
  \[G_D\colon F_\cA(i)F_\cB(q) \Rightarrow F_\cB(p)F_\cA(j),\]
$D$ running over \emph{Cartesian} $(\cA,\cB)$-squares in $\cC$ of the form
\eqref{e.2}, and an isomorphism $\rho\colon F_\cB\res\cA\cap \cB\to
F_\cA\res\cA\cap \cB$ in $\PsFun(\cA\cap\cB,\cD)$, such that $\lvert
\rho\rvert = \one_{F_0}$, where $F_0=\lvert F_\cA\rvert = \lvert F_\cB
\rvert$, and satisfying conditions (b), (\bprime) of Remark \ref{r.gd} and
conditions (c), (\cprime) of Remark \ref{r.cgd} for Cartesian
$(\cA,\cB)$-squares.

A \emph{morphism} $(F_\cA,F_\cB,G,\rho)\to (F'_\cA,F'_\cB,G',\rho')$ of
$\GD^\Cart_{\cA,\cB}(\cC,\cD)$ is a pair $(\alpha_\cA,\alpha_\cB)$
consisting of a morphism $\alpha_\cA\colon F_\cA \to F'_\cA$ of
$\PsFun(\cA,\cD)$ and a morphism $\alpha_\cB\colon F_\cB \to F'_\cB$ of
$\PsFun(\cB,\cD)$, such that $\lvert \alpha_\cA\rvert = \lvert \alpha_\cB
\rvert$, satisfying condition (m) of Remark \ref{r.gd} for Cartesian
$(\cA,\cB)$-squares, and the following condition
\begin{description}
\item[(n)] The following square commutes
  \[\xymatrix{F_\cB\res\cA\cap \cB\ar[r]^{\rho}\ar[d]_{\alpha_\cB\res\cA\cap\cB} & F_\cA\res\cA\cap\cB\ar[d]^{\alpha_\cA\res\cA\cap \cB}\\
  F'_\cB\res\cA\cap \cB\ar[r]^{\rho'} & F'_\cA\res\cA\cap\cB.}\]
\end{description}

A \emph{$2$-cell} of $\GD^\Cart_{\cA,\cB}(\cC,\cD)$ is a pair
$(\Xi_\cA,\Xi_\cB)\colon (\alpha_\cA,\alpha_\cB)\Rightarrow
(\alpha'_\cA,\alpha'_\cB)$ consisting of a $2$-cell $\Xi_\cA\colon
\alpha_\cA\Rightarrow \alpha'_\cA$ of $\PsFun(\cA,\cD)$ and a $2$-cell
$\Xi_\cB\colon \alpha_\cB\Rightarrow \alpha'_\cB$ of $\PsFun(\cB,\cD)$ such
that $\lvert \Xi_\cA\rvert = \lvert \Xi_\cB \rvert$.

We view $\GD^\Cart_{\cA,\cB}(\cC,\cD)$ as a $\cD^{\Ob(\cC)}$-$2$-category
via the $2$-functor given by
\[
  (F_\cA,F_\cB,G,\rho)\mapsto \lvert F_\cA \rvert = \lvert F_\cB \rvert,\quad
  (\alpha_\cA,\alpha_\cB)\mapsto \lvert \alpha_\cA\rvert =\lvert \alpha_\cB\rvert,\quad
  (\Xi_\cA,\Xi_\cB)\mapsto \lvert \Xi_\cA \rvert =\lvert \Xi_\cB \rvert.
\]
\end{Definition}

Construction \ref{s.gdp} defines a $\cD^{\Ob(\cC)}$-$2$-functor
  \begin{equation}\label{e.gdp}
  \GD_{\cA,\cB}(\cC,\cD) \to \GD^\Cart_{\cA,\cB}(\cC,\cD),
  \end{equation}
  which is clearly $2$-faithful.

\begin{Remark}\label{r.ff}
If $(\cA,\cB)$ is squaring in $\cC$ (Definition \ref{s.square}), then
\eqref{e.gdp} is $2$-fully faithful. In fact, for objects $(F_\cA,F_\cB,G)$
and $(F'_\cA,F'_\cB,G')$ of $\GD_{\cA,\cB}(\cC,\cD)$ and any morphism
  \[(\alpha_A,\alpha_B)\colon (F_\cA,F_\cB,G,\rho)\to (F'_\cA,F'_\cB,G',\rho')\]
of $\GD^\Cart_{\cA,\cB}(\cC,\cD)$ whose source and target are respectively
the images of $(F_\cA,F_\cB,G)$ and $(F'_\cA,F'_\cB,G')$ under
\eqref{e.gdp}, the pair $(\alpha_\cA,\alpha_\cB)\colon
(F_\cA,F_\cB,G)\to(F'_\cA,F'_\cB,G')$ is a morphism of
$\GD_{\cA,\cB}(\cC,\cD)$. Indeed, for any $(\cA,\cB)$-square $D$
\eqref{e.2}, decomposing it as \eqref{e.square}, we see that the following
diagram commutes
\[\xymatrix{\alpha_0(W)F_\cA(i)F_\cB(q)\ar@=[r]_-{F_\cB(\gamma)}\ar@/^1.5pc/@=[rrr]^{G_D}\ar@=[d]_{\alpha_{\cA}(i)}
  & \alpha_0(W)F_\cA(i)F_\cB(r)F_\cB(f)\ar@=[r]^{G_{D'}\rho(f)}\ar@=[d]^{\alpha_\cA(i)}
  & \alpha_0(W)F_\cB(p)F_\cA(k)F_\cA(f)\ar@=[r]_-{F_{\cA}(\delta)}\ar@=[d]^{\alpha_\cB(p)}
  & \alpha_0(W)F_\cB(p)F_\cA(j)\ar@=[d]^{\alpha_{\cB}(p)}\\
  F'_\cA(i)\alpha_0(Z)F_\cB(q)\ar@=[dd]_{\alpha_\cB(q)}\ar@=[r]^{F_\cB(\gamma)}
  & F'_\cA(i)\alpha_0(Z)F_\cB(r)F_\cB(f)\ar@=[d]^{\alpha_{\cB}(r)}
  & F'_\cB(p)\alpha_0(Y)F_\cA(k)F_\cA(f)\ar@=[d]^{\alpha_\cA(k)}\ar@=[r]^{F_\cA(\delta)}
  & F'_\cB(p)\alpha_0(Y)F_\cA(j)\ar@=[dd]^{\alpha_\cA(j)}\\
  & F'_\cA(i)F'_\cB(r)\alpha_0(X')F_\cB(f) \ar@=[r]^{G'_{D'}\rho(f)}\ar@=[d]^{\alpha_\cB(f)}
  & F'_\cB(p)F'_\cA(k)\alpha_0(X')F_\cA(f)\ar@=[d]^{\alpha_\cA(f)}\\
  F'_\cA(i)F'_\cB(q)\alpha_0(X)\ar@=[r]^-{F'_\cB(\gamma)}\ar@/_1.5pc/@=[rrr]_{G'_D}
  & F'_\cA(i)F'_\cB(r)F'_\cB(f)\alpha_0(X)\ar@=[r]^{G'_{D'}\rho'(f)}
  & F'_\cB(p)F'_\cA(k)F'_\cA(f)\alpha_0(X)\ar@=[r]^-{F'_{\cA}(\delta)}
  & F'_\cB(p)F'_\cA(j)\alpha_0(X).}\]
Here $\alpha_0=\lvert \alpha_\cA \rvert = \lvert \alpha_\cB \rvert$, and
$D'$ is the inner square of \eqref{e.square}.
\end{Remark}

\begin{Theorem}\label{p.gdp}
Let $\cC$ be a $(2,1)$-category and let $\cA$ and $\cB$ be locally full
sub-$2$-categories of~$\cC$, each containing all objects of $\cC$. Let $\cD$
be a $2$-category. Assume that every equivalence in $\cC$ is contained in
$\cA\cap\cB$, and the pairs $(\cA,\cB)$, $(\cA,\cA\cap \cB)$,
$(\cB,\cA\cap\cB)$ are squaring in $\cC$ (Definition \ref{s.square}). Then
\eqref{e.gdp} is an isomorphism of $\cD^{\Ob(\cC)}$-$2$-categories.
\end{Theorem}

\begin{proof}
We construct the inverse of \eqref{e.gdp} as follows. Let
$(F_\cA,F_\cB,G,\rho)$ be an object of $\GD^\Cart_{\cA,\cB}(\cC,\cD)$. For
any square $D$ \eqref{e.2}, decompose it as \eqref{e.square},
  and denote the inner square by $D'$. Let $\bar{G}_D$ be the composition
  \[F_\cA(i)F_\cB(q)\xRightarrow{F_\cB(\gamma)}
   F_\cA(i)F_\cB(r)F_\cB(f)\xRightarrow{G_{D'}\rho(f)}
  F_\cB(p)F_\cA(k)F_\cA(f)\xRightarrow{F_{\cA}(\delta)}
  F_\cB(p)F_\cA(j).
  \]
  This does not depend on the choice of the decomposition. In fact, if
  \[
  \xymatrix{X\ar[rd]^g\xuppertwocell[rrd]{}^j{^\zeta}\xlowertwocell[rdd]{}_q{^\epsilon}\\
  &X''\ar[r]^l\ar[d]^s\ar@{}[rd]|{D''} & Y\ar[d]^p\\
  &Z\ar[r]^i & W}
  \]
  is another decomposition with $l$ in $\cA$, $s$ in $\cB$, $g$ in $\cA\cap \cB$, and $D''$ Cartesian in $\cC$, then they can be combined into
  \[\xymatrix{X\ar[rd]^{g}\xlowertwocell[rrdd]{}_f{^\omega}\xuppertwocell[rrrdd]{}<12>^j{^\zeta}\xlowertwocell[dddrr]{}<-12>_q{^\epsilon} \\
  &X''\ar[rd]^h\xuppertwocell[rrd]{}^l{^\psi} \xlowertwocell[ddr]{}_s{^\phi}\\
  &&X'\ar[r]^k\ar[d]^r\ar@{}[rd]|{D'} & Y\ar[d]^p\\
  &&Z\ar[r]^i & W, }\]
  where $h$ is an equivalence. Applying the axioms of Remarks \ref{r.gd} and \ref{r.cgd} to the following decomposition of $D''$
  \[\xymatrix{X''\ar[rr]^l\ar@{=}[d]\xtwocell[rrd]{}\omit{^\psi} && Y\ar@{=}[d]\\
  X''\ar[r]^h\ar[d]_s\drtwocell\omit{^\phi} & X'\ar[r]^k\ar[d]^r\ar@{}[rd]|{D'} & Y\ar[d]^p\\
  Z\ar@{=}[r] & Z\ar[r]^i &W,}\]
  we obtain the following commutative diagram
  \[\xymatrix{F_\cA(i)F_\cB(s)\ar@=[rr]^{G_{D''}}_{\text{(b),(c)}}\ar@=[rrd]_{\text{(\bprime)}}\ar@=[rd]_{\text{(c)}}\ar@=[d]_{F_\cB(\phi)}
  && F_\cB(p)F_\cA(l)\\
  F_\cA(i)F_\cB(r)F_\cB(h)\ar@=[r]_{\rho(h)} & F_\cA(i)F_\cB(r)F_\cA(h)\ar@=[r]_{G_{D'}} & F_\cB(p)F_\cA(k)F_\cA(h).\ar@=[u]_{F_\cA(\psi)}}
  \]
  Hence the following diagram commutes
  \[\xymatrix{F_\cA(i)F_\cB(q)\ar@=[d]_{F_\cB(\epsilon)}\ar@=[rrd]^{F_\cB(\gamma)}\\
  F_\cA(i)F_\cB(s)F_\cB(g)\ar@=[d]_{\rho(g)}\ar@=[r]^-{F_\cB(\phi)}
  & F_\cA(i)F_\cB(r)F_\cB(h)F_\cB(g)\ar@=[d]^{\rho(g)}\ar@=@/^5pc/[dd]^{\rho(hg)}
  & F_\cA(r)F_\cB(r)F_\cB(f)\ar@=[dd]^{\rho(f)}\ar@=[l]^-{F_\cB(\omega)}\\
  F_\cA(i)F_\cB(s)F_\cA(g)\ar@=[dd]_{G_{D''}}\ar@=[r]^-{F_\cB(\phi)}
  & F_\cA(i)F_\cB(r)F_\cB(h)F_\cA(g)\ar@=[d]^{\rho(h)}\\
  & F_\cA(i)F_\cB(r)F_\cA(h)F_\cA(g)\ar@=[d]^{G_{D'}}
  & F_\cA(i)F_\cB(r)F_\cA(f)\ar@=[d]^{G_{D'}}\ar@=[l]^-{F_\cA(\omega)}\\
  F_\cB(p)F_\cA(l)F_\cA(g)\ar@=[d]_{F_\cA(\zeta)}
  & F_\cB(p)F_\cA(k)F_\cA(h)F_\cA(g)\ar@=[l]_-{F_\cA(\psi)}
  & F_\cB(p)F_\cA(k)F_\cA(f).\ar@=[l]_-{F_\cA(\omega)}\ar@=[lld]^{F_\cA(\delta)}\\
  F_\cB(p)F_\cA(j)}\]

Next we show that $(F_\cA,F_\cB,\bar{G})$ is an object of
$\GD_{\cA,\cB}(\cC,\cD)$. Axioms (a) and (\aprime) of Remark \ref{r.gd} for
$\bar{G}$ follow from axioms (c) and (\cprime) of Remark \ref{r.cgd}. Let
$D$, $D'$, and $D''$ be squares as in axiom (\bprime) of Remark \ref{r.gd}
for~$\bar{G}$. Decompose it as
  \[\xymatrix{X_1\xuppertwocell[rrrd]{}^j{^\epsilon}\ar[dr]^f\xlowertwocell[dddrr]{}<-12>_{p_1}{^\delta}\\
  &W\xtwocell[rrrd]{}\omit{^\eta} \ar[rr]^l\ar[dr]^g & & X_2\ar[rd]^h\xuppertwocell[rrd]{}^{j'}{^\gamma}\xlowertwocell[ddr]{}_{p_2}{^\beta}\\
  && Z_1\ar[rr]^(.4)k\ar[d]^{q_1}\ar@{}[rrd]|E && Z_2\ar[r]^{k'}\ar[d]^{q_2}\ar@{}[dr]|{E'} & X_3\ar[d]^{p_3}\\
  && Y_1\ar[rr]^i && Y_2\ar[r]^{i'} & Y_3,}\]
where horizontal arrows are morphisms of $\cA$, vertical arrows are
morphisms of $\cB$, oblique arrows are morphisms of $\cA\cap \cB$, the
squares $E$, $E'$ and the square $H$ containing $\eta$ are Cartesian in
$\cC$. Let $E''=E'\circ E$, $I=H\circ E$. Since $I$ is the outer square of
the diagram
  \[\xymatrix{W\ar[r]^l\ar[d]_{q_1 g}\ar@{}[rd]|J & X_2\ar@{=}[r]\ar[d]^{p_2}\drtwocell\omit{^\beta} & X_2\ar[d]^{q_2 h}\\
  Y_1\ar[r]^i & Y_2\ar@{=}[r] & Y_2,}\]
axiom (\bprime)  of Remark \ref{r.gd} and axiom (\cprime) of Remark
\ref{r.cgd} imply the commutativity of the following triangle
  \[\xymatrix{F_\cA(i)F_\cB(q_1 g)\ar@=[d]_{G_J} \ar@=[rd]^{G_I}\\
  F_\cB(p_2)F_\cA(l)\ar@=[r]^{F_\cB(\beta)} & F_\cB(q_2 h) F_\cA(l).}\]
  It follows that the following diagram commutes
  {\scriptsize\[\rotatebox{-90}
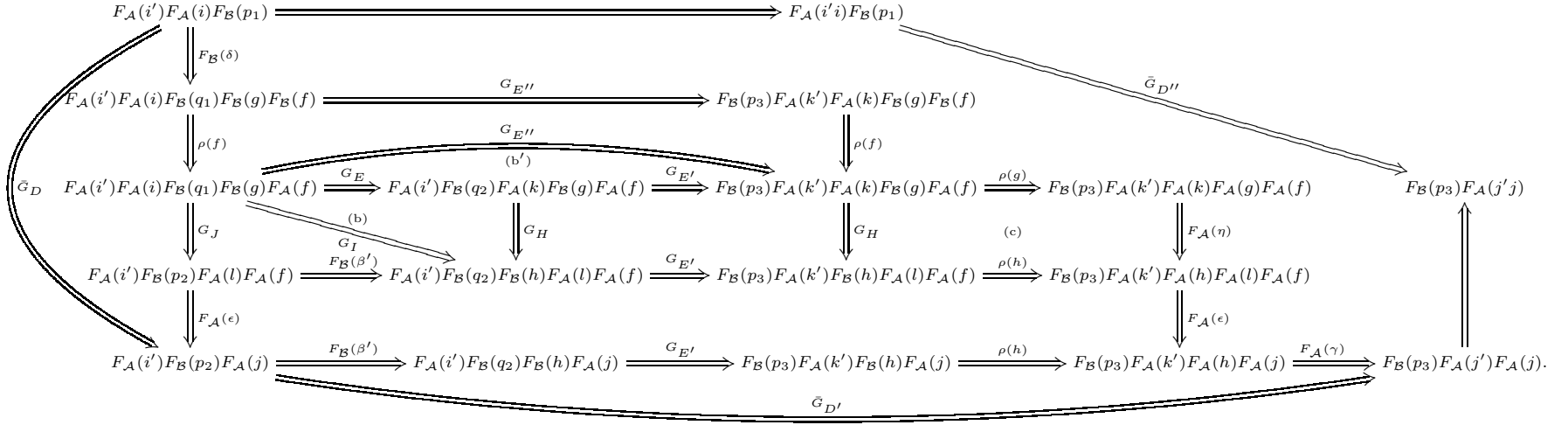
{
  \xymatrix{F_\cA(i')F_\cA(i)F_\cB(p_1)\ar@=[d]^{F_\cB(\delta)}\ar@=[rr]\ar@=@/_6.5pc/[dddd]^{\bar{G}_D}
  && F_\cA(i'i)F_\cB(p_1)\ar@=[rrdd]^{\bar{G}_{D''}}\\
  F_\cA(i')F_\cA(i)F_\cB(q_1)F_\cB(g)F_\cB(f)\ar@=[rr]^{G_{E''}}\ar@=[d]^{\rho(f)}
  &&F_\cB(p_3)F_\cA(k')F_\cA(k)F_\cB(g)F_\cB(f)
  \ar@=[d]^{\rho(f)}\\
  F_\cA(i')F_\cA(i)F_\cB(q_1)F_\cB(g)F_\cA(f)\ar@=[d]^{G_J}\ar@=[r]^{G_E}\ar@=[rd]_{G_I}^{\text{(b)}}\ar@=@/^1.5pc/[rr]^{G_{E''}}_{\text{(\bprime)}}
  &F_\cA(i')F_\cB(q_2)F_\cA(k)F_\cB(g)F_\cA(f)\ar@=[d]^{G_H}\ar@=[r]^{G_{E'}} &
  F_\cB(p_3)F_\cA(k')F_\cA(k)F_\cB(g)F_\cA(f)\ar@=[d]^{G_H}\ar@=[r]^{\rho(g)}\ar@{}[rd]|{\text{(c)}} &
  F_\cB(p_3)F_\cA(k')F_\cA(k)F_\cA(g)F_\cA(f)\ar@=[d]^{F_\cA(\eta)}
  & F_\cB(p_3)F_\cA(j'j)\\
  F_\cA(i')F_\cB(p_2)F_\cA(l)F_\cA(f)\ar@=[d]^{F_\cA(\epsilon)}\ar@=[r]^{F_\cB(\beta')} &
  F_\cA(i')F_\cB(q_2)F_\cB(h)F_\cA(l)F_\cA(f)\ar@=[r]^{G_{E'}}
  &F_\cB(p_3)F_\cA(k')F_\cB(h)F_\cA(l)F_\cA(f)\ar@=[r]^{\rho(h)}
  &F_\cB(p_3)F_\cA(k')F_\cA(h)F_\cA(l)F_\cA(f) \ar@=[d]^{F_\cA(\epsilon)}\\
  F_\cA(i')F_\cB(p_2)F_\cA(j)\ar@=[r]^{F_\cB(\beta')}\ar@=@/_2pc/[rrrr]^{\bar{G}_{D'}}
  & F_\cA(i')F_\cB(q_2)F_\cB(h)F_\cA(j)\ar@=[r]^{G_{E'}}
  & F_\cB(p_3)F_\cA(k')F_\cB(h)F_\cA(j)\ar@=[r]^{\rho(h)}
  & F_\cB(p_3)F_\cA(k')F_\cA(h)F_\cA(j)\ar@=[r]^{F_\cA(\gamma)} & F_\cB(p_3)F_\cA(j')F_\cA(j).\ar@=[uu]}
  }
  \]
}One establishes axiom (b)  of Remark \ref{r.gd} for $\bar{G}$ in a similar
way.

Let $(\alpha_\cA,\alpha_\cB)\colon (F_\cA,F_\cB,G,\rho)\to
(F'_\cA,F'_\cB,G',\rho')$ be a morphism of $\GD^\Cart_{\cA,\cB}(\cC,\cD)$.
Then $(\alpha_\cA,\alpha_\cB)\colon (F_\cA,F_\cB,\bar{G})\to
(F'_\cA,F'_\cB,\bar{G}')$ is a morphism of $\GD_{\cA,\cB}(\cC,\cD)$ by
Remark \ref{r.ff}. Let
  \[(\Xi_\cA,\Xi_\cB)\colon (\alpha_\cA,\alpha_\cB)\Rightarrow (\alpha'_\cA,\alpha'_\cB)\]
be a $2$-cell of $\GD^\Cart_{\cA,\cB}(\cC,\cD)$. Then $(\Xi_\cA,\Xi_\cB)$ is
a $2$-cell of $\GD_{\cA,\cB}(\cC,\cD)$.

The $2$-functor defined in this way is clearly the inverse of \eqref{e.gdp}.
\end{proof}

Theorem \ref{p.gdp} relates gluing data and Cartesian gluing data for two
pseudo functors. In the next section we will establish an analogue for
finitely many pseudo functors.

\section{Cartesian gluing data for finitely many pseudo functors}\label{s.7}
In this section, we generalize the definitions and results of the previous
section to the case of finitely many pseudo functors. The main result is a
general criterion for $\GD^\Cart$ and $\GD$ to be isomorphic (Theorem
\ref{p.gdp'}).

Let $\cC$ be a $(2,1)$-category, let $\cA_1,\dots, \cA_n$ be locally full
sub-$2$-categories of $\cC$, each containing all objects of $\cC$. Let $\cD$
be a $2$-category.

\begin{Remark}[Properties of the pseudo natural isomorphisms $\rho_{ij}$ associated to a gluing
datum]\label{s.gdp'}\index{rhoij@$\rho_{ij}$} Let $((F_i),G)$ be an object
of $\GD_{\cA_1,\dots,\cA_n}(\cC,\cD)$. For $1\le i,j\le n$, the triple
$(F_i,F_j,G_{ij})$ is an object of $\GD_{\cA_i,\cA_j}(\cC,\cD)$. Let
  \[\rho_{ij}\colon F_j\res\cA_i\cap\cA_j \to F_i\res\cA_i\cap \cA_j\]
be the isomorphism in $\PsFun(\cA_i\cap \cA_j,\cD)$ associated to the triple
by \eqref{e.gdp}. Then $\rho_{ii}=\one_{F_i}$ and
$\rho_{ji}=\rho_{ij}^{-1}$. We claim that $\rho_{ij}$ has the following
properties:
\begin{description}
\item[(E)] For $1\le i,j,k\le n$ and any $(\cA_i,\cA_j\cap \cA_k)$-square
    $D$ \eqref{e.2'}, the following square commutes
    \[\xymatrix{F_i(a)F_k(q)\ar@=[r]^{\rho_{jk}(q)}\ar@=[d]_{G_{ikD}} & F_i(a)F_j(q)\ar@=[d]^{G_{ijD}}
    \\
    F_k(p)F_i(b)\ar@=[r]^{\rho_{jk}(p)} & F_j(p)F_i(b).
    }
    \]

\item[(F)] (cocycle condition) For $1\le i,j,k\le n$, the following
    triangle commutes
    \[\xymatrix{F_k\res\cA_{ijk} \ar[r]^{\rho_{jk}\res\cA_{ijk}}\ar[rd]_{\rho_{ik}\res\cA_{ijk}} & F_j\res\cA_{ijk}\ar[d]^{\rho_{ij}\res\cA_{ijk}}\\
    &F_i\res\cA_{ijk}.}
    \]
    Here $\cA_{ijk}=\cA_i\cap\cA_j\cap\cA_k$.
\end{description}
In fact, (E) follows from axiom (D) of Remark \ref{r.gdn} applied to the
cube
  \[\xymatrix{X\ar[rr]^b\ar[rd]^q\ar[dd]_q && Y\ar[rd]^p\ar[dd]^(.3)p|\hole\\
  &Z\ar@{=}[dd]\ar[rr]^(.3)a && W\ar@{=}[dd]\\
  Z\ar[rr]^(.3)a|\hole\ar@{=}[rd] && W\ar@{=}[rd]\\
  &Z\ar[rr]^a&& W}\]
whose top and back faces are $D$ and whose other faces have identity
$2$-cells. Condition (F) follows from (E) applied to the square
  \[\xymatrix{X\ar[r]^f \ar[d]_f\drtwocell\omit{=} & Y\ar@{=}[d]\\
  Y\ar@{=}[r] & Y}\]
  for every morphism $f$ of $\cA_{ijk}$.
\end{Remark}

We extend Definition \ref{d.gdp} as follows.

\begin{Definition}[Cartesian gluing data for finitely many pseudo
functors]\label{d.gdp'}\index{GDCart@$\GD^\Cart_{\cA_1,\dots,\cA_n}(\cC,\cD)$}
We define a $2$-category $\GD^\Cart_{\cA_1,\dots,\cA_n}(\cC,\cD)$ as
follows. An \emph{object} of this $2$-category is a triple
$\left((F_i)_{1\le i \le n},(G_{ij})_{1\le i<j\le n},(\rho_{ij})_{1\le
i<j\le n}\right)$, where $F_i\colon \cA_i\to \cD$ is an object of
$\PsFun(\cA_i,\cD)$ for $1\le i\le n$, and $(F_i,F_j,G_{ij},\rho_{ij})$ is
an object of $\GD^\Cart_{\cA_i,\cA_j}(\cC,\cD)$ for $1\le i<j\le n$,
satisfying condition (D) of Remark \ref{r.gdn} for $1\le i<j<k\le n$ and
cubes with Cartesian faces, condition (E) of Remark \ref{s.gdp'} for $1\le
i\le n$, $1\le j<k\le n$, $i\neq j,k$ and Cartesian squares (here we put
$G_{ij}=G_{ji}^*$ for $i>j$), and condition (F) of Remark \ref{s.gdp'} for
$1\le i<j<k\le n$.

A \emph{morphism} $\left((F_i),G,\rho\right)\to
\left((F'_i),G',\rho'\right)$ of $\GD^\Cart_{\cA_1,\dots,\cA_n}(\cC,\cD)$ is
a collection
  $(\alpha_i)_{1\le i\le n}$
  of morphisms $\alpha_i\colon F_i\to F'_i$ of $\PsFun(\cA_i,\cD)$, such that for $1\le i<j\le n$,
 the pair $(\alpha_i,\alpha_j)\colon
  (F_i,F_j,G_{ij},\rho_{ij})\to(F'_i,F'_j,G'_{ij},\rho'_{ij})$
  is a morphism of $\GD^\Cart_{\cA_i,\cA_j}(\cC,\cD)$.

A \emph{$2$-cell} of $\GD^\Cart_{\cA_1,\dots,\cA_n}(\cC,\cD)$ is a
collection
  $(\Xi_i)_{1\le i\le n}\colon (\alpha_i)_{1\le i \le n}\Rightarrow (\alpha'_i)_{1\le i \le
  n}$
  of $2$-cells $\Xi_i\colon \alpha_i\Rightarrow \alpha'_i$ of $\PsFun(\cA_i,\cD)$ such that $\lvert \Xi_1 \rvert =\dots = \lvert \Xi_n \rvert$.

  We view $\GD^\Cart_{\cA_1,\dots, \cA_n}(\cC,\cD)$ as a $\cD^{\Ob(\cC)}$-$2$-category via the $2$-functor given by
\[
    ((F_i),G,\rho)\mapsto \lvert F_1\rvert =\dots=\lvert F_n\rvert,\quad
    (\alpha_i)\mapsto \lvert \alpha_1\rvert = \dots=\lvert \alpha_n\rvert , \quad (\Xi_i)\mapsto \lvert \Xi_1\rvert = \dots = \lvert \Xi_n \rvert.
\]
\end{Definition}

Remark \ref{s.gdp'} defines a $\cD^{\Ob(\cC)}$-$2$-functor
  \begin{equation}\label{e.gdp'}
    \GD_{\cA_1,\dots,\cA_n}(\cC,\cD)\to \GD^\Cart_{\cA_1,\dots,\cA_n}(\cC,\cD),
  \end{equation}
which is clearly $2$-faithful. If, for all $1\le i<j\le n$, the pair
$(\cA_i,\cA_j)$ is squaring in $\cC$ (Definition \ref{s.square}), then
\eqref{e.gdp'} is $2$-fully faithful by Remark \ref{r.ff}.

The following generalizes Theorem \ref{p.gdp}.

\begin{Theorem}\label{p.gdp'}
Let $\cC$ be a $(2,1)$-category and let $\cA_1,\dots, \cA_n$ be locally full
sub-$2$-categories of $\cC$, each containing all objects of $\cC$. Let $\cD$
be a $2$-category. Assume that every morphism of $\cC$ that is an
equivalence is contained in $\cA_1\cap \dots \cap\cA_n$, and for $1\le
i,j\le n$, $i\ne j$, the pairs $(\cA_i,\cA_j)$ and $(\cA_i,\cA_i\cap \cA_j)$
are squaring in $\cC$ (Definition \ref{s.square}). Assume moreover that for
pairwise distinct numbers $1\le i,j,k\le n$, the pair $(\cA_i,\cA_j)$ is
$\cA_k$-squaring and the pair $(\cA_i\cap\cA_j,\cA_k)$ is squaring. Then
\eqref{e.gdp'} is an isomorphism of $\cD^{\Ob(\cC)}$-$2$-categories.
\end{Theorem}

One sufficient condition for the assumptions of Theorem \ref{p.gdp'} to hold
true is that $\cC$ admits pseudo fiber products, and $\cA_i$ is stable under
base change in $\cC$ and taking diagonals in $\cC$ for all $1\le i\le n$.

\begin{proof}
We construct the inverse of \eqref{e.gdp'} as follows. Let $((F_i),G,\rho)$
be an object of $\GD^\Cart_{\cA_1,\dots,\cA_n}(\cC,\cD)$. For $1\le i,j\le
n$, $i\neq j$, let $(F_i,F_j,\bar{G}_{ij})$ be the image of
$(F_i,F_j,G_{ij},\rho_{ij})$ under the inverse of \eqref{e.gdp}. To show
that $((F_i),\bar G)$ is an object of $\GD_{\cA_1,\dots,\cA_n}(\cC,\cD)$, it
suffices to check axiom (D) of Remark \ref{r.gdn} for $1\le i<j<k \le n$.

First note that for pairwise distinct numbers $1\le i,j,k \le n$, $G$
satisfies axiom (D)  of Remark \ref{r.gdn} for cubes with Cartesian faces,
$G$ and $\rho$ satisfy axiom (E) of Remark \ref{s.gdp'} for Cartesian
squares, and $\rho$ satisfies axiom (F) of Remark \ref{s.gdp'}. Here for
$1\le i<j\le n$ we put $G_{ji}=G_{ij}^*$ and $\rho_{ji}=\rho_{ij}^{-1}$.
Next we show that $\bar{G}$ satisfies axiom (E) of Remark \ref{s.gdp'} for
pairwise distinct numbers $1\le i,j,k \le n$ and all $(\cA_i,\cA_j\cap
\cA_k)$-squares $D$ \eqref{e.2'}. Decompose $D$ as
  \[\xymatrix{X\ar[rd]^f\xuppertwocell[rrd]{}^b{^\delta}\xlowertwocell[rdd]{}_q{^\gamma}\\
  &X'\ar[r]^c\ar[d]^r\drtwocell\omit{^\beta} & Y\ar[d]^p\\
  &Z\ar[r]^a & W}\]
where $c$ is a morphism of $\cA_i$, $r$ is a morphism of $\cA_j\cap \cA_k$,
$f$ is a morphism of $\cA_i\cap\cA_j\cap \cA_k$, and the inner square $D'$
is Cartesian. Then the following diagram commutes
\[\xymatrix{F_i(a)F_k(q)\ar@=[d]_{\rho_{jk}(q)}\ar@=[r]_-{F_k(\gamma)}\ar@=@/^2pc/[rrrr]^{\bar{G}_{ikD}} & F_i(a)F_k(r)F_k(f)\ar@=[d]_{\rho_{jk}(r)\rho_{jk}(f)}\ar@=[r]^{\rho_{ik}(f)}\ar@{}[rd]|{\text{(F)}} & F_i(a)F_k(r)F_i(f)\ar@=[r]^{G_{ikD'}}\ar@=[d]^{\rho_{jk}(r)}\ar@{}[rd]|{\text{(E)}} & F_k(p)F_i(c)F_i(f)\ar@=[d]^{\rho_{jk}(p)}\ar@=[r]_-{F_i(\delta)} & F_k(p)F_i(b)\ar@=[d]^{\rho_{jk}(p)}\\
  F_i(a)F_j(q)\ar@=[r]^-{F_j(\gamma)}\ar@=@/_2pc/[rrrr]^{\bar{G}_{ijD}} & F_i(a)F_j(r)F_j(f)\ar@=[r]^{\rho_{ij}(f)} & F_i(a)F_j(r)F_i(f)\ar@=[r]^{G_{ijD'}} & F_j(p)F_i(c)F_i(f)\ar@=[r]^-{F_i(\delta)} & F_j(p)F_i(b).}\]

For pairwise distinct numbers $1\le i,j,k \le n$, we show axiom (D) of
Remark \ref{r.gdn} for $\bar{G}$ by descending induction on the number $m$
of pairs of Cartesian opposite faces in the cube \eqref{e.3}. If $m=3$, all
the faces of the cube are Cartesian, so the assertion is identical to axiom
(D) for $G$. If $m<3$, by symmetry, we may assume that either the bottom
face $K$ or the top face $K'$ is not Cartesian. Decompose the cube as
\eqref{e.cubing}. The inner cube has more than $m$ pairs of Cartesian
opposite faces, hence axiom (D) holds for the inner cube by induction
hypothesis. Therefore, the following diagram commutes
  {\scriptsize\[\xymatrix{F_i(a)F_j(q)F_k(x)\ar@=[rr]^{\bar{G}_{jkI'}}\ar@=[d]^{F_j(\gamma)}\ar@=@/_5.5pc/[dddd]^{\bar{G}_{ijK}}
  && F_i(a)F_k(z)F_j(q') \ar@=[r]^{\bar{G}_{ikJ}}\ar@=[d]^{F_j(\gamma')}
  & F_k(w)F_i(a')F_j(q')\ar@=[d]_{F_j(\gamma')}\ar@=@/^5.5pc/[dddd]^{\bar{G}_{ijK'}} \\
  F_i(a)F_j(r)F_j(f)F_k(x) \ar@=[r]^{\bar{G}_{jkM}}\ar@=[d]^{\rho_{ij}(f)}\ar@{}[rd]|{\text{(E)}}
  & F_i(a)F_j(r)F_k(v)F_j(f') \ar@=[r]^{\bar{G}_{jkI''}}\ar@=[d]^{\rho_{ij}(f')}
  & F_i(a)F_k(z)F_j(r')F_j(f') \ar@=[r]^{\bar{G}_{ikJ}}
  & F_k(w)F_i(a')F_j(r')F_j(f')\ar@=[d]^{\rho_{ij}(f')}\\
  F_i(a)F_j(r)F_i(f)F_k(x) \ar@=[r]^{\bar{G}_{ikM}}\ar@=[d]^{G_{ijL}}
  & F_i(a)F_j(r)F_k(v)F_i(f')\ar@=[r]^{\bar{G}_{jkI''}}\ar@=[d]^{G_{ijL}}\ar@{}[rrd]|{\text{(D)}}
  & F_i(a)F_k(z)F_j(r')F_i(f')\ar@=[r]^{\bar{G}_{ikJ}}
  & F_k(w)F_i(a')F_j(r')F_i(f') \ar@=[d]^{G_{ijL'}}\\
  F_j(p)F_i(c)F_i(f')F_k(x)\ar@=[d]^{F_i(\delta)}\ar@=[r]^{\bar{G}_{ikM}}
  & F_j(p)F_i(c)F_k(v)F_i(f')\ar@=[r]^{\bar{G}_{ikJ''}}
  & F_j(p)F_k(y)F_i(c')F_i(f')\ar@=[r]^{\bar{G}_{jkI}}\ar@=[d]^{F_i(\delta')}
  & F_k(w)F_j(p')F_i(c')F_i(f')\ar@=[d]_{F_i(\delta')}\\
  F_j(p)F_i(b)F_k(x)\ar@=[rr]^{\bar{G}_{ikJ'}}
  && F_j(p)F_k(y)F_i(b')\ar@=[r]^{\bar{G}_{jkI}}
  & F_k(w)F_j(p')F_i(b').}\]
  }Here $M$ is the square $X'V'XV$.

  Any morphism
  $(\alpha_i)\colon ((F_i),G,\rho)\to ((F'_i),G',\rho')$
  of $\GD^\Cart_{\cA_1,\dots,\cA_n}(\cC,\cD)$ is a morphism $((F_i),\bar{G})\to ((F'_i),\bar{G}')$ of $\GD_{\cA_1,\dots,\cA_n}(\cC,\cD)$. Any $2$-cell $(\Xi_i)\colon (\alpha_i)\Rightarrow (\alpha'_i)$ of $\GD^\Cart_{\cA_1,\dots,\cA_n}(\cC,\cD)$ is a $2$-cell of $\GD_{\cA_n,\dots,\cA_n}(\cC,\cD)$.

The $2$-functor defined in this way is clearly the inverse of
\eqref{e.gdp'}.
\end{proof}

\begin{Remark}\label{r.inf}
We can consider a simpler $2$-category of gluing data
\[\GD^\Cart_{\cA_1,\dots,\cA_n}(\cC,\cD)'=\twoFunps(\cL\cT_n\dQ^\Cart_{\cA_1,\dots,\cA_n}\cC,\cD)\]
by dropping $(\rho_{ij})$ from the definition of $\GD^\Cart$, where
$\dQ^\Cart_{\cA_1,\dots,\cA_n}\cC\subseteq \dQ_{\cA_1,\dots,\cA_n}\cC$ is
the $n$-fold subcategory spanned by Cartesian squares. An
$\infty$-categorical variant of $\dQ^\Cart_{\cA_1,\dots,\cA_n}\cC$ is
studied in \cite[Section~5]{LZ1}.
\end{Remark}

In this section and the previous one, we studied the relationship between
gluing data and Cartesian gluing data. In the next section we discuss ways
to construct Cartesian gluing data.

\section{Gluing data from base change maps}\label{s.adjoint}
In this section we show how to produce (Cartesian) gluing data for two
sub-$2$-categories from base change maps (Constructions \ref{c.bc} and
\ref{s.bcA}) by taking adjoints. We deal with the axioms individually and
refer the reader to Remark \ref{r.gdadj} for a synthesis. A somewhat more
systematic treatment is possible in the $\infty$-categorical setting (for
the gluing data mentioned in Remark \ref{r.inf}) \cite[Section 1.4]{LZ2}.
Throughout this section, we fix a $2$-category $\cD$.

\begin{Definition}[$2$-Category of adjoint pairs
$\cD^{\adj}$]\label{s.adj}\index{-adj@$(-)^{\adj}$!$\cD^{\adj}$} We define
the \emph{$2$-category $\cD^\adj$ of adjoint pairs} with the same objects as
$\cD$ by taking $\cD^\adj(X,Y)$ to be the category of adjoint pairs from $X$
to $Y$ for every pair $(X,Y)$ of objects of~$\cD$. More explicitly, a
\emph{morphism} $X\to Y$ of $\cD^\adj$ is a quadruple $(f,g,\eta,\epsilon)$
consisting of morphisms $f\colon X\to Y$, $g\colon Y\to X$ and $2$-cells
$\eta\colon \one_Y\Rightarrow fg$, $\epsilon\colon gf\Rightarrow \one_X$
of~$\cD$ such that the following triangles commute
  \[\xymatrix{f\ar@=[r]^{\eta *f}\ar@=[rd]_{\one_f} & fgf\ar@=[d]^{f*\epsilon} & g\ar@=[r]^{g*\eta}\ar@=[rd]_{\one_g} & gfg\ar@=[d]^{\epsilon*g} \\
  & f & & g.}\]
  The composition of $(f_1,g_1,\eta_1,\epsilon_1)\colon X\to Y$ and $(f_2,g_2,\eta_2,\epsilon_2)\colon Y\to Z$ is
  \[(f_2f_1, g_1g_2,\eta_1\eta_2,\epsilon_1\epsilon_2)\colon X\to Z,\]
  where $\eta_1\eta_2$ is the composition
  \[\one_Z\xRightarrow{\eta_2}  f_2g_2\xRightarrow{f_2*\eta_1*g_2}  f_2f_1g_1g_2\]
  and $\epsilon_1\epsilon_2$ is the composition
  \[g_1g_2f_2f_1\xRightarrow{g_1*\epsilon_2*f_1} g_1f_1\xRightarrow{\epsilon_1} \one_X.\]
  The \emph{identity morphism} of an object $X$ is $(\one_X,\one_X,\one_{\one_X},\one_{\one_X})$. A \emph{$2$-cell} $(f,g,\eta,\epsilon)\Rightarrow (f',g',\eta',\epsilon')$ of $\cD^\adj$ is a pair $(\alpha,\beta)$ of $2$-cells $\alpha\colon f\Rightarrow f'$ and $\beta\colon g'\Rightarrow g$ of $\cD$ such that the following squares commute
  \[\xymatrix{\one_Y\ar@=[r]^{\eta}\ar@=[d]_{\eta'} & fg\ar@=[d]^{\alpha} & g'f\ar@=[d]_{\beta}\ar@=[r]^\alpha & g'f'\ar@=[d]^{\epsilon'}\\
  f'g'\ar@=[r]^{\beta} & f'g & gf \ar@=[r]^\epsilon & \one_X.}\]
The \emph{projection $2$-functors} $P_1\colon \cD^\adj\to \cD$ and
$P_2\colon \cD^\adj\to \cD^\coop$, sending $(f,g,\eta,\epsilon)$ to $f$ and
$g$ respectively, are locally fully faithful (Definition \ref{s.faith}).
\end{Definition}

\begin{Remark}
  Let $\cC$ be a $2$-category.  The projection $2$-functors $P_1$ and $P_2$ induce locally fully faithful $2$-functors
  \[P_1\colon \PsFun(\cC,\cD^\adj)\to \PsFun(\cC,\cD), \quad P_2\colon \PsFun(\cC,\cD^\adj)\to \PsFun(\cC,\cD^\coop).\]
An object $F$ of $\PsFun(\cC,\cD)$ (resp.\ $\PsFun(\cC,\cD^\coop)$) is in
the image of $P_1$ (resp.\ $P_2$) if and only if for every morphism $a$ of
$\cC$, the image $F(a)$ can be completed into an adjoint pair
$(f,F(a),\eta,\epsilon)$ (resp. $(F(a), g, \eta, \epsilon)$).
\end{Remark}

In the rest of this section, we fix a $2$-category $\cC$, a pseudo functor
$F\colon \cC\to \cD^\coop$, a locally full sub-$2$-category $\cB$ of $\cC$,
and a pseudo functor $B\colon \cB\to \cD^\adj$ such that $P_2(B)$ is the
restriction $F\res\cB$. We do \emph{not} assume that $\cC$ is a
$(2,1)$-category. We denote $F$ by\index{-*@$(-)^*$!$f^*$}
  \[f\mapsto f^*, \quad \alpha\mapsto \alpha^*,\]
and the pseudo functor $R=P_1(B)\colon \cB\to \cD$ by
  \[p\mapsto p_*, \quad \alpha\mapsto \alpha_*.\]

\begin{Construction}[Base change map $B_D$ associated to a $(\cC,\cB)$-down-square
$D$]\label{c.bc}\index{BD@$B_D$} In this section, it is convenient to use
down-squares in $\cC$
  \begin{equation}\label{e.-2}
  \xymatrix{X\ar[r]^j\ar[d]_q\drtwocell\omit{\alpha} & Y\ar[d]^p\\
  Z\ar[r]^i & W.}
  \end{equation}
Let $D$ be such a square with $p$ and $q$ in $\cB$. The \emph{base change
map} $B_D$ is by definition the following $2$-cell of $\cD$
  \[
  \xymatrix{i^* p_* \ar@=[r]^-{\eta_q} & q_*q^*i^*p_* \ar@=[r] & q_* (iq)^* p_* \ar@=[r]^{\alpha^*} & q_* (pj)^* p_* \ar@=[r] & q_* j^* p^*p_* \ar@=[r]^-{\epsilon_p} & q_* j^*.}
  \]
Here we have as usual denoted horizontal composition of $2$-cells with
morphisms simply by the $2$-cells.

If $i$ and $j$ are also morphisms of $\cB$, then $B_D$ is also the
composition
  \[\xymatrix{i^*p_*\ar@=[r]^-{\eta_j} & i^*p_*j_*j^*\ar@=[r] & i^* (pj)_* j^* \ar@=[r]^{\alpha_*} & i^* (iq)_* j^* \ar@=[r] & i^* i_* q_* j^* \ar@=[r]^-{\epsilon_i} & q_*j^*.}\]
In fact, the following diagram commutes
  \[\xymatrix{& i^*p_* \ar@=[rr]^{\eta_q} && q_*q^*i^*p_* \ar@=[dd]^{\alpha^*} \\
  i^* p_*\ar@=[r]^{\eta_i}\ar@=[d]^{\eta_p}\ar@=[ru]^{\one}\ar@=@/_3pc/[dd]_{\one} &i^*i_*i^*p_*\ar@=[u]^{\epsilon_i}\ar@=[r]^{\eta_q}
  & i^* i_* q_* q^* i^* p_* \ar@=[d]^{\alpha^*} \\
  i^*p_*p^*p_*\ar@=[d]^{\epsilon_p}\ar@=[r]^{\eta_j} & i^*p_*j_*j^*p^*p_* \ar@=[r]^{\alpha_*} & i^*i_* q_* j^* p^* p_*\ar@=[d]^{\epsilon_p}\ar@=[r]^{\epsilon_i} & q_* j^* p^* p_*\ar@=[d]^{\epsilon_p}\\
  i^* p_* \ar@=[r]^{\eta_j}& i^*p_*j_*j^*\ar@=[r]^{\alpha_*} & i^*i_*q_*j^*\ar@=[r]^{\epsilon_i} & q_* j^*.}\]
\end{Construction}

\begin{Proposition}\label{p.bc}\leavevmode
\begin{enumerate}
\item \label{p.bc1} Let $D$, $D'$, and $D''$ be respectively the upper,
    lower, and outer squares of the diagram in $\cC$
  \[\xymatrix{X_1\ar[r]^{i_1}\ar[d]_{q}\drtwocell\omit{\alpha} &  Y_1\ar[d]^{p}\\
  X_2\ar[r]^{i_2}\drtwocell\omit{\alpha'}\ar[d]_{q'} & Y_2\ar[d]^{p'}\\
  X_3\ar[r]^{i_3} & Y_3}\]
  where the vertical arrows are morphisms of $\cB$. Then the following diagram commutes
    \[\xymatrix{i_3^* p'_* p_*\ar@=[r]^{B_{D'}}\ar@=[d]& q'_*i_2^* p_*\ar@=[r]^{B_{D}}& q'_*q_*i_1^*\ar@=[d]\\
    i_3^*(p'p)_*\ar@=[rr]^{B_{D''}} & & (q'q)_*i_1^*.}\]

\item \label{p.bc2} Let $D$, $D'$, and $D''$ be respectively the left,
    right, and outer squares of the diagram in~$\cC$
  \[\xymatrix{X_1\ar[r]^j\ar[d]_{p_1}\drtwocell\omit{\alpha} &  X_2\ar[d]^{p_2}\ar[r]^{j'}\drtwocell\omit{\alpha'} & X_3\ar[d]^{p_3}\\
  Y_1\ar[r]^i & Y_2\ar[r]^{i'} &Y_3}\]
  where the vertical arrows are morphisms of $\cB$.
    Then the following diagram commutes
    \[\xymatrix{i^*i'^*p_{3*}\ar@=[r]^{B_{D'}}& i^*p_{2*}j'^*\ar@=[r]^{B_{D}}& p_{1*}j^*j'^*\\
    (i'i)^*p_{3*}\ar@=[rr]^{B_{D''}}\ar@=[u] & & p_{1*}(j'j)^*.\ar@=[u]}\]
\end{enumerate}
\end{Proposition}

\begin{proof}
This is equivalent to \cite[Propositions 1.1.11, 1.1.12]{Ayoub}. We provide
a proof for the sake of completeness.

\begin{itemize}
\item[\ref{p.bc1}] The following diagram commutes
  \[\xymatrix{i_3^*p'_*p_*\ar@=[r]_{\eta_{q'}} \ar@=@/^1.5pc/[rrr]^{B_{D'}}\ar@=@/_4pc/[rrrddd]_{B_{D''}}
  &q'_*q'{}^*i_3^* p'_* p_* \ar@=[r]^{\alpha'{}^*}\ar@=[d]_{\eta_q}
  & q'_*i_2^* p'{}^* p'_* p_* \ar@=[r]_{\epsilon_{p'}}\ar@=[d]^{\eta_q}
  & q'_*i_2^* p_* \ar@=[d]_{\eta_q}\ar@=@/^3pc/[ddd]^{B_D} \\
  & q'_*q_*q^*q'{}^* i_3^* p'_*p_* \ar@=[r]^{\alpha'{}^*}
  & q'_*q_*q^*i_2^*p'{}^*p'_*p_* \ar@=[r]^{\epsilon_{p'}}\ar@=[d]_{\alpha^*}
  & q'_* q_* q^* i_2^* p_*\ar@=[d]^{\alpha^*}\\
  && q'_*q_*i_1^* p^*p'{}^* p'_* p_* \ar@=[r]^{\epsilon_{p'}} & q'_*q_*i_1^* p^*p_* \ar@=[d]^{\epsilon_p}\\
  &&& q'_* q_* i_1^*}\]

\item[\ref{p.bc2}] Similar to \ref{p.bc1}.
\end{itemize}
\end{proof}

In the rest of this section, we further fix a locally full sub-$2$-category
$\cA$ of $\cC$ and a pseudo functor $A\colon \cA\to (\cD^\coop)^\adj$ such
that $P_1(A)$ is the restriction $F\res \cA$. We denote the pseudo functor
$L=P_2(A)\colon \cA\to (\cD^\coop)^\coop=\cD$ by
  \[i\mapsto i_!, \quad \alpha\mapsto \alpha_!.\]

\begin{Construction}[Base change map $A_D$ associated to an $(\cA,\cC)$-down-square
$D$]\label{s.bcA}\index{AD@$A_D$}
  Let $D$ be a square \eqref{e.-2} in $\cC$ where $i$ and $j$ are morphisms of $\cA$. The \emph{base change map} $A_D$ is by definition the following $2$-cell in $\cD$
  \[
  \xymatrix{j_! q^* \ar@=[r]^-{\eta_i} & j_! q^* i^* i_! \ar@=[r] & j_!(iq)^* i_! \ar@=[r]^{\alpha^*} & j_! (pj)^* i_! \ar@=[r] & j_! j^* p^* i_! \ar@=[r]^-{\epsilon_j} & p^* i_!. }
  \]
  If $p$ and $q$ are also morphisms of $\cA$, then $A_D$ is also the composition
  \[\xymatrix{j_!q^* \ar@=[r]^-{\eta_p} & p^*p_! j_! q^* \ar@=[r] & p^* (pj)_! q^* \ar@=[r]^{\alpha_!} & p^* (iq)_! q^*\ar@=[r] & p^*i_!q_!q^*\ar@=[r]^-{\epsilon_q} & p^* i_!.}\]
  We have an analogue of Proposition \ref{p.bc} for $A_D$.
\end{Construction}

\begin{Construction}[Construction of $G_D$]\label{s.GD}\index{GD@$G_D$}
  Let $D$ be a square \eqref{e.-2} in $\cC$ where the horizontal arrows $i$ and $j$ are morphisms of $\cA$, and the vertical arrows $p$ and $q$ are morphisms of $\cB$. Then
  \[(B_D,A_D)\colon (i^*p_*,p^*i_!,\eta_p\eta_i,\epsilon_p\epsilon_i)\Rightarrow (q_*j^*, j_!q^*,\eta_q\eta_j,\epsilon_q\epsilon_j)\]
  is a $2$-cell of $\cD^\adj$. In fact, the following diagrams commute
  
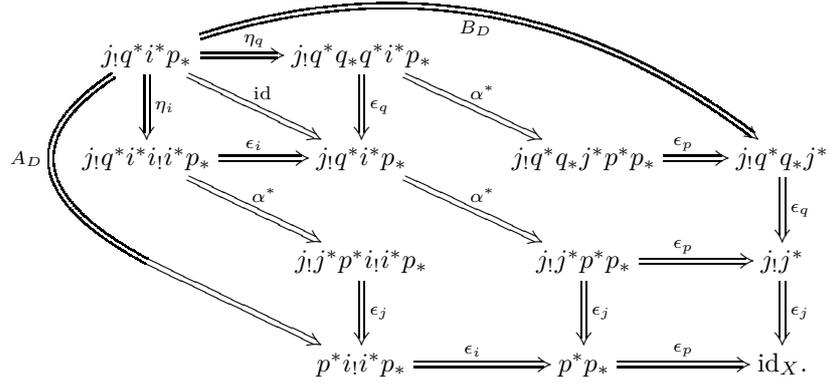
\[
  \xymatrix{\one_Z\ar@=[r]^{\eta_i}\ar@=[d]_{\eta_q} &i^* i_!\ar@=[r]^{\eta_p}\ar@=[d]^{\eta_q}
  & i^*p_*p^* i_!\ar@=[d]^{\eta_q}\ar@{=}[rd]\\
  q_*q^* \ar@=[r]^{\eta_i}\ar@=[d]_{\eta_j} & q_* q^* i^* i_!\ar@=[rd]^{\alpha^*} & q_* q^* i^* p_*p^* i_!\ar@=[rd]^{\alpha^*} & *{}\ar@=@/^3pc/[dd]^{B_D}\\
  q_*j^*j_!q^*\ar@=[r]^{\eta_i}\ar@=@/_3pc/[rrrd]^{A_D} & q_*j^*j_!q^*i^*i_!\ar@=[rd]^{\alpha^*}
  & q_*j^*p^* i_! \ar@=[r]^{\eta_p}\ar@=[d]^{\eta_j}\ar@=[rd]^{\one} & q_*j^*p^*p_*p^* i_!\ar@=[d]^{\epsilon_p}\\
  && q_*j^* j_!j^*p^* i_! \ar@=[r]^{\epsilon_j} & q_* j^* p^* i_!}
  \]
  \[
  \xymatrix{
  j_!q^*i^*p_*\ar@=@/^3pc/[rrrd]_{B_D}\ar@=[r]^{\eta_q}\ar@=[rd]^{\one}\ar@=[d]^{\eta_i}\ar@{=}@/_3pc/[dd]_{A_D} & j_!q^*q_*q^*i^*p_*\ar@=[d]^{\epsilon_q}\ar@=[rd]^{\alpha^*}\\
  j_!q^*i^*i_!i^*p_*\ar@=[r]^{\epsilon_i}\ar@=[rd]^{\alpha^*} & j_!q^*i^*p_*\ar@=[rd]^{\alpha^*} & j_!q^*q_*j^*p^*p_*\ar@=[r]^{\epsilon_p}
  & j_!q^*q_*j^*\ar@=[d]^{\epsilon_q}\\
  *{}\ar@=[rd] & j_!j^*p^*i_!i^*p_*\ar@=[d]^{\epsilon_j} & j_!j^*p^*p_*\ar@=[r]^{\epsilon_p}\ar@=[d]^{\epsilon_j} & j_!j^*\ar@=[d]^{\epsilon_j}\\
  & p^*i_!i^*p_*\ar@=[r]^{\epsilon_i} & p^*p_*\ar@=[r]^{\epsilon_p} & \one_X.}
  \]

  It follows that $B_D$ is invertible if and only if $A_D$ is. In this case, the diagram
  \[\xymatrix{i_! q_* \ar@=[r]^{\eta_j}\ar@=[d]_{\eta_p} & i_!q_*j^*j_!\ar@=[r]^{B_D^{-1}} & i_!i^*p_*j_!\ar@=[d]^{\epsilon_i} \\
  p_*p^*i_!q_*\ar@=[r]^{A_D^{-1}} & p_*j_!q^*q_*\ar@=[r]^{\epsilon_q} & p_*j_!}\]
commutes and we define $G_D\colon i_!q_*\Rightarrow p_* j_!$ to be the
composition. In fact, the following diagram commutes
  \[\xymatrix{& i_! q_*\ar@=[rr]^{\eta_j} && i_!q_*j^*j_!\ar@=[dd]^{B_D^{-1}}\\
  i_!q_*\ar@{}[rrd]|H\ar@=[ur]^{\one}\ar@=[r]^{\eta_q}\ar@=@/_2pc/[dd]_{\one}\ar@=[d]^{\eta_i} & i_!q_*q^*q_*\ar@=[r]^{\eta_j}\ar@=[u]^{\epsilon_q}
  & i_!q_*j^*j_!q^*q_* \ar@=[d]^{B_D^{-1}}\\
  i_!i^*i_!q_*\ar@=[r]^{\eta_p}\ar@=[d]^{\epsilon_i} & i_!i^*p_*p^*i_!q_*\ar@=[r]^{A_D^{-1}}
  & i_!i^*p_*j_!q^*q_*\ar@=[r]^{\epsilon_p}\ar@=[d]^{\epsilon_i} & i_!i^*p_*j_!\ar@=[d]^{\epsilon_i}\\
  i_!q_*\ar@=[r]^{\eta_p} & p_*p^*a_!q_*\ar@=[r]^{A_D^{-1}} & p_*j_!q^*q_* \ar@=[r]^{\epsilon_q} & p_*j_!,}\]
  where the hexagon $H$ commutes because the following diagram commutes
  \[\xymatrix{& q_*q^*\ar@=[r]^{\eta_j} & q_* j^*j_!q^* \ar@=[rd]_{B_D^{-1}} \ar@=[dl]^{A_D}\\
  \one_Z\ar@=[ru]^{\eta_q}\ar@=[rd]_{\eta_i} &q_*j^*p^*i_!\ar@=[rd]^{B_D^{-1}} && i^*p_*j_!q^* \ar@=[ld]^{A_D}\\
  & i^*i_! \ar@=[r]^{\eta_p} & i^*p_*p^*i_!.}\]
\end{Construction}

\begin{Construction}[Construction of $\rho_f$]\label{s.rho}\index{rho@$\rho$}
  Let $f\colon X\to Y$ be a morphism of $\cA\cap \cB$. Then
  \[(\epsilon_f^B,\eta^A_f)\colon (f^*f_*,f^*f_!,\eta_f^B\eta_f^A,\epsilon_f^B\epsilon_f^A)\Rightarrow (\one_X,\one_X,\one_{\one_X},\one_{\one_X})\]
  is a $2$-cell of $\cD^\adj$. It follows that $\epsilon_f^B$ is invertible if and only if $\eta_f^A$ is. In this case, the following diagram commutes
  \[\xymatrix{f_!\ar@=[r]^{\eta_f^B}\ar@=[d]_{(\epsilon_f^B)^{-1}} & f_*f^*f_!\ar@=[d]^{(\eta_f^A)^{-1}}\\
  f_!f^*f_*\ar@=[r]^{\epsilon_f^A} & f_*}\]
  and we define $\rho_f\colon f_!\Rightarrow f_*$ to be the composition. In fact, the following diagram commutes
  \[\xymatrix{&& f_*\ar@=[rd]^{\eta^A}\\
  & f_!f^*f_*\ar@=[r]^{\eta^A}\ar@=[d]^{\eta^B}\ar@=[ur]^{\epsilon^A}\ar@=[ld]_{\epsilon^B} & f_!f^*f_*f^*f_!\ar@=[r]^{\epsilon^A}\ar@=[d] & f_*f^*f_!\ar@=[d]^{\eta^B} \ar@=@/^3pc/[dd]^{\one}\\
  f_!\ar@=[rd]_{\eta^B} & f_*f^*f_!f^*f_*\ar@=[r]\ar@=[d]^{\epsilon^B} & f_*f^*f_!f^*f_*f^*f_! \ar@=[r]\ar@=[d]
  & f_*f^*f_*f^*f_!\ar@=[d]^{\epsilon^B}\\
  & f_*f^* f_! \ar@=[r]^{\eta^A} \ar@=@/_1.5pc/[rr]^{\one} & f_*f^*f_!f^*f_!\ar@=[r]^{\epsilon^A} & f_* f^*f_!.}\]

Let $X\xrightarrow{f}Y\xrightarrow{g} Z$ be a pair of composable morphisms
of $\cA\cap \cB$ with $\epsilon_f^B$ and $\epsilon_g^B$ invertible. Then the
following diagram commutes
  \[\xymatrix{g_!f_!\ar@=[r]_{\eta_f^B}\ar@=@/^1.5pc/[rr]^{\rho_f}\ar@=@/_3pc/[rrdd]_{\rho_{gf}}
  & g_!f_*f^*f_!\ar@=[r]_-{(\eta_f^A)^{-1}}\ar@=[d]_{\eta_g^B} & g_!f_*\ar@=[d]^{\eta_g^B}\ar@=@/^3pc/[dd]^{\rho_g}\\
  & g_*g^*g_!f_*f^*f_! \ar@=[r]^-{(\eta_f^A)^{-1}} & g_*g^*g_!f_*\ar@=[d]_{(\eta_g^A)^{-1}}\\
  && g_*f_*.}\]
\end{Construction}

The following properties of $G_D$ and $\rho_f$ are similar to axioms (b),
(\bprime) of Remark \ref{r.gd} and axioms (c), (\cprime) of Remark
\ref{r.cgd}.

\begin{Proposition}\label{p.ax}\leavevmode
\begin{enumerate}
\item \label{p.ax1} In the situation of Proposition \ref{p.bc}
    \ref{p.bc1}, assume that the horizontal arrows are morphisms of $\cA$,
    the vertical arrows are morphisms of $\cB$, and the $2$-cells $B_D$
    and $B_{D'}$ are invertible. Then the following diagram commutes
  \[\xymatrix{i_{3!}q'_* q_* \ar@=[r]^{G_{D'}}\ar@=[d] & p_*i_{2!}q_* \ar@=[r]^{G_{D}} & p'_*p_*i_{1!}\ar@=[d]\\
  i_{3!}(q'q)_* \ar@=[rr]^{G_{D''}} & & (p'p)_* i_{1!}.}
  \]

\item \label{p.ax2} In the situation of Proposition \ref{p.bc}
    \ref{p.bc2}, assume that the horizontal arrows are morphisms of $\cA$,
    the vertical arrows are morphisms of $\cB$, and the $2$-cells $B_D$
    and $B_{D'}$ are invertible. Then the following diagram commutes
  \[\xymatrix{i'_!i_!p_{1*}\ar@=[d]\ar@=[r]^{G_D} & i'_!p_{2*}j_! \ar@=[r]^{G_{D'}} & p_{3*}j'_!j_!\ar@=[d]\\
  (i'i)_!p_{1*}\ar@=[rr]^{G_{D''}} && p_{3*}(j'j)_!.}
  \]

\item \label{p.ax3} Let $D$ be a square \eqref{e.-2} in $\cC$ where the
    horizontal arrows $i$ and $j$ are morphisms of $\cA$, the vertical
    arrows $p$ and $q$ are morphisms of $\cA\cap \cB$, and the $2$-cells
    $\epsilon_p^B$, $\epsilon_q^B$ and $B_D$ are invertible. Then the
    following diagram commutes
  \[\xymatrix{i_!q_*\ar@=[d]_{G_D} & i_!q_!\ar@=[l]_{\rho_q} & (iq)_!\ar@=[l]\\
  p_*j_! & p_!j_!\ar@=[l]^{\rho_p} & (pj)_!.\ar@=[l]\ar@=[u]_{\alpha_!}}\]

\item \label{p.ax4} Let $D$ be a square \eqref{e.-2} in $\cC$ where the
    horizontal arrows $i$ and $j$ are morphisms of $\cA\cap \cB$, the
    vertical arrows $p$ and $q$ are morphisms of $\cB$, and the $2$-cells
    $\epsilon_i^B$, $\epsilon_j^B$ and $B_D$ are invertible. Then the
    following diagram commutes
  \[\xymatrix{i_!q_* \ar@=[r]^{\rho_i}\ar@=[d]_{G_D} & i_*q_* \ar@=[r] & (iq)_*\\
  p_*j_!\ar@=[r]^{\rho_j} & p_*j_* \ar@=[r] & (pj)_*.
  \ar@=[u]_{\alpha_*}}
  \]
\end{enumerate}
\end{Proposition}

\begin{proof}
\begin{itemize}
\item[\ref{p.ax1}] Similar to \ref{p.ax2}.

\item[\ref{p.ax2}] The following diagram commutes
  \[\xymatrix{i'_!i_!p_{1*}\ar@=[r]_{\eta_j}\ar@=@/^1.5pc/[rrr]^{G_D}\ar@=@/_3.5pc/[rrrddd]_{G_{D''}}
  &i'_!i_!p_{1*}j^*j_!\ar@=[r]^{B_{D}^{-1}}\ar@=[d]_{\eta_{j'}}
  & i'_!i_*i^*p_{2*}j_!\ar@=[r]_{\epsilon_i}\ar@=[d]^{\eta_{j'}}
  & i'_! p_{2*} j_!\ar@=[d]^{\eta_{j'}} \ar@=@/^3pc/[ddd]^{G_{D'}}\\
  & i'_!i_!p_{1*}j^*j'{}^{*}j'_!j_! \ar@=[r]^{B_D^{-1}}
  & i'_!i_!i^*p_{2*}j'{}^*j'_!j_! \ar@=[r]^{\epsilon_i} \ar@=[d]_{B_{D'}^{-1}}
  & i'_!p_{2*}j'{}^*j'_!j_!\ar@=[d]^{B_{D'}^{-1}}\\
  && i'_!i_!i^*i'{}^*p_{3*}j'_!j_*\ar@=[r]^{\epsilon_i}
  & i'_1i'{}^*p_{3*}j'_!j_! \ar@=[d]^{\epsilon_{i'}}\\
  &&& p_{3*}j'_!j_!.}\]

\item[\ref{p.ax3}] Similar to \ref{p.ax4}.

\item[\ref{p.ax4}] The following diagram commutes
  \[\xymatrix{i_!q_*\ar@=[r]_{\eta_i^B}\ar@=[d]^{\eta_j^A}\ar@=@/^1.5pc/[rr]^{\rho_i}\ar@=@/_3pc/[ddd]_{G_D}
  & i_*i^*i_!q_*\ar@=[r]_{(\eta_i^A)^{-1}}\ar@=[d] & i_*q_*\ar@=[d]_{\eta_j^A}^{\cong}\\
  i_!q_*j^*j_!\ar@=[r]\ar@=[d]^{B_D^{-1}} & i_*i^*i_!q_*j^*j_!\ar@=[r]\ar@=[d] & i_*q_*j^*j_!\ar@=[d]_{B_D^{-1}}^\cong\\
  i_!i^*p_*j_!\ar@=[r]^{\eta_i^B}\ar@=[d]^{\epsilon_i^A}
  & i_*i^*i_!i^*p_*j_! \ar@=[r]^{(\eta_f^A)^{-1}}\ar@=[d]^{\epsilon_i^A} & i_*i^*p_*j_!\\
  p_*j_!\ar@=[r]^{\eta_i^B}\ar@=@/_1.5pc/[rr]_(.7){\eta_j^B}\ar@=@/_3pc/[rrr]_{\rho_j}
  & i_*i^*p_*j_!\ar@=[ur]^{\one} & p_*j_*j^*j_!\ar@=@/_3pc/[uu]^{\alpha_*}\ar@=[r]^-{(\eta_j^A)^{-1}}
  & p_*j_*,\ar@=@/_1.5pc/[uuul]_{\alpha_*}}\]
  where the pentagon commutes because the following diagram commutes
  \[\xymatrix{p_*\ar@=[r]^{\eta_i^B}\ar@=[d]^{\eta_p}\ar@=@/_3pc/[dd]_{\one_{p_*}}
  & i_*i^*p_*\ar@=[r]^{\eta_q}\ar@{=}@/^2pc/[rr] &  i_* q_*q^*i^*p_*\ar@=[d]^{\alpha^*} & *{}\ar@=@/^2pc/[ddl]^{B_D}\\
  p_*p^*p_*\ar@=[r]^{\eta_j^B}\ar@=[d]^{\epsilon_p} & p_*j_*j^*p^*p_*\ar@=[r]^{\alpha^*}
  & i_*q_*j^*p^*p_*\ar@=[d]^{\epsilon_p}\\
  p_*\ar@=[r]^{\eta_j^B} & p_*j_*j^*\ar@=[r]^{\alpha_*} & i_*q_*b^*.}\]
\end{itemize}
\end{proof}

\begin{Remark}\label{r.gdadj}
If $\cC$ is a $(2,1)$-category, $B_D$ and $G_D$ are invertible for every
Cartesian $(\cA,\cB)$-square $D$, and $\epsilon_f^B$ and $\rho_f$ are
invertible for every morphism $f$ of $\cA\cap \cB$, then Proposition
\ref{p.ax} shows that $(L,R,(G_D),\rho)$ is an object of
$\GD^\Cart_{\cA,\cB}(\cC,\cD)$ (Definition \ref{d.gdp}). Here we have used
the correspondence via inverting $2$-cells between down-squares \eqref{e.-2}
used in this section and up-squares \eqref{e.2} used in earlier sections.
\end{Remark}

We conclude this section by a couple of criteria for the axioms for
morphisms of (Cartesian) gluing data to hold true.

The following property is similar to condition (n) of Definition
\ref{d.gdp}.

\begin{Proposition}\label{p.rhoAB}
Let $D$ be a square \eqref{e.-2} in $\cC$ where $i$ and $j$ are morphisms of
$\cA\cap \cB$ and such that the $2$-cells $\epsilon_i^B$, $\epsilon_j^B$,
and $\alpha$ are invertible, and let $D'$  be the square obtained by
inverting $\alpha$. Then the following diagram commutes
  \[\xymatrix{j_!q^*\ar@=[d]_{A_D}\ar@=[r]^{\rho_j} & j_*q^*\\
  p^* i_!\ar@=[r]^{\rho_i} & p^*i_*.\ar@=[u]_{B_{D'}}}\]
\end{Proposition}

\begin{proof}
  The following diagram commutes
  \[\xymatrix{j_!q^*\ar@=[r]_{\eta_j^B}\ar@=@/^1.5pc/[rr]^{\rho_j}\ar@=[d]^{\eta_i^A}\ar@=@/_3pc/[ddd]_{A_D}
  & j_!j^*j_*q^*\ar@=[r]_{(\eta_j^A)^{-1}} & j_*q^*\ar@=[d]^{\eta_i^A}\\
  j_!q^*i^*i_!\ar@=[d]^{\alpha^*} && j_*q^*i^*i_!\ar@=[d]^{\alpha^*}\ar@=@/^3pc/[ddd]_{\eta_i^B}\ar@=@/^1.5pc/[dddr]^{\one}\\
  j_!j^*p^*i_!\ar@=[r]^{\eta_j^B}\ar@=[d]^{\epsilon_j^A}
  & j_!j^*j_*j^*p^*i_!\ar@=[d]^{\epsilon_j^A}\ar@=[r]^{(\eta_j^A)^{-1}}
  & j_*j^*p^*i_!\ar@=[dl]^{\one}\\
  p^*i_!\ar@=[d]^{\eta_i^B}\ar@=@/_3pc/[dd]_{\rho_i} & j_*j^*p^*i_!\ar@=[d]^{\eta_i^B}\\
  p^*i_*i^*i_!\ar@=[d]^{(\eta_i^A)^{-1}}\ar@=[r]^{\eta_j^B}
  & j_*j^*p^*i_*i^*i_!\ar@=[r]^{(\alpha^{-1})^*}\ar@=[d]^{(\eta_i^A)^{-1}}
  & j_*q^*i^*i_*i^*i_!\ar@=[r]^{\epsilon_i^B}
  & j_*q^*i^*i_! \ar@=[d]^{(\eta_i^A)^{-1}}\\
  p^*i_*\ar@=[r]^{\eta_j^B}\ar@=@/_1.5pc/[rrr]^{B_{D'}}
  & j_*j^*p^*i_*\ar@=[r]^{(\alpha^{-1})^*} & j_* q^*i^*i_*\ar@=[r]^{\epsilon_i^B} & j_*q^*.}\]
\end{proof}

The following property is similar to condition (m) of Remark \ref{r.gd}.

\begin{Proposition}\label{p.GAB}
Let
     \begin{equation}
    \xymatrix{X'\ar[dd]_{x}\ar[dr]^{q'}\ar[rr]^{j'} && Y'\ar[dd]^(.3){y}|\hole \ar[dr]^{p'}\\
    & Z'\ar[dd]^(.3){z}\ar[rr]^(.3){i'} && W'\ar[dd]^w \\
    X\ar[dr]_q\ar[rr]^(.3){j}|\hole && Y\ar[dr]^p\\
    & Z\ar[rr]^i && W }
     \end{equation}
be a cube in $\cC$, where the horizontal arrows $i,j,i',j'$ are morphisms of
$\cA$, the oblique arrows $p,q,p',q'$ are morphisms of $\cB$, and the
$2$-cells of the right, left, front, back, bottom, top faces,
$I,I',J,J',K,K'$, are respectively
\[py\Rightarrow wp', \quad qx \Rightarrow zq', \quad \beta\colon wi'\Rightarrow iz, \quad \beta' \colon yj'\Rightarrow jx,\quad pj\Rightarrow iq, \quad p'j'\Rightarrow i'q'.\]
Assume that $B_K$ and $B_{K'}$ are invertible. Then the following
diagram commutes
    \[\xymatrix{i'_!z^*q_*\ar@=[r]^{B_{I'}}\ar@=[d]_{A_J} & i'_!q'_*x^*\ar@=[r]^{G_{K'}} & p'_*j'_!x^*\ar@=[d]^{A_{J'}}\\
    w^*i_!q_*\ar@=[r]^{G_K} & w^* p_* j_!\ar@=[r]^{B_I} & p'_*y^* j_!.}\]
\end{Proposition}

\begin{proof}
  The following diagram commutes
  \[\xymatrix{&i'_!z^*q_*\ar@=[r]_{\eta_i}\ar@=[ddl]_{B_{I'}}\ar@=@/^1.5pc/[rrr]^{A_J}\ar@=[d]^{\eta_j}
  & i'_! z^* i^*i_!q_*\ar@=[r]^{\beta^*}
  & i'_!i'{}^* w^*i_! q_*\ar@=[r]_{\epsilon_{i'}} \ar@=[d]_{\eta_j}
  & w^*i_!q_*\ar@=[d]^{\eta_j}\ar@=@/^3pc/[ddd]^{G_K}\\
  & i'_!q^*q_*j^*j_!\ar@=[d]^{B_{I'}}\ar@=[r]^{\eta_i}
  & i'_!z^*i^*i_!q_*j^*j_!\ar@=[r]^{\beta^*}
  & i'_!i'{}^*w^*i_!q_*j^*j_!\ar@=[r]^{\epsilon_{i'}}\ar@=[d]_{B_{K}^{-1}}
  & w^*i_!q_*j^*j_!\ar@=[d]^{B_K^{-1}}\\
  i'_!q'_*x^*\ar@=[r]^{\eta_j}\ar@=[d]^{\eta_{j'}}\ar@=@/_3pc/[ddd]^{G_{K'}} & i'_!q'_*x^*j^*j_!\ar@=[d]^{\eta_{j'}}
  && i'_!i'{}^*w^*i_!i^*p_*j_!\ar@=[d]_{\epsilon_i} & w^*i_!i^*p_*j_!\ar@=[d]^{\epsilon_i}\\
  i'_!q'_*j'{}^*j'_!x^*\ar@=[d]^{B_{K'}^{-1}} & i'_!q'_*j'{}^*j'_!x^*j^*j_!\ar@=[d]^{B_{K'}^{-1}}
  && i'_!i'{}^*w^*p_*j_!\ar@=[r]^{\epsilon_{i'}}\ar@=[d]_{B_I} & w^*p_*j_!\ar@=[ddl]^{B_I}\\
  i'_!i'{}^*p'_*j'_!x^*\ar@=[r]^{\eta_j}\ar@=[d]^{\epsilon_{i'}}
  & i'_!i'{}^*p'_*j'_!x^*j^*j_!\ar@=[r]^{\beta'{}^*}\ar@=[d]^{\epsilon_{i'}}
  & i'_!i'{}^*p'_*j'_!j'{}^*y^*j_!\ar@=[r]^{\epsilon_{j'}}
  & i'_!i'{}^*p'_*y^*j_!\ar@=[d]_{\epsilon_{i'}}\\
  p'_*j'_!x^*\ar@=[r]^{\eta_j}\ar@=@/_1.5pc/[rrr]^{A_{J'}}
  & p'_*j'_!x^*j^*j_!\ar@=[r]^{\beta'{}^*} & p'_*j'_!j'{}^*y^*j_!\ar@=[r]^{\epsilon_{j'}}
  & p'_*y^*j_!,}\]
  where the decagon is the outline of the following commutative diagram
  \[\xymatrix{&&q'_*x^*j^*\ar@=[r]^{\eta_{j'}}\ar@=[d]^{\beta'{}^*} & q'_*j'{}^*j'_!x^*j^*\ar@=[r]^{B_{K'}^{-1}}\ar@=[d]^{\beta'{}^*}
  & i'{}^*p'_*j'_!x^*j^*\ar@=[d]^{\beta'{}^*}\\
  z^*q_*j^*\ar@=[r]_{B_K^{-1}}\ar@=[d]_{\eta_i}\ar@=[urr]^{B_{I'}}
  & z^*i^*p_*\ar@=[d]^{\eta_i}\ar@=[rd]^{\one}
  & q'_*j'{}^*y^*\ar@=[rd]^{\one} \ar@=[r]^{\eta_{j'}}
  & q'_*j'{}^*j'_!j'{}^*y^*\ar@=[d]^{\epsilon_{j'}}\ar@=[r]^{B_{K'}^{-1}}
  & i'{}^*p'_*j'_*j'{}^*y^*\ar@=[d]^{\epsilon_{j'}}\\
  z^*i^*i_!q_*j^*\ar@=[r]^{B_K^{-1}}\ar@=[d]_{\beta^*}
  & z^*i^*i_!i^*p_*\ar@=[d]^{\beta^*}\ar@=[r]^{\epsilon_i}
  & z^*i^*p_*\ar@=[d]^{\beta^*}
  & q'_*j'{}^*y^*\ar@=[r]^{B_{K'}^{-1}}
  & i'{}^*p'_*y^*j_!\\
  i'{}^*w^*i_!q_*j^*\ar@=[r]^{B_K^{-1}}
  & i'{}^*w^*i_!i^*p_*\ar@=[r]^{\epsilon_i}
  & i'{}^*w^*p_*,\ar@=[urr]^{B_I}}\]
  where the octagon commutes by Proposition \ref{p.bc}.
\end{proof}

This ends our discussion on the construction of gluing data. The next
section is included for completeness.

\section{Proof of Proposition \ref{p.equiv}}\label{s.proof}
\begin{proof}[Proof of Proposition \ref{p.equiv} \ref{p.equiv1}]
Let $G$ and $H$ be pseudo functors from $\cD$ to $\cE$. We need to show that
the functor
\[\PsNat(G,H)\to \PsNat(GF,HF)\quad \alpha\mapsto \alpha F\]
between categories of pseudo natural transformations and modifications is
fully faithful and injective on objects. Let $\alpha$ and $\beta$ be pseudo
natural transformations from $G$ to $H$.

We identify the set of modifications $\alpha\Rightarrow \beta$ and the set
of modifications $\alpha F\Rightarrow \beta F$ with subsets of the set of
$2$-cells $\lvert \alpha \rvert \Rightarrow \lvert \beta \rvert$ in
$\cE^{\Ob(\cC)}$. Let $\Xi\colon \alpha F\Rightarrow \beta F$ be a
modification. Let $g\colon X\to Y$ be a morphism of $\cD$. By assumption,
there exists a decomposition $ g\xRightarrow{\delta} h\xRightarrow{\gamma}
g$ in $\cD$, where $h=Ff$ for some morphism $f$ of $\cC$. All inner cells of
the diagram
\[\xymatrix{\alpha(Y)G(g)\ar@=[d]_{\alpha(g)}\ar@=[r]_{G(\delta)}\ar@=@/^1pc/[rrr]_{\Xi Y} &\alpha(Y)G(h)\ar@=[d]^{\alpha(h)}\ar@=[r]_{\Xi Y}& \beta(Y)G(h)\ar@=[d]^{\beta(h)}\ar@=[r]_{G(\gamma)}& \beta(Y)G(g)\ar@=[d]^{\beta(g)}\\
H(g)\alpha(X)\ar@=[r]^{H(\delta)}\ar@=@/_1pc/[rrr]^{\Xi X} & H(h)\alpha(X)\ar@=[r]^{\Xi
X}&H(h)\beta(X)\ar@=[r]^{H(\gamma)} & H(g)\beta(X)}
\]
commutes. Therefore, $\Xi$ is a modification $\alpha\Rightarrow \beta$.

Assume $\alpha F=\beta F$. Let $g$ and $\delta$, $\gamma$ be as above. The
square
\[\xymatrix{\alpha(Y)G(g)\ar@=[d]_{\alpha(g)}\ar@=[r]^{G(\delta)} &\alpha(Y)G(h)\ar@=[d]^{\alpha(h)}\ar@=[r]^{G(\gamma)}& \alpha(Y)G(g)\ar@=[d]^{\alpha(g)}\\
H(g)\alpha(X)\ar@=[r]^{H(\delta)} & H(h)\alpha(X)\ar@=[r]^{H(\gamma)} & H(g)\alpha(X)}
\]
commutes and the composition of each of the two lines is an identity. The
same holds for $\beta$. Since $\alpha (h)=(\alpha F)(f)=(\beta F)(f)=\beta
(h)$, we have
\[\alpha(g)=H(\gamma)\alpha(h)G(\delta)=H(\gamma)\beta(h)G(\delta)=\beta(g).\]
Therefore, $\alpha=\beta$.
\end{proof}

\begin{proof}[Proof of Proposition \ref{p.equiv} \ref{p.equiv2}]
It is clear that \ref{p.equivc} implies \ref{p.equivd}.

Next we show that \ref{p.equivd} implies \ref{p.equiva}. Since $\Phi_\cC$ is
pseudo surjective, there exists a pseudo functor $G\colon \cD\to \cC$ such
that $\Phi_\cC(G)=GF$ is equivalent to $\id_\cC$. Then $\Phi_\cD(FG)=FGF$ is
equivalent to $\Phi_\cD(\id_{\cD})=F$. Since $\Phi_\cD$ is a local
equivalence, it follows that $FG$ is equivalent to $\id_\cD$. Therefore, $F$
is a bi-equivalence.

To prove that \ref{p.equiva} implies \ref{p.equivb}, let $G\colon \cD\to
\cC$ be a pseudo functor endowed with pseudo natural equivalences
$\eta\colon \one_\cC\to GF$ and $FG\to \one_\cD$. For every object $Y$ of
$\cD$, let $G'Y=\Ob( F)^{-1} Y$. For every morphism $g\colon Y\to Y'$ of
$\cD$, choose a morphism $G'g\colon G'Y\to G'Y'$ and an invertible $2$-cell
$\psi_g$ in $\cC$:
\[\xymatrix{G'Y\ar[r]^{G'g}\ar[d]_{\eta(G'Y)}\drtwocell\omit{^{\psi_g}} & G'Y'\ar[d]^{\eta(G'Y')}\\
GY\ar[r]^{Gg}& GY'.}
\]
By Lemma \ref{l.ps}, this determines a pseudo functor $G'\colon \cD\to \cC$
such that $\Ob(G') \colon\Ob(\cD)\to \Ob(\cC)$ is the inverse of $\Ob(F)$
and a pseudo natural equivalence $\psi\colon G'\to G$ such that
$\psi(Y)=\eta(G'Y)$ for every object $Y$ of $\cD$. For any morphism $f\colon
X\to X'$ of $\cC$, $\eta$ induces the following square in $\cC$:
\[\xymatrix{X\ar[r]^{f}\ar[d]_{\eta(X)}\drtwocell\omit{^{\eta_f}} & X'\ar[d]^{\eta(X')}\\
GFX\ar[r]^{GFf}& GFX'.}
\]
Since $\eta (X')$ is an equivalence in $\cC$, there exists an invertible
$2$-cell $\eta'_f\colon f\Rightarrow G'Ff$ such that
$\eta_f^{-1}\eta'_f=\psi_{Ff}^{-1}$. This defines a pseudo natural
isomorphism $\eta'\colon \one_\cC\to G'F$ satisfying $\eta'(X)=\one_X$ for
every object $X$ of~$\cC$. Since $F$ is locally essentially surjective, for
every morphism $g\colon Y\to Y'$ of $\cD$, there exist a morphism $f\colon
G'Y\to G'Y'$ of $\cC$ and an invertible $2$-cell $\alpha\colon g\Rightarrow
Ff$. The composition
\[\epsilon_g\colon FG'g \xRightarrow{FG'\alpha} FG'Ff \xRightarrow{(\eta'_{Ff})^{-1}}
Ff\xRightarrow{\alpha^{-1}} g
\]
does not depend on the choice of $(f,\alpha)$. This defines a pseudo natural
isomorphism $\epsilon\colon FG'\to \one_\cD$ such that $\epsilon(Y)=\one_Y$
for every object $Y$ of $\cD$.

It remains to show that \ref{p.equivb} implies \ref{p.equivc}. By
Proposition \ref{p.UPs}, it suffices to show that the restriction of
$\Phi_\cE$ to $\UPsFun(\cD,\cE)$ is a $\cE^{\Ob(\cC)}$-$2$-equivalence. For
this, we may assume that $F$ is strictly unital. Then $G$ is strictly
unital. Consider the $\cE^{\Ob(\cC)}$-$2$-functors
\[\Phi'_\cE\colon\UPsFun(\cD,\cE)\to \UPsFun(\cC,\cE), \quad \Psi'_\cE\colon\UPsFun(\cC,\cE)\to \UPsFun(\cD,\cE)\]
induced by $F$ and $G$, respectively. Then $\eta$ and
    $\epsilon$ induce $\cE^{\Ob(\cC)}$-$2$-natural isomorphisms $\one \to
    \Phi'_\cE\Psi'_\cE$ and $\Psi'_\cE\Phi'_\cE\to \one$.
    Thus \ref{p.equivb} implies \ref{p.equivc}.
\end{proof}

\bibliographystyle{abbrv}
{\small\bibliography{glue}}

\printindex

\end{document}